\documentclass[amstex,12pt]{article}
\usepackage{}
\usepackage{amsmath}
\usepackage{amssymb,amsmath,amsthm}
\newtheorem{theorem}{Theorem}[section]

\newtheorem{lemma}{Lemma}[section]
\newtheorem{definition}{Definition}[section]
\newtheorem{remark}{Remark}[section]

\newtheorem{example}{Example}[section]

\newcommand{\beq}{\begin{equation}}
\newcommand{\eeq}{\end{equation}}
\newcommand{\beqn}{\begin{eqnarray}}
\newcommand{\eeqn}{\end{eqnarray}}

\allowdisplaybreaks

\begin{document}
\allowdisplaybreaks

\title{Asymptotical stability of almost periodic solution for an impulsive multispecies competition-predation system with time delays on time scales\thanks{This work is
supported by the National Natural Sciences Foundation of People's
Republic of China under Grant 11361072.}}
\author {Yongkun Li\thanks{%
The corresponding author.} and Pan Wang\\
Department of Mathematics, Yunnan University\\
Kunming, Yunnan 650091\\
 People's Republic of China}
\date{}
\maketitle{}
\begin{abstract}
In this paper, we consider the almost periodic dynamics of an impulsive multispecies Lotka-Volterra competition system with time delays on time scales.
By establishing some comparison theorems of dynamic equations with impulses and delays on time scales, a permanence result for the model is obtained. Furthermore,
based on the permanence result, by studying the Lyapunov stability theory of impulsive dynamic equations on time scales, we establish a criterion for the existence and uniformly asymptotic stability of a unique positive almost periodic solution of the system. Finally, we give an example to show the feasibility of our main results and our example also shows that the continuous time system and its corresponding discrete time system have the same dynamics. Our results of this paper are completely new even if for both  the case of the time scale $\mathbb{T}=\mathbb{R}$ and the case of $\mathbb{Z}$.
\end{abstract}

\section{Introduction}

\setcounter{equation}{0}
{\setlength\arraycolsep{2pt}}
 \indent

In recent years, applications of the theory of differential equations in mathematical ecology have developed rapidly. Various mathematical models have been proposed in the study of population dynamics. One of the famous models for population dynamics is the Lotka-Volterra competition system. The traditional Lotka-Volterra competition system can be expressed as follows:
\begin{equation}\label{e11}
\dot{x}_i(t)=x_i(t)\bigg[r_i(t)-\sum\limits_{j=1}^na_{ij}(t)x_j(t)\bigg], i=1,2,\ldots,n,
\end{equation}
where $r_i(t)$ represents the intrinsic growth rate of species $i$ at time $t$ and $a_{ij}(t)$ the competing coefficients
between species $j$ and $i$.
Many excellent results which are concerned with permanence, extinction and global attractivity of periodic solutions or almost periodic solutions of system (1.1) are obtained (see \cite{1,2,3,4,5}). Moreover, in the real world, models with delays are much more realistic, as in reality time delays occur in almost every biological situation and assumed to be one of the reasons of regular fluctuations in population density. Many important and interesting population dynamical systems with delays have been extensively studied (see \cite{6,7,8,9,10,k1}). As is known that one of the most interesting questions in mathematical biology concerns the survival of species in ecological models. Biologically, when a system of interacting species is persistent in a suitable sense, it means that all the species survive in the long term. Since permanence is one of the most important topics on the study of population dynamics, it is reasonable to ask for conditions under which the system is permanent. For example, in \cite{7}, Chen discussed the permanence and global
stability of the non-autonomous Lotka-Volterra system with
predator-prey and delays by using a comparison theorem
and constructing a suitable Lyapunov function:
{\setlength\arraycolsep{2pt}
\begin{eqnarray}\label{e12}
\left\{
\begin{array}{ll}
\dot{x_i}(t)=x_i(t)\bigg[b_i(t)-\sum\limits_{k=1}^n
a_{ik}(t)x_k(t-\tau_{ik}(t))-\sum\limits_{k=1}^m
c_{ik}(t)y_k(t-\sigma_{ik}(t))\bigg],\\
\dot{y_j}(t)=y_j(t)\bigg[-r_j(t)+\sum\limits_{k=1}^n
d_{jk}(t)x_k(t-\xi_{jk}(t))-\sum\limits_{k=1}^m
e_{jk}(t)y_k(t-\eta_{jk}(t))\bigg],\\
\quad\quad\quad \,\,i=1,2,\ldots,n, j=1,2,\ldots,m.
\end{array}
\right.
\end{eqnarray}}

On the other hand, many natural and man-made factors (e.g., fire, drought, flooding deforestation, hunting, harvesting,
breeding etc.) always lead to rapid decrease or increase of population numbers at fixed times. Such sudden changes can often
be characterized mathematically in the form of impulses. With the development of the theory of impulsive differential equations \cite{11,12}, nonautonomous $n$-species Lotka-Volterrra competitive systems with impulsive effects have been studied by many authors and many important and significant results are obtained (see \cite{13,14,15,16,p1,chen1}).
 However,  at present, since  few comparison theorems of solutions of differential or difference equations with both impulses and delays are available,
very few works  have been done for the permanence of population models with both impulses and delays, especially, for the models with impulses and discrete delays.
Besides, it is well known that both continuous time systems and discrete time systems are equally important in theory and applications. But, up to now, there are few papers published on the permanence of discrete time population models with impulses and delays. Also, it is well known that the study of dynamic equations on time scales has been created in order to unify the study
of differential and difference equations.

Motivated by the above reasons, in this paper, we are concerned with the following impulsive multispecies competition-predation system with time delays on time scales:
{\setlength\arraycolsep{2pt}
\begin{eqnarray}\label{e13}
\left\{
\begin{array}{ll}
x_i^\Delta(t)=b_i(t)-\sum\limits_{l=1}^na_{il}(t)\exp\{x_l(t-\tau_{il}(t))\}-\sum\limits_{h=1}^mc_{ih}(t)\exp\{y_h(t-\delta_{ih}(t))\},\,\,t\neq t_k,\,\,t\in J, \\
y_j^\Delta(t)=-r_j(t)+\sum\limits_{l=1}^nd_{jl}(t)\exp\{x_l(t-\xi_{jl}(t))\}\\
\quad\quad\quad\,\,\,\,-\sum\limits_{h=1}^me_{jh}(t)\exp\{y_h(t-\eta_{jh}(t))\},\,\,t\neq t_k,\,\,t\in J, \\
x_i(t_{k}^{+})=x_i(t_{k})+\ln(1+\lambda_{ik}),\,\,t=t_k,\\
y_j(t_{k}^{+})=y_j(t_{k})+\ln(1+\lambda_{jk}),\,\,t=t_k,\,\,i=1,2,\ldots,n,\,\,j=1,2,\ldots,m,
\end{array}
\right.
\end{eqnarray}}
where $J=[t_0, +\infty)_{\mathbb{T}}$ and $0\leq t_0\in\mathbb{T}$, $x_i(t)$ is the density of species $x_i$ at time $t$, $y_j(t)$ is the density of species $y_j$ at time $t$ and $\{t_k\}\subset\mathcal{B}:=\{\{t_k\}$, $t_k\in\mathbb{T}$: $t_k<t_{k+1}$, $k\in\mathbb{N}$, $\lim_{t\rightarrow\infty}t_{k}=\infty\}$.

\begin{remark}
Let $z_i(t)=\exp\{x_i(t)\}$, $w_j(t)=\exp\{y_j(t)\}$ for $i=1,2,\ldots,n$, $j=1,2,\ldots,m$. If $\mathbb{T}=\mathbb{R}$, then
$\eqref{e12}$ can be written as {\setlength\arraycolsep{2pt}
\begin{eqnarray}\label{e14}
\left\{
\begin{array}{ll}
\dot{z_i}(t)=z_i(t)\bigg[b_i(t)-\sum\limits_{l=1}^na_{il}(t)z_l(t-\tau_{il}(t))-\sum\limits_{h=1}^mc_{ih}(t)w_h(t-\delta_{ih}(t))\bigg],\,\,t\neq t_k,\\
\dot{w_j}(t)=w_j(t)\bigg[-r_j(t)+\sum\limits_{l=1}^nd_{jl}(t)z_l(t-\xi_{jl}(t))-\sum\limits_{h=1}^me_{jh}(t)w_h(t-\eta_{jh}(t))\bigg],\,\,t\neq t_k,\\
z_i(t_{k}^{+})=(1+\lambda_{ik})z_i(t_{k}),\,\,t=t_k,\\
w_j(t_{k}^{+})=(1+\lambda_{jk})w_j(t_{k}),\,\,t=t_k,\,\,i=1,2,\ldots,n,\,\,j=1,2,\ldots,m,
\end{array}
\right.
\end{eqnarray}}
where $t\in[t_0, +\infty)$.
If  $\mathbb{T}=\mathbb{Z}$, then $\eqref{e12}$ becomes
\begin{eqnarray}\label{e15}
\left\{
\begin{array}{ll}
z_i(t+1)=z_i(t)\exp\bigg\{b_i(t)-\sum\limits_{l=1}^na_{il}(t)z_l(t-\tau_{il}(t))-\sum\limits_{h=1}^mc_{ih}(t)w_h(t-\delta_{ih}(t))\bigg\},\,\,t\neq t_k,\\
w_j(t+1)=w_j(t)\exp\bigg\{-r_j(t)+\sum\limits_{l=1}^nd_{jl}(t)z_l(t-\xi_{jl}(t))-\sum\limits_{h=1}^me_{jh}(t)w_h(t-\eta_{jh}(t))\bigg\},\,\,t\neq t_k,\\
z_i(t_{k}^{+})=(1+\lambda_{ik})z_i(t_{k}),\,\,t=t_k,\\
w_j(t_{k}^{+})=(1+\lambda_{jk})w_j(t_{k}),\,\,t=t_k,\,\,i=1,2,\ldots,n,\,\,j=1,2,\ldots,m,
\end{array}
\right.
\end{eqnarray}
where $t\in[t_0, +\infty)_{\mathbb{Z}}$.
\end{remark}

From the point of view of biology, we focus our discussion on the positive solutions of system $\eqref{e13}$. So it is assumed that the initial conditions of system $\eqref{e13}$ are of the form
{\setlength\arraycolsep{2pt}
\begin{eqnarray*}
 \left\{%
\begin{array}{lcrcl}
x_{i}(\theta; t_{0}, \phi_{0}^{i})=\phi_{0}^{i}(\theta)\geq 0,\,\, \phi_{0}^{i}(t_{0})>0,\,\,\theta\in[t_{0}-\hat{\tau},t_{0}]_{\mathbb{T}},\,\,i=1,2,\ldots,n,\\
y_{j}(\theta; t_{0}, \psi_{0}^{j})=\psi_{0}^{j}(\theta)\geq 0,\,\,\psi_{0}^{j}(t_{0})>0,\,\,\theta\in[t_{0}-\hat{\eta}, t_{0}]_{\mathbb{T}},\,\,j=1,2,\ldots,m,\\
\end{array}\right.
\end{eqnarray*}}
where $\phi_{0}^{i},\psi_{0}^{j}\in C([t_{0}-\hat{\tau},t_{0}]_{\mathbb{T}},[0,+\infty))$,
\[\hat{\tau}=\max\{\tau^{+}, \delta^{+}\},\,\hat{\eta}=\max\{\xi^{+}, \eta^{+}\},
\tau^{+}=\max\limits_{1\leq i,l\leq n}\sup\limits_{t\in\mathbb{T}}\{\tau_{il}(t)\},\, \tau^{-}=\min\limits_{1\leq i,l\leq n}\inf\limits_{t\in\mathbb{T}}\{\tau_{il}(t)\},\]
\[\delta^{+}=\max\limits_{1\leq i\leq n,1\leq h\leq m}\sup\limits_{t\in\mathbb{T}}\{\delta_{ih}(t)\},\,\delta^{-}=\min\limits_{1\leq i\leq n,1\leq h\leq m}\inf\limits_{t\in\mathbb{T}}\{\delta_{ih}(t)\},\,
\xi^{+}=\max\limits_{1\leq j\leq m,1\leq l\leq n}\sup\limits_{t\in\mathbb{T}}\{\xi_{jl}(t)\},\]
\[ \xi^{-}=\min\limits_{1\leq j\leq m,1\leq l\leq n}\inf\limits_{t\in\mathbb{T}}\{\xi_{jl}(t)\},\,
\eta^{+}=\max\limits_{1\leq j,h\leq m}\sup\limits_{t\in\mathbb{T}}\{\eta_{jh}(t)\},\,\eta^{-}=\min\limits_{1\leq j,h\leq m}\inf\limits_{t\in\mathbb{T}}\{\eta_{jh}(t)\}.\]

For convenience, we denote
$f^L=\inf\limits_{t\in{\mathbb{T}}}|f(t)|, f^U=\sup\limits_{t\in{\mathbb{T}}}|f(t)|,$
where $f$ is an almost periodic function on $\mathbb{T}$ and $\bar{\mu}=\sup\limits_{\theta\in \mathbb{T}}\{\mu(\theta)\}$, where $\mu(t)$ is the forword graininess  of $\mathbb{T}$.

Throughout this paper, we assume that
\begin{itemize}
\item[$(H_{1})$] $b_i(t)$, $a_{il}(t)$, $c_{ih}(t)$, $\tau_{il}(t)$, $\delta_{ih}(t)$, $r_j(t)$, $d_{jl}(t)$, $e_{jh}(t)$, $\xi_{jl}(t)$, $\eta_{jh}(t)$ are nonnegative almost periodic functions for $t\in{\mathbb{T}}$, $i,l=1,2,\ldots,n$, $j, h=1,2,\ldots,m$;
 \item [$(H_{2})$]
  $\{\lambda_{ik}\}$ and $\{\lambda_{jk}\}$ are almost periodic sequences,
 $0<r\leq\max\bigg\{\prod\limits_{t_{0}<t_{k}<t}(1+\lambda_{ik}), \prod\limits_{t_{0}<t_{k}<t}(1+\lambda_{jk})\bigg\}\leq 1$
   for $t\geq t_{0}$ and $-1< \{\lambda_{ik},\lambda_{jk}\}\leq0$ for $k\in\mathbb{N}$, $i=1,2,\ldots,n$, $j=1,2,\ldots,m$;
  \item [$(H_{3})$]
  the set of sequences $\{t_{k}\}\in UAPS$, where $UAPS=\{\{t_{k}^{j}\}$, $t_{k}^{j}=t_{k+j}-t_{k}$, $k, j\in\mathbb{N}$ is uniformly almost periodic and $\inf_{k}t_{k}^{1}=\theta>0\}\subset\mathcal{B}$.
\end{itemize}

The main purpose of this paper is to discuss the permanence of system \eqref{e13} by establishing  some comparison theorems of dynamic equations with impulses and delays on time scales
and based on the obtained permanence result, by studying the Lyapunov stability theory of impulsive dynamic equations on time scales, we establish the existence and uniformly asymptotic stability of a unique positive almost periodic solution of system \eqref{e13}.

\begin{remark}
To our knowledge, there have no studies been reported on the
the permanence and almost periodicity of system \eqref{e14} and system \eqref{e15}  until now.
\end{remark}
The organization of this paper is as follows: In Section 2, we introduce some notations and definitions, state some preliminary results. In Section 3, we   establish some comparison theorems of dynamic equations with impulses and delays which are needed
in later sections. In Section 4, we obtain some sufficient conditions for the permanence of \eqref{e13} by using the comparison theorems obtained in Section 3. In Section 5, by studying the Lyapunov stability theory of impulsive dynamic equations on time scales we establish some sufficient conditions for the existence and uniformly asymptotic stability of unique positive almost periodic solution of \eqref{e13}. In Section 6, we give an example to illustrate the feasibility and effectiveness of our results obtained in previous sections. Finally, we draw a conclusion in Section 7.

\section{Preliminaries}

\setcounter{equation}{0}
{\setlength\arraycolsep{2pt}}
 \indent

In this section, we shall introduce some basic definitions, lemmas which are used in what follows.

A time scale $\mathbb{T}$ is an arbitrary nonempty closed subset of the real numbers, the forward and backward jump operators $\sigma$, $\rho:\mathbb{T}\rightarrow \mathbb{T}$ and the forward graininess $\mu:\mathbb{T}\rightarrow \mathbb{R}^{+}$ are defined, respectively, by
\[
\sigma(t):=\inf \{s\in\mathbb{T}:s> t\},\,\,\rho(t):=\sup\{s\in\mathbb{T}:s<t\}\,\,
\text{and}\,\,\mu(t)=\sigma(t)-t.
\]

A point $t$ is said to be left-dense if $t>\inf\mathbb{T}$ and $\rho(t)=t$, right-dense if $t<\sup\mathbb{T}$ and $\sigma(t)=t$, left-scattered if $\rho(t)<t$ and right-scattered if $\sigma(t)>t$. If $\mathbb{T}$ has a left-scattered maximum $m$, then $\mathbb{T}^{k}=\mathbb{T}\backslash m$, otherwise $\mathbb{T}^{k}=\mathbb{T}$. If $\mathbb{T}$ has a right-scattered minimum $m$, then $\mathbb{T}_{k}=\mathbb{T}\backslash m$, otherwise $\mathbb{T}^{k}=\mathbb{T}$.

A function $f : \mathbb{T}\rightarrow \mathbb{R}$ is right-dense continuous or rd-continuous provided it is continuous at right-dense points in $\mathbb{T}$ and its left-sided limits exist (finite) at left-dense points in $\mathbb{T}$. If $f$ is continuous at each right-dense point and each left-dense point, then $f$ is said to be a continuous function on $\mathbb{T}$.

For $f:\mathbb{T}\rightarrow\mathbb{R}$ and $t\in{\mathbb{T}^{k}}$, then $f$ is called delta differentiable at $t\in{\mathbb{T}}$ if there exists $c\in\mathbb{R}$ such that for given any $\varepsilon\geq{0}$, there is an open neighborhood $U$ of  $t$ satisfying
\[
\left|[f(\sigma(t))-f(s)]-c[\sigma(t)-s]\right|\leq\varepsilon\left|\sigma(t)-s\right|
\]
for all $s\in U$. In this case, $c$ is called the delta derivative of $f$ at $t\in{\mathbb{T}}$, and is denoted by $c=f^{\Delta}(t)$. For $\mathbb{T}=\mathbb{R}$, we have $f^{\Delta}=f^{'}$, the usual derivative, and for $\mathbb{T}=\mathbb{Z}$ we have the backward difference operator, $f^{\Delta}(t)=\Delta f(t):=f(t+1)-f(t)$.

A function $p:\mathbb{T}\rightarrow\mathbb{R}$ is called regressive provided $1+\mu(t)p(t)\neq 0$ for all $t\in{\mathbb{T}^{k}}$. The
set of all regressive and rd-continuous functions $p:\mathbb{T}\rightarrow\mathbb{R}$ will be denoted by $\mathcal{R}=\mathcal{R}(\mathbb{T})=\mathcal{R}(\mathbb{T}, \mathbb{R})$. We define the set $\mathcal{R}^{+}=\mathcal{R}^{+}(\mathbb{T}, \mathbb{R})=\{p\in\mathcal{R}: 1+\mu(t)p(t)> 0, \forall t\in{\mathbb{T}}\}$.

If $r\in\mathcal{R}$, then the generalized exponential function $e_{r}$ is defined by
\begin{eqnarray*}
e_{r}(t, s)=\exp\bigg\{\int_s^t\xi_{\mu(\tau)}(r(\tau))\Delta\tau\bigg\}
\end{eqnarray*}
for all $s,t\in\mathbb{T}$, with the cylinder transformation
\begin{eqnarray*}
\xi_h(z)=\bigg\{\begin{array}{ll} {\displaystyle\frac{\mathrm{Log}(1+hz)}{h}},\,\,h\neq 0,\\
z,\,\,\,\,\,\,\,\quad\quad\quad\quad h=0.\\
\end{array}
\end{eqnarray*}

Let $p,q:\mathbb{T}\rightarrow\mathbb{R}$ be two regressive functions, we define
\[
p\oplus q=p+q+\mu pq,\,\,\,\,\ominus p=-\frac{p}{1+\mu p},\,\,\,\,p\ominus q=p\oplus(\ominus q)=\frac{p-q}{1+\mu q}.
\]
Then the generalized exponential function has the following
properties.

\begin{lemma}\cite{17}
Assume that $p,q:\mathbb{T}\rightarrow\mathbb{R}$ are two regressive
functions, then
\begin{itemize}
    \item  [$(i)$]  $e_{0}(t,s)\equiv 1$ $\mathrm{and}$ $e_p(t,t)\equiv 1$;
    \item  [$(ii)$] $e_p(\sigma(t),s)=(1+\mu(t)p(t))e_p(t,s)$;
    \item  [$(iii)$]$e_p(t,s)=1/e_p(s,t)=e_{\ominus p}(s,t)$;
    \item  [$(iv)$]  $e_p(t,s)e_p(s,r)=e_p(t,r)$;
    \item  [$(v)$] $e_p(t,s)e_q(t,s)=e_{p\oplus q}(t,s)$;
    \item  [$(vi)$] $e_p(t,s)/e_q(t,s)=e_{p\ominus q}(t,s)$;
    \item  [$(vi)$] $\big(\frac{1}{e_p(t,s)}\big)^{\Delta}=\frac{-p(t)}{e^{\sigma}_p(t,s)}$.
\end{itemize}
\end{lemma}

\begin{definition}\cite{22} A time scale $\mathbb{T}$ is called an almost periodic time scale if
\begin{eqnarray*}
\Pi=\big\{\tau\in\mathbb{R}: t\pm\tau\in\mathbb{T}, \forall t\in{\mathbb{T}}\big\}\neq\{0\}.
\end{eqnarray*}
\end{definition}
\begin{definition}\label{defli2}
Let $\mathbb{T}$ be an  almost periodic
time scale.  A function  $f\in
C(\mathbb{T}\times D,\mathbb{E}^n)$ is  called an almost
periodic function in $t\in \mathbb{T}$ uniformly for $x\in D$ if the
$\varepsilon$-translation set of $f$
$$E\{\varepsilon,f,S\}=\{\tau\in\Pi:|f(t+\tau,x)-f(t,x)|<\varepsilon,\,\,
\forall (t,x)\in   \mathbb{T}\times S\}$$ is relatively
dense for all $\varepsilon>0$ and for each
compact subset $S$ of $D$, that is, for any given $\varepsilon>0$
and each compact subset $S$ of $D$, there exists a constant
$l(\varepsilon,S)>0$ such that each interval of length
$l(\varepsilon,S)$ contains a $\tau(\varepsilon,S)\in
E\{\varepsilon,f,S\}$ such that
\begin{equation*}
|f(t+\tau,x)-f(t,x)|<\varepsilon, \,\,\forall t\in
\mathbb{T}\times S.
\end{equation*}
This $\tau$ is called the $\varepsilon$-translation number of $f$.
\end{definition}

For convenience, $PC_{rd}(\mathbb{T}, \mathbb{R}^{n})$ denotes the set of all piecewise continuous functions with respect to a sequence $\{t_{k}\}$, $k\in\mathbb{Z}$. For any integers $i$ and $j$, denote $t_{k}^{j}=t_{k+j}-t_{k}$ and consider the sequence $\{t_{k}^{j}\}$, $k, j\in\mathbb{Z}$. It is easy to verify that the number $t_{k}^{j}$, $k, j\in\mathbb{Z}$ satisfy
\begin{eqnarray*}
t_{k+i}^{j}-t_{k}^{j}=t_{k+j}^{i}-t_{k}^{i},\,\,t_{k}^{j}-t_{k}^{i}=t_{k+i}^{j-i},\,\,k, j,i\in\mathbb{Z}.
\end{eqnarray*}

\begin{definition}\cite{25}
The set of sequences $\{t_{k}^{j}\}$, $t_{k}^{j}=t_{k+j}-t_{k}$, $k, j\in\mathbb{Z}$ is said to be uniformly almost periodic, if for an arbitrary $\varepsilon>0$, there exists a relatively dense set of $\varepsilon$-almost periods, common for all sequences $\{t_{k}^{j}\}$.
\end{definition}

\begin{definition}\label{def24}
Let $\mathbb{T}$ be an almost periodic time scale and assume that $\{t_{k}\}\subset\mathbb{T}$, we call a function $\varphi\in PC_{rd}(\mathbb{T}, \mathbb{R}^{n})$ is almost periodic, if the following holds:
\begin{itemize}
     \item  [$(i)$] $\{t_{k}^{j}\}$, $t_{k}^{j}=t_{k+j}-t_{k}$, $k, j\in\mathbb{Z}$ is uniformly almost periodic;
     \item  [$(ii)$] for any $\varepsilon>0$, there is a positive number $\delta=\delta(\varepsilon)$ such that if the points $t^{'}$ and $t^{''}$ belong to the same interval of continuity and $|t^{'}-t^{''}|<\delta$, then $\|\varphi(t^{'})-\varphi(t^{''})\|<\varepsilon$;
    \item  [$(iii)$] for any $\varepsilon>0$, there is relative dense set $\Gamma_{\varepsilon}\subset\Pi$ of $\varepsilon$-almost periods such that, if $\tau\in\Gamma_{\varepsilon}$, then $\|\varphi(t+\tau)-\varphi(t)\|<\varepsilon$ for all $t\in\mathbb{T}$ satisfying the condition $|t-t_{k}|>\varepsilon$, $k\in\mathbb{Z}$.
\end{itemize}
\end{definition}

We give a concept of almost periodic functions in the sense of  Bohr as follows.
\begin{definition}\label{def25}
Let $\mathbb{T}$ be an almost periodic time scale and assume that $\{t_{k}\}\subset\mathbb{T}$, we call a function $\varphi\in PC_{rd}(\mathbb{T}\times D, \mathbb{R}^{n})$ is almost periodic in t uniformly for $x\in D$, if, for each compact subset $S$ of $D$, the following holds:
\begin{itemize}
     \item  [$(i)$] $\{t_{k}^{j}\}$, $t_{k}^{j}=t_{k+j}-t_{k}$, $k, j\in\mathbb{Z}$ is uniformly almost periodic;
     \item  [$(ii)$] for any $\varepsilon>0$, there is a positive number $\delta=\delta(\varepsilon, S)$ such that if the points $t^{'}$ and $t^{''}$ belong to the same interval of continuity and $|t^{'}-t^{''}|<\delta$, then $\|\varphi(t^{'})-\varphi(t^{''})\|<\varepsilon$;
    \item  [$(iii)$] for any $\varepsilon>0$, there is relative dense set $\Gamma_{\varepsilon}\subset\Pi$ of $\varepsilon$-almost periods such that, if $\tau\in\Gamma_{\varepsilon}$, then $\|\varphi(t+\tau)-\varphi(t)\|<\varepsilon$ for all $(t,x)\in\mathbb{T}\times S$ satisfying the condition $|t-t_{k}|>\varepsilon$, $k\in\mathbb{Z}$.
\end{itemize}
\end{definition}

Let $T,P\in\mathcal{B}$, and let $s(T\cup P):\mathcal{B}\rightarrow\mathcal{B}$ be a map such that the set $s(T\cup P)$ forms a strictly increasing sequence. For $D\subset\mathbb{T}$ and $\varepsilon>0$, $F_{\varepsilon}(D)$ is a closed $\varepsilon$-neighborhood of the set $D$.

\begin{definition}\label{def26}
The set $T\in\mathcal{B}$ is almost periodic, if for every sequence $\{s_m^{'}\}\subset\Pi$ there exists a subsequence $\{s_n\}$, $s_n=s_{m_{n}}^{'}$ such that $T-s_n=\{t_k-s_n\}$ is uniformly convergent for $n\rightarrow\infty$ to the set $T_{1}\in\mathcal{B}$.
\end{definition}

Similar to the proof of Theorem 1 in \cite{139}, one can easily show the following lemma:
\begin{lemma}\label{lemd11} Let $\mathbb{T}$ be an almost periodic time scale and   $\{t_{i}\}\subset\mathbb{T}$. The set of sequences $\{t_k^j\}, t_k^j=t_{k+j}-t_k, k,j\in \mathbb{Z}$ is uniformly almost periodic if and only if for every sequence $\{s_m^{'}\}\subset\Pi$ there exists a subsequence $\{s_n\}$, $s_n=s_{m_{n}}^{'}$ such that $T-s_n=\{t_k-s_n\}$ is uniformly convergent for $n\rightarrow\infty$ on $\mathcal{B}$.
\end{lemma}
\begin{definition}\label{def27}
The sequence $\{\phi_n\}$, $\phi_n=(\varphi_n(t), T_n)\in PC_{rd}(\mathbb{T}, \mathbb{R}^{n})\times\mathcal{B}$ is convergent to $\phi$, $\phi=(\varphi(t), T)\in PC_{rd}(\mathbb{T}, \mathbb{R}^{n})\times\mathcal{B}$, if and only if, for every $\varepsilon>0$ there exists $n_0>0$ such that $n\geq n_0$ implies
\begin{eqnarray*}
 \rho(T, T_n)<\varepsilon,\,\,\|\varphi_n(t)-\varphi(t)\|<\varepsilon
  \end{eqnarray*}
uniformly for $t\in\mathbb{T}\setminus F_{\varepsilon}(s(T_n\cup T))$, $\rho(\cdot, \cdot)$ is an arbitrary distance in $\mathcal{B}$.
\end{definition}

We give another concept of almost periodic functions as follows.
\begin{definition}\label{def28}
The function $\varphi\in PC_{rd}(\mathbb{T}, \mathbb{R}^{n})$ is said to be an almost
periodic piecewise continuous function with points of discontinuity of the first
kind from the set $T\in\mathcal{B}$, if:
\begin{itemize}
     \item  [$(i)$] for every $\varepsilon>0$, there is a positive number $\delta=\delta(\varepsilon)$ such that if the points $t^{'}$ and $t^{''}$ belong to the same interval of continuity and $|t^{'}-t^{''}|<\delta$, then $\|\varphi(t^{'})-\varphi(t^{''})\|<\varepsilon$;
     \item  [$(ii)$] for every sequence $\{s_m^{'}\}\subset\Pi$ there exists a subsequence $\{s_n\}\subset\{s_m^{'}\}$ such that $(\varphi(t+s_n), T-s_n)$ is uniformly convergent on $PC_{rd}(\mathbb{T}, \mathbb{R}^{n})\times\mathcal{B}$.
\end{itemize}
\end{definition}

\begin{remark}
In Definition \ref{def28}, if we remove Condition $(i)$, then  Definition \ref{def28} is  the definition of almost periodic functions in the sense of Bochner.
\end{remark}

\begin{definition}\label{def29}
The sequence $\{\phi_n\}$, $\phi_n=(\varphi_n(t,x), T_n)\in PC_{rd}(\mathbb{T}\times\Omega, \mathbb{R}^{n})\times\mathcal{B}$ is convergent to $\phi$, $\phi=(\varphi(t,x), T)\in PC_{rd}(\mathbb{T}\times\Omega, \mathbb{R}^{n})\times\mathcal{B}$, if and only if, for every $\varepsilon>0$ and for each compact subset $S$ of $\Omega$, there exists $n_0>0$ such that $n\geq n_0$ implies
\begin{eqnarray*}
 \rho(T, T_n)<\varepsilon,\,\,\|\varphi_n(t,x)-\varphi(t,x)\|<\varepsilon
  \end{eqnarray*}
uniformly for $(t,x)\in(\mathbb{T}\setminus F_{\varepsilon}(s(T_n\cup T)))\times S$, $\rho(\cdot, \cdot)$ is an arbitrary distance in $\mathcal{B}$.
\end{definition}

\begin{definition}\label{def210}
The function $\varphi\in PC_{rd}(\mathbb{T}\times D, \mathbb{R}^{n})$ is said to be almost periodic in $t$ uniformly for $x\in D$, if for each compact subset $S$ of $D$:
\begin{itemize}
     \item  [$(i)$]  for every $\varepsilon>0$,  there is a positive number $\delta=\delta(\varepsilon, S)$ such that if the points $t^{'}$ and $t^{''}$ belong to the same interval of continuity and $|t^{'}-t^{''}|<\delta$, then $\|\varphi(t^{'},\cdot)-\varphi(t^{''},\cdot)\|<\varepsilon$;
     \item  [$(ii)$] for every sequence $\{s_m^{'}\}\subset\Pi$ there exists a subsequence $\{s_n\}\subset\{s_m^{'}\}$ such that $(\varphi(t+s_n,x), T-s_n)$ is uniformly convergent on $PC_{rd}(\mathbb{T}\times D, \mathbb{R}^{n})\times\mathcal{B}$.
\end{itemize}
\end{definition}

\begin{remark}According to Lemma \ref{lemd11} and Theorem 129 in \cite{ws},
it is easy to see that Definition \ref{def24} is equivalent to Definition \ref{def28} and Definition \ref{def25} is equivalent to Definition \ref{def210}, respectively.
\end{remark}

Let $\varphi, \psi\in PC_{rd}(\mathbb{T}, \mathbb{R}^{n})$ with points of discontinuity of the first
kind from the set $T\in\mathcal{B}$, then similar to the proofs of Theorem 1.15, Theorem 1.17 and Theorem 1.18 in \cite{25}, one can easily show the following lemmas:
\begin{lemma}
If $\varphi$ is almost periodic, then $\varphi$ is bounded.
\end{lemma}

\begin{lemma}
If $\varphi$ is almost periodic and $F(\cdot)$ is uniformly continuous on the value field of $\varphi$, then $F\circ\varphi$ is almost periodic.
\end{lemma}

\begin{lemma}
If $\varphi, \psi$ are almost periodic, then $\varphi+\psi$ is almost periodic.
\end{lemma}
\begin{lemma} \cite{17}
Assume that $a\in\mathcal{R}$ and $t_{0}\in{\mathbb{T}}$, if $a\in\mathcal{R}^{+}$ on $\mathbb{T}^{k}$, then $e_a(t,t_{0})> 0$ for all $t\in{\mathbb{T}}$.
\end{lemma}

Similar to the proof of Lemma 2.3 in \cite{21}, one can show that:
\begin{lemma}\label{lemd12} Let $f\in PC^{1}_{rd}[\mathbb{T}, \mathbb{R}]$, if $f(t)>0$ for $t\in\mathbb{T}$, then
\begin{eqnarray*}
\frac{f^{\Delta}(t)}{f^{\sigma}(t)}\leq[\ln (f(t))]^{\Delta}\leq\frac{f^{\Delta}(t)}{f(t)},
\end{eqnarray*}
where $PC_{rd}^{1}[\mathbb{T}, \mathbb{R}]=\{y:\mathbb{T}\rightarrow\mathbb{R}$ is rd-continuous except at $t_k$, $k=1, 2,\ldots$, for which
$y(t_{k}^{-})$, $y(t_{k}^{+})$, $y^{\Delta}(t_{k}^{-})$, $y^{\Delta}(t_{k}^{+})$ exist with $y(t_{k}^{-})=y(t_{k})$, $y^{\Delta}(t_{k}^{-})=y^{\Delta}(t_{k})\}.$
\end{lemma}

\begin{lemma}\label{lem27}\cite{26}
Assume that $x\in PC^{1}_{rd}[\mathbb{T}, \mathbb{R}]$ and
\begin{eqnarray*}
 \left\{%
\begin{array}{lcrcl}
x^{\Delta}(t)\leq(\geq) p(t)x(t)+q(t),\,\,t\neq t_{k},\,\,t\in[t_{0}, +\infty)_{\mathbb{T}},\\
x(t_{k}^{+})\leq(\geq) d_{k}x(t_{k})+b_{k},\,\,t=t_{k},\,\,k\in\mathbb{N},
\end{array}\right.
\end{eqnarray*}
then for $t\geq t_{0}\geq0$,
\begin{eqnarray*}
x(t)&\leq(\geq)& x(t_{0})\prod_{t_{0}<t_{k}<t}d_{k}e_{p}(t, t_{0})+\sum_{t_{0}<t_{k}<t}\bigg(\prod_{t_{0}<t_{j}<t}d_{j}e_{p}(t, t_{k})\bigg)b_{k}\\
&&+\int_{t_{0}}^{t}\prod_{s<t_{k}<t}d_{k}e_{p}(t, \sigma(s))q(s)\Delta s.
\end{eqnarray*}
\end{lemma}

\section{Comparison theorems}

\setcounter{equation}{0}
{\setlength\arraycolsep{2pt}}
 \indent

 In this section, we state and prove some comparison theorems.
\begin{lemma}\label{lem211}
Assume that $x\in PC^{1}_{rd}[\mathbb{T}, \mathbb{R}]$, $x(t)>0$ on $\mathbb{T}, d\geq0, a, b> 0, t-\tau(t)\in\mathbb{T}$ for $t\in \mathbb{T}, 0<d_k\leq 1,b_k\leq 0$, where $\tau:\mathbb{T}\rightarrow\mathbb{R}^{+}$ is a rd-continuous function  and $\bar{\tau}=\sup\limits_{t\in \mathbb{T}}\{\tau(t)\}, k\in \mathbb{N}$, and there exist positive constants $\alpha,\beta$ such that $\alpha\leq\prod_{t_{0}<t_{k}<t}d_{k}\leq \beta$ for $t\in \mathbb{T}$. Then the following hold:
\begin{itemize}
    \item  [$(i)$] If
     {\setlength\arraycolsep{2pt}
\begin{eqnarray}\label{e24}
 \left\{%
\begin{array}{lcrcl}
x^{\Delta}(t)\leq x^\sigma(t)(b-ax(t-\tau(t)))+d,\,\,t\neq t_{k},\,\,t\in[t_{0}, +\infty)_{\mathbb{T}},\\
x(t_{k}^{+})\leq d_{k}x(t_{k})+b_{k},\,\,t=t_{k},\,\,k\in\mathbb{N},
\end{array}\right.
\end{eqnarray}}
where $t_{0}\in{\mathbb{T}}$, with initial condition
\begin{eqnarray*}
x(t; t_{0}, \phi_{0})=\phi_{0}(t),\,\,t\in[t_{0}-\bar{\tau}, t_{0}]_{\mathbb{T}},
\end{eqnarray*}
where $\phi_0\in C([t_{0}-\bar{\tau}, t_{0}],(0,+\infty))$, then
 \begin{eqnarray*}
 \limsup_{t\rightarrow+\infty}x(t)\leq -\frac{d\beta}{b}+\bigg(\frac{d\beta}{b}+ \bar{x}\beta\bigg)\exp\bigg\{\frac{b\bar{\tau}}{1-b\bar{\mu}}\bigg\}:=M,
   \end{eqnarray*}
where $\bar{x}$ is the unique positive root of $x(ax-b)-d=0$.

Especially, if $d=0$, then
${M}=\frac{b\beta}{a}\exp\big\{\frac{b\bar{\tau}}{1-b\bar{\mu}}\big\}.$

    \item  [$(ii)$] If {\setlength\arraycolsep{2pt}
\begin{eqnarray}\label{e25}
 \left\{%
\begin{array}{lcrcl}
x^{\Delta}(t)\geq x(t)(b-ax(t-\tau(t)))+d,\,\,t\neq t_{k},\,\,t\in[t_{0}, +\infty)_{\mathbb{T}},\\
x(t_{k}^{+})\geq d_{k}x(t_{k}),\,\,t=t_{k},\,\,k\in\mathbb{N},
\end{array}\right.
\end{eqnarray}}
where $t_{0}\in{\mathbb{T}}$,  with initial condition
\begin{eqnarray*}
x(s; t_{0}, \phi_{0})=\phi_{0}(s),\,\,\phi_{0}(t_0)>0,\,\,s\in[t_{0}-\bar{\tau}, t_{0}]_{\mathbb{T}},
\end{eqnarray*}
where $\phi_0\in C([t_{0}-\bar{\tau}, t_{0}],[0,+\infty)]$, and there exists a positive constant $N$ such that
 $0\leq\limsup_{t\rightarrow+\infty}x(t)\leq N<+\infty$ and $1-aN\bar{\mu}>0$, then
    \begin{eqnarray*}
 \liminf_{t\rightarrow+\infty}x(t)\geq \frac{b\alpha^2}{a}\exp\bigg\{\frac{\log(1-aN\bar{\mu})}{\bar{\mu}}\bar{\tau}\bigg\}:=m.
   \end{eqnarray*}
\end{itemize}
\end{lemma}
\begin{proof}
The proof of $(i)$.  Suppose that $\limsup\limits_{t\rightarrow+\infty}x(t)=+\infty$. Then we claim that
there must exist an  $s_{1}(>0)\in \mathbb{T}$ and $s_1\notin \{t_k\}$  such that
 \begin{eqnarray*}
  x(s_{1})\geq \bar{x}+1,\quad x^{\Delta}(t)|_{t=s_{1}}\geq0.
\end{eqnarray*}
Otherwise, since $\limsup\limits_{t\rightarrow+\infty}x(t)=+\infty$, for every $t\in [t_0,+\infty)_\mathbb{T}\setminus \{t_k\}$, we have the following two cases:
\begin{itemize}
  \item Case (1) For every $t\in [t_0,+\infty)_\mathbb{T}\setminus \{t_k\}, x(t)<\bar{x}+1$. In this case, since $x(t_k^+)\leq d_{k}x(t_{k})+b_{k}\leq x(t_k), k\in \mathbb{N}$, so, $x(t)$ is bounded on $[t_0,+\infty)_\mathbb{T}$, which contradicts $\limsup\limits_{t\rightarrow+\infty}x(t)=+\infty$.
  \item Case (2) For every $t\in [t_0,+\infty)_\mathbb{T}\setminus \{t_k\}$, if $x(t)\geq \bar{x}+1$, then  $x^{\Delta}(t)<0$. In this case, since $x^{\Delta}(t)<0$ and $x(t_k^+)\leq d_{k}x(t_{k})+b_{k}\leq x(t_k), k\in \mathbb{N}$, so, $x(t)$ is decreasing on $[t_0,+\infty)_\mathbb{T}$, which also contradicts $\limsup\limits_{t\rightarrow+\infty}x(t)=+\infty$.
\end{itemize}
Therefore, our claim holds. Similarly, by induction, we can choose a  sequence $\{s_i\}\in (0,+\infty)_\mathbb{T}\setminus\{t_k\}$ satisfying
 \begin{eqnarray*}
  x(s_{i})\geq \bar{x}+i,\quad x^{\Delta}(t)|_{t=s_{i}}\geq0,\quad i=1, 2, \ldots.
\end{eqnarray*}
Thus, from \eqref{e24} we have
\begin{eqnarray*}
x^\sigma(s_{i})(b-ax(s_{i}-\tau(s_{i})))+d\geq 0,\,\,i=1, 2, \ldots,
\end{eqnarray*}
so
\begin{eqnarray}\label{e26}
x(s_{i}-\tau(s_{i}))\leq\frac{1}{a}\bigg(b+\frac{d}{x^\sigma(s_{i})}\bigg)\leq \frac{1}{a}\bigg(b+\frac{d}{\bar{x}}\bigg)=\bar{x},\,\,i=1, 2, \ldots.
\end{eqnarray}
Since
\begin{eqnarray*}
x^{\Delta}(t)\leq bx^{\sigma}(t)+d=b(x(t)+\mu(t)x^{\Delta}(t))+d,
\end{eqnarray*}
we have
\begin{eqnarray*}
(1-b\bar{\mu})x^{\Delta}(t)\leq(1-b\mu(t))x^{\Delta}(t)\leq bx(t)+d,
\end{eqnarray*}
then
\begin{eqnarray*}
x^{\Delta}(t)\leq \frac{b}{1-b\bar{\mu}}x(t)+\frac{d}{1-b\bar{\mu}}.
\end{eqnarray*}
Consider the following inequality:
\begin{eqnarray*}
 \left\{%
\begin{array}{lcrcl}
x^{\Delta}(t)\leq \frac{b}{1-b\bar{\mu}}x(t)+\frac{d}{1-b\bar{\mu}},\,\,t\neq t_{k},\,\,t\in[t_{0}^{\ast}, +\infty)_{\mathbb{T}},\,\,t_{0}^{\ast}\geq t_{0},\\
x(t_{k}^{+})\leq d_{k}x(t_{k})+b_{k},\,\,t=t_{k},\,\,k\in\mathbb{N}.
\end{array}\right.
\end{eqnarray*}
Since $b_k\leq 0$, for $t>t_{0}^{\ast}\geq t_{0}$, by use of Lemma \ref{lem27}, we have
\begin{eqnarray*}
x(t)\leq x(t_{0}^{\ast})\prod_{t_{0}^{\ast}<t_{k}<t}d_{k}e_{\frac{b}{1-b\bar{\mu}}}(t, t_{0}^{\ast})+\int_{t_{0}^{\ast}}^{t}\prod_{s<t_{k}<t}d_{k}e_{\frac{b}{1-b\bar{\mu}}}(t, \sigma(s))\frac{d}{1-b\bar{\mu}}\Delta s.
\end{eqnarray*}
In view of $\prod_{t_{0}<t_{k}<t}d_{k}\leq \beta$, for $t>t^*$, we have
\begin{eqnarray}\label{e27}
x(t)&\leq& x(t_{0}^{\ast})\beta e_{\frac{b}{1-b\bar{\mu}}}(t, t_{0}^{\ast})+\beta\int_{t_{0}^{\ast}}^{t}e_{\frac{b}{1-b\bar{\mu}}}(t, \sigma(s))\frac{d}{1-b\bar{\mu}}\Delta s\nonumber\\
&\leq& x(t_{0}^{\ast})\beta e_{\frac{b}{1-b\bar{\mu}}}(t, t_{0}^{\ast})-\frac{d\beta}{b}[1-e_{\frac{b}{1-b\bar{\mu}}}(t, t_{0}^{\ast})]\nonumber\\
&=&-\frac{d\beta}{b}+\bigg(\frac{d\beta}{b}+ x(t_{0}^{\ast})\beta\bigg) e_{\frac{b}{1-b\bar{\mu}}}(t, t_{0}^{\ast}).
\end{eqnarray}
According to \eqref{e26} and \eqref{e27}, for $i=1, 2, \ldots$, we obtain
\begin{eqnarray}\label{e28}
x(s_{i})&\leq&-\frac{d\beta}{b}+\bigg(\frac{d\beta}{b}+ x(s_{i}-\tau(s_{i}))\beta\bigg) e_{\frac{b}{1-b\bar{\mu}}}(s_{i}, s_{i}-\tau(s_{i}))\nonumber\\
&\leq&-\frac{d\beta}{b}+\bigg(\frac{d\beta}{b}+ \bar{x}\beta\bigg) e_{\frac{b}{1-b\bar{\mu}}}(s_{i}, s_{i}-\tau(s_{i})).
\end{eqnarray}

For every $\theta\in\mathbb{T}$, if $\mu(\theta)=0$, then
\begin{eqnarray*}
\xi_{\mu}\bigg(\frac{b}{1-b\bar{\mu}}\bigg)=\frac{b}{1-b\bar{\mu}},
\end{eqnarray*}
if $\mu(\theta)\neq0$, then
\begin{eqnarray*}
\xi_{\mu}\bigg(\frac{b}{1-b\bar{\mu}}\bigg)=\frac{\log(1+\frac{b}{1-b\bar{\mu}}\mu(\theta))}{\mu(\theta)}\leq \frac{b}{1-b\bar{\mu}}.
\end{eqnarray*}
Hence, we have
\begin{eqnarray*}
\int_{s_{i}-\tau(s_{i})}^{s_{i}}\xi_{\mu}\bigg(\frac{b}{1-b\bar{\mu}}\bigg)\Delta\theta&\leq&\int_{s_{i}-\tau(s_{i})}^{s_{i}}\frac{b}{1-b\bar{\mu}}\Delta\theta\leq b\tau(s_i)\leq \frac{b\bar{\tau}}{1-b\bar{\mu}},\,\,i=1, 2, \ldots,
\end{eqnarray*}
so
\begin{eqnarray*}
\exp\bigg\{\int_{s_{i}-\tau(s_{i})}^{s_{i}}\xi_{\mu}\bigg(\frac{b}{1-b\bar{\mu}}\bigg)\bigg\}\Delta\theta\leq \exp\bigg\{\frac{b}{1-b\bar{\mu}}\bigg\},\,\,k=1, 2, \ldots.
\end{eqnarray*}
Thus
\begin{eqnarray}\label{e29}
e_{\frac{b}{1-b\bar{\mu}}}(s_{i}, t_{k})<e_{\frac{b}{1-b\bar{\mu}}}(s_{i}, s_{i}-\tau(s_{i}))\leq \exp\bigg\{\frac{b}{1-b\bar{\mu}}\bigg\},\,\,i=1, 2, \ldots.
\end{eqnarray}
It follows from $\eqref{e28}$ and $\eqref{e29}$ that
\begin{eqnarray*}
x(s_{i})\leq-\frac{d\beta}{b}+\bigg(\frac{d\beta}{b}+ \bar{x}\beta\bigg)\exp\bigg\{\frac{b}{1-b\bar{\mu}}\bigg\},\,\,i=1, 2, \ldots.
\end{eqnarray*}
Hence, $\limsup\limits_{i\rightarrow+\infty}x(s_{i})<+\infty$. This contradicts the assumption.

Similarly, we can get
$\limsup_{t\rightarrow+\infty}x(t)\leq M.$

The proof of (ii).
For any positive constant $\varepsilon$ small enough, it follows from $\limsup_{t\rightarrow+\infty}x(t)\leq N$ that there exists large enough $T_{1}$
such that
\begin{eqnarray*}
 x(t)\leq N+\varepsilon,\,\, t>T_{1},
\end{eqnarray*}
then $x(t-\tau(t))\leq N+\varepsilon$ for $t>T_{1}+\bar{\tau}$. So we have
\begin{eqnarray*}\label{e211}
 x^{\Delta}(t)&\geq& x(t)(b-ax(t-\tau(t)))+d\nonumber\\
 &\geq& -a({N}+\varepsilon)x(t),\,\,t\geq T_1+\bar{\tau},\,\,t\neq t_{k}, k\in \mathbb{N}.
\end{eqnarray*}

Consider the following inequality
\begin{eqnarray*}
 \left\{%
\begin{array}{lcrcl}
x^{\Delta}(t)\geq -a(N+\varepsilon)x(t),\,\,t\neq t_{k},\,\,t\in[t_{0}^{\ast}, +\infty)_{\mathbb{T}},\,\,t_{0}^{\ast}\geq T_{1}+\bar{\tau}, t_0^*\in \mathbb{T},\\
x(t_{k}^{+})\geq d_{k}x(t_{k}),\,\,t=t_{k},\,\,k\in\mathbb{N}.
\end{array}\right.
\end{eqnarray*}
For $t\geq t_{0}^{\ast}$, by use of Lemma \ref{lem27}, we have
\begin{eqnarray*}\label{2.9}
 x(t)\geq x(t_{0}^{\ast})\alpha  e_{-a(N+\varepsilon)}(t, t_{0}^{\ast}).
 \end{eqnarray*}
 Especially, we have
 \begin{eqnarray}\label{2.10}
 x(t)\geq x(t-\tau(t))\alpha  e_{-a(N+\varepsilon)}(t, t-\tau(t)), \,\, t\geq t_0^*+\bar{\tau}.
 \end{eqnarray}

For every $\theta\in\mathbb{T}$, if $\mu(\theta)=0$, then
\begin{eqnarray*}
\xi_{\mu}(-a(N+\varepsilon))=-a(N+\varepsilon),
\end{eqnarray*}
if $\mu(\theta)\neq0$, then
\begin{eqnarray*}
\xi_{\mu}(-a(N+\varepsilon))=\frac{\log(1-a(N
+\varepsilon)\mu(\theta))}{\mu(\theta)}
\geq\frac{\log(1-a(N+\varepsilon)\bar{\mu})}{\bar{\mu}}<-a(N+\varepsilon).
\end{eqnarray*}
Hence,  we have
\begin{eqnarray*}
\int_{t-\tau(t)}^{t}\xi_{\mu}(-a(N+\varepsilon))\Delta\theta
&\geq&\min\bigg\{\int_{t-\tau(t)}^{t}-a(N+\varepsilon)\Delta\theta, \int_{t-\tau(t)}^{t}\frac{\log(1-a(N+\varepsilon)\bar{\mu})}{\bar{\mu}}\Delta\theta\bigg\}\\
&=& \frac{\log(1-a(N+\varepsilon)\bar{\mu})\tau(t)}{\bar{\mu}}\\
&\geq&\frac{\log(1-a(N+\varepsilon)\bar{\mu})}{\bar{\mu}}\bar{\tau}, \,\, t\geq t_0^*+\bar{\tau},
\end{eqnarray*}
so
\begin{eqnarray*}
\exp\bigg\{\int_{t-\tau(t)}^{t}\xi_{\mu}(-a(N+\varepsilon))\bigg\}\Delta\theta
\geq \exp\bigg\{\frac{\log(1-a(\tilde{N}+\varepsilon)\bar{\mu})}{\bar{\mu}}\bar{\tau}\bigg\}, \,\, t\geq t_0^*+\bar{\tau}.
\end{eqnarray*}
Thus
\begin{eqnarray}\label{e214}
e_{-a(N+\varepsilon)}(t, t-\tau(t))\geq \exp\bigg\{\frac{\log(1-a(N+\varepsilon)\bar{\mu})}{\bar{\mu}}\bar{\tau}\bigg\}, \,\, t\geq t_0^*+\bar{\tau}.
\end{eqnarray}
By use of \eqref{2.10} and \eqref{e214}, we obtain
\begin{eqnarray}\label{2.12}
 x(t)\geq \alpha x(t-\tau(t))\exp\bigg\{\frac{\log(1-a(N+\varepsilon)\bar{\mu})}{\bar{\mu}}\bar{\tau}\bigg\}, \,\, t\geq t_0^*+\bar{\tau}.
   \end{eqnarray}

 From \eqref{e25}, by using Lemma \ref{lemd12},  we have
{\setlength\arraycolsep{2pt}
\begin{eqnarray}\label{e251}
 \left\{%
\begin{array}{lcrcl}
(\ln x(t))^\Delta\geq\frac{x(t)}{x(\sigma(t))}( b-ax(t-\tau(t))),\,\,t\neq t_{k},\,\,t\in[t_{0}, +\infty)_{\mathbb{T}},\\
x(t_{k}^{+})\geq d_{k}x(t_{k}),\,\,t=t_{k},\,\,k\in\mathbb{N}.
\end{array}\right.
\end{eqnarray}}
 Integrating
system \eqref{e251} from $t_0$ to $t$,  we have
\begin{eqnarray}\label{e2521}
x(t)\geq x(t_0)\exp \bigg(\int_{t_0}^{t}\frac{x(s)}{x(\sigma(s))}(b-ax(s-\tau(s)))\Delta s+\sum\limits_{t_0\leq t_k<t}\ln d_k\bigg).
\end{eqnarray}

For any fixed $\vartheta\in (0,1)$, we only need to consider
the following three cases.
\begin{itemize}
  \item Case I: there is a $t_1 > t_{0}$ such that $x(t-\tau(t)) \geq \frac{\vartheta b}{a}$ for all $t \geq t_1$.
  \item Case II: there is a $t_1 > t_{0}$ such that $x(t-\tau(t)) \leq \frac{\vartheta b}{a}$ for all $t \geq t_1$.
  \item Case III: $x(t-\tau(t))$ is oscillatory about $\frac{\vartheta b}{a}$ for all $t\geq t_0$.
\end{itemize}

Case I. In this case, it follows from \eqref{2.12}  that for $t\geq \max\{t_1,t_0^*+\bar{\tau}\}$,
 \begin{eqnarray*}
 x(t)\geq\frac{\vartheta b\alpha}{a}\exp\bigg\{\frac{\log(1-a(N+\varepsilon)\bar{\mu})}{\bar{\mu}}\bar{\tau}\bigg\}.
   \end{eqnarray*}

Letting $\varepsilon\rightarrow0$ and $\vartheta\rightarrow 1$, we have
\begin{eqnarray}\label{a12}
 \liminf_{t\rightarrow+\infty}x(t)\geq\frac{ b\alpha}{a}\exp\bigg\{\frac{\log(1-aN\bar{\mu})}{\bar{\mu}}\bar{\tau}\bigg\}\geq\frac{ b\alpha^2}{a}\exp\bigg\{\frac{\log(1-aN\bar{\mu})}{\bar{\mu}}\bar{\tau}\bigg\}=m.
   \end{eqnarray}

Case II. In this case, it follows from  \eqref{e25} that $x^\Delta(t)\geq 0$ for $t\geq t_1$. Hence, by \eqref{e2521}, we find
\begin{eqnarray*}\label{e253}
x(t)&\geq &x(t_1+\bar{\tau})\exp \bigg(\frac{x(t_1+\bar{\tau})a}{\vartheta b}\int_{t_1+\bar{\tau}}^{t}(b-ax(s-\tau(s)))\Delta s+\sum\limits_{t_1+\bar{\tau}\leq t_k<t}\ln d_k\bigg)\nonumber\\
&\geq &x(t_1+\bar{\tau})\exp \bigg(\frac{x(t_1+\bar{\tau})a}{\vartheta b}(1-\vartheta)b(t-t_1+\bar{\tau})+\sum\limits_{t_1+\bar{\tau}\leq t_k<t}\ln d_k\bigg)\rightarrow+\infty
\end{eqnarray*}
as $ t\rightarrow+\infty,$ which contradicts $\limsup_{t\rightarrow+\infty}x(t)\leq N$.

Case III. In this case, from the oscillation of $x(t-\tau(t))$ about $\frac{\vartheta b}{a}$, we can choose two sequences $\{s_i\}$ and $\{s_i'\}$
 satisfying
$t_0<s_1<s_1'<\cdots<s_i<s_i'<\cdots$   and $\lim\limits_{i\rightarrow+\infty}s_i=\lim\limits_{i\rightarrow+\infty}s_i'=+\infty$
such that
\[
x(s_i-\tau(s_i)))\leq \frac{\vartheta b}{a}, \,\,\,x(s_i^+-\tau(s_i^+)))\geq \frac{\vartheta b}{a}, \,\,\, x(s_i'-\tau(s_i')))\geq \frac{\vartheta b}{a}, \,\,\,x(s_i'^+-\tau(s_i'^+)))\leq \frac{\vartheta b}{a},
\]
\[
x(t-\tau(t))\geq \frac{\vartheta b}{a}, \,\,\, t\in (s_i,s_i')
\]
and
\[
x(t-\tau(t))\leq \frac{\vartheta b}{a}, \,\,\, t\in (s_i',s_{i+1}).
\]

For any $t\geq t_0$,  if $t\in (s_i',s_{i+1}]$,  then by \eqref{e2521}, we have
\begin{eqnarray}\label{2.16}
x(t)&\geq& x(s_i')\exp \bigg(\int_{s_i'}^{t}\frac{x(s)}{x(\sigma(s))}(b-ax(s-\tau(s)))\Delta s+\sum\limits_{s_i'\leq t_k<t}\ln d_k\bigg)\nonumber\\
&\geq&x(s_i')\prod_{s_i'\leq t_{k}<t}d_{k}\geq \alpha x(s_i').
\end{eqnarray}
Noticing that $x(s_i'-\tau(s_i')))\geq \frac{\vartheta b}{a}$ and according to \eqref{2.12}, for $s_i'\geq \max\{t_0,t_0^*+\bar{\tau}\}$, we have
\begin{eqnarray}\label{2.17}
x(s_i')&\geq&  \alpha x(s_i'-\tau(s_i'))\exp\bigg\{\frac{\log(1-a(N+\varepsilon)\bar{\mu})}{\bar{\mu}}\bar{\tau}\bigg\}\nonumber\\
 &\geq&\frac{\vartheta b\alpha}{a}\exp\bigg\{\frac{\log(1-a(N+\varepsilon)\bar{\mu})}{\bar{\mu}}\bar{\tau}\bigg\}.
   \end{eqnarray}
    Combine \eqref{2.16} and \eqref{2.17},  we have
   \begin{eqnarray}\label{2.18}
x(t)\geq\frac{\vartheta b\alpha^2}{a}\exp\bigg\{\frac{\log(1-a(N+\varepsilon)\bar{\mu})}{\bar{\mu}}\bar{\tau}\bigg\}.
   \end{eqnarray}

For any $t\geq t_0^*+\bar{\tau}$,  if $t\in (s_i,s_{i}']$,  then by \eqref{2.12}, we have
\begin{eqnarray}\label{2.19}
x(t)&\geq&  \alpha x(t-\tau(t))\exp\bigg\{\frac{\log(1-a(N+\varepsilon)\bar{\mu})}{\bar{\mu}}\bar{\tau}\bigg\}\nonumber\\
 &\geq&\frac{\vartheta b\alpha}{a}\exp\bigg\{\frac{\log(1-a(N+\varepsilon)\bar{\mu})}{\bar{\mu}}\bar{\tau}\bigg\}.
   \end{eqnarray}
Letting $\varepsilon\rightarrow 0$ and $\vartheta\rightarrow 1$, from \eqref{2.18} and \eqref{2.19}, we have
  \begin{eqnarray}\label{2.17}
\liminf_{i\rightarrow+\infty}x(t)\geq \frac{b\alpha^2}{a}\exp\bigg\{\frac{\log(1-aN\bar{\mu})}{\bar{\mu}}\bar{\tau}\bigg\}.
   \end{eqnarray}
From the above discussion, we see that Case II cannot occur. So, it follows from \eqref{a12} and \eqref{2.17} that
 \begin{eqnarray*}
 \liminf_{i\rightarrow+\infty}x(t)\geq\frac{ b\alpha^2}{a}\exp\bigg\{\frac{\log(1-aN\bar{\mu})}{\bar{\mu}}\bar{\tau}\bigg\}=m.
   \end{eqnarray*}The proof is complete.
\end{proof}

Similarly, we can easily obtain the following results.

\begin{lemma}
Assume that $x\in PC^{1}_{rd}[\mathbb{T}, \mathbb{R}]$, $x(t)>0$ on $\mathbb{T}, b, d\geq0, a> 0, t-\tau(t)\in\mathbb{T}$ for $t\in \mathbb{T}, 0<d_k\leq 1,b_k\leq 0$, where $\tau:\mathbb{T}\rightarrow\mathbb{R}^{+}$ is a rd-continuous function  and $\bar{\tau}=\sup\limits_{t\in \mathbb{T}}\{\tau(t)\}, k\in \mathbb{N}$, and there exist positive constants $\alpha,\beta$ such that $\alpha\leq\prod_{t_{0}<t_{k}<t}d_{k}\leq \beta$ for $t\in \mathbb{T}$. Then the following hold:
\begin{itemize}
    \item  [$(i)$] If
     {\setlength\arraycolsep{2pt}
\begin{eqnarray*}
 \left\{%
\begin{array}{lcrcl}
x^{\Delta}(t)\leq x(t)(b-ax(t-\tau(t)))+d,\,\,t\neq t_{k},\,\,t\in[t_{0}, +\infty)_{\mathbb{T}},\\
x(t_{k}^{+})\leq d_{k}x(t_{k})+b_{k},\,\,t=t_{k},\,\,k\in\mathbb{N},
\end{array}\right.
\end{eqnarray*}}
where $t_{0}\in{\mathbb{T}}$, with initial condition
\begin{eqnarray*}
x(t; t_{0}, \phi_{0})=\phi_{0}(t),\,\,t\in[t_{0}-\bar{\tau}, t_{0}]_{\mathbb{T}},
\end{eqnarray*}
where $\phi_0\in C([t_{0}-\bar{\tau}, t_{0}],(0,+\infty))$, then
    \begin{eqnarray*}
 \limsup_{t\rightarrow+\infty}x(t)\leq -\frac{d\beta}{b}+\bigg(\frac{d\beta}{b}+ \bar{x}\beta\bigg)e^{b\bar{\tau}}:=M,
   \end{eqnarray*}
where $\bar{x}$ is the unique positive root of $x(ax-b)-d=0$.

Especially, if $d=0$, then
$M=\frac{b\beta}{a}e^{b\bar{\tau}}.$

 \item  [$(ii)$] If {\setlength\arraycolsep{2pt}
\begin{eqnarray*}
 \left\{%
\begin{array}{lcrcl}
x^{\Delta}(t)\geq x^{\sigma}(t)(b-ax(t-\tau(t)))+d,\,\,t\neq t_{k},\,\,t\in[t_{0}, +\infty)_{\mathbb{T}},\\
x(t_{k}^{+})\geq d_{k}x(t_{k})+b_{k},\,\,t=t_{k},\,\,k\in\mathbb{N},
\end{array}\right.
\end{eqnarray*}}
where $t_{0}\in{\mathbb{T}}$, $b_{k}\geq 0$, with initial condition
\begin{eqnarray*}
x(t; t_{0}, \phi_{0})=\phi_{0}(t),\,\,t\in[t_{0}-\bar{\tau}, t_{0}]_{\mathbb{T}},
\end{eqnarray*} and there exists a positive constant $N$ such that
 $0\leq\limsup_{t\rightarrow+\infty}x(t)\leq \tilde{N}<+\infty$, then
    \begin{eqnarray*}
 \liminf_{t\rightarrow+\infty}x(t)\geq \frac{b\alpha^2}{a}\exp\bigg\{\frac{-\log(1+aN\bar{\mu})}{\bar{\mu}}\bar{\tau}\bigg\}:=\tilde{m}.
   \end{eqnarray*}
\end{itemize}
\end{lemma}

\section{Permanence}

\setcounter{equation}{0}
{\setlength\arraycolsep{2pt}}
 \indent

In this section, we will give our main results about the permanence of system $\eqref{e13}$. For convenience, we introduce the following notations: for $i=1,2,\ldots,n$, $j=1,2,\ldots,m$,
\begin{eqnarray*}
&&x_{i}^{\vee}=\ln\bigg(\frac{b_{i}^{U}}{a_{ii}^{L}}\exp\bigg\{\frac{b_{i}^{U}\tau^{+}}{1-b_{i}^{U}\bar{\mu}}\bigg\}\bigg),\\
&&y_{j}^\vee=\ln\bigg\{\frac{1}{e_{jj}^{L}}\sum\limits_{l=1}^nd_{jl}^U\exp\{x_l^\vee\}\exp\bigg\{\frac{\eta^{+}}{\bigg(1/\sum\limits_{l=1}^nd_{jl}^U
 \exp\{x_l^\vee\}\bigg)-\bar{\mu}}\bigg\}\bigg\},\\
&&x_{i}^{\wedge}=\ln\bigg\{\frac{r^2}{a_{ii}^U}\bigg(b_i^L-\sum\limits_{l=1,l\neq i}^na_{il}^U\exp\{x_{l}^{\vee}\}-\sum\limits_{h=1}^mc_{ih}^U\exp\{y_{h}^{\vee}\}\bigg)
\exp\bigg\{\frac{\log(1-a_{ii}^Ue^{x_{i}^{\vee}}\bar{\mu})}{\bar{\mu}}\tau^{+}\bigg\}\bigg\},\\
&&y_{j}^{\wedge}=\ln\bigg\{\frac{r^2}{e_{jj}^U}\bigg(\sum\limits_{l=1}^nd_{jl}^L\exp\{x_l^\wedge\}-r_j^U
-\sum\limits_{h=1,h\neq j}^me_{jh}^U\exp\{y_{h}^{\vee}\}\bigg)
\exp\bigg\{\frac{\log(1-e_{jj}^Ue^{y_{j}^{\vee}}\bar{\mu})}{\bar{\mu}}\eta^{+}\bigg\}\bigg\}.
\end{eqnarray*}
\begin{enumerate}
  \item [$(H_{4})$]
  $-b_{i}^{U},-\sum\limits_{l=1}^nd_{jl}^U\exp\{x_l^\vee\}, -a_{ii}^Ue^{x_{i}^{\vee}},-e_{jj}^Ue^{y_{j}^{\vee}}\in\mathcal{R}^{+}$ and for $i=1,2,\ldots,n$, $j=1,2,\ldots,m$,
   \[
 b_{i}^{U}\exp\bigg\{\frac{b_{i}^{U}\tau^{+}}{1-b_{i}^{U}\bar{\mu}}\bigg\}>a_{ii}^{L},\,\,\sum\limits_{l=1}^nd_{jl}^U\exp\{x_l^\vee\}\exp\bigg\{\frac{\eta^{+}}{\bigg(1/\sum\limits_{l=1}^nd_{jl}^U
 \exp\{x_l^\vee\}\bigg)-\bar{\mu}}\bigg\}>e_{jj}^{L},\]
 \[ r^2\bigg(b_i^L-\sum\limits_{l=1,l\neq i}^na_{il}^U\exp\{x_{l}^{\vee}\}-\sum\limits_{h=1}^mc_{ih}^U\exp\{y_{h}^{\vee}\}\bigg)
\exp\bigg\{\frac{\log(1-a_{ii}^Ue^{x_{i}^{\vee}}\bar{\mu})}{\bar{\mu}}\tau^{+}\bigg\}>a_{ii}^U,\]
\[r^2\bigg(\sum\limits_{l=1}^nd_{jl}^L\exp\{x_l^\wedge\}-r_j^U
-\sum\limits_{h=1,h\neq j}^me_{jh}^U\exp\{y_{h}^{\vee}\}\bigg)
\exp\bigg\{\frac{\log(1-e_{jj}^Ue^{y_{j}^{\vee}}\bar{\mu})}{\bar{\mu}}\eta^{+}\bigg\}>e_{jj}^U.\]
\end{enumerate}

\begin{lemma}\label{lem31}
Assume that $(H_1)$--$(H_4)$ hold. Then every solution $(x(t), y(t))=(x_{1}(t), x_{2}(t), \ldots,\\ x_{n}(t),y_{1}(t), y_{2}(t), \ldots,y_{m}(t))$ of system \eqref{e13} satisfies
\begin{eqnarray*}
x_{i}^{\wedge}\leq\liminf_{t\rightarrow+\infty}x_{i}(t)\leq\limsup_{t\rightarrow+\infty}x_{i}(t)\leq x_{i}^{\vee},\,\,i=1,2,\ldots,n,\\
y_{j}^{\wedge}\leq\liminf_{t\rightarrow+\infty}y_{j}(t)\leq\limsup_{t\rightarrow+\infty}y_{j}(t)\leq y_{j}^{\vee},\,\,j=1,2,\ldots,m.
 \end{eqnarray*}
\end{lemma}
\begin{proof}
From the first equation of $\eqref{e13}$, we have
\begin{eqnarray*}
x_i^\Delta(t)\leq b_i^U-a_{ii}^L\exp\{x_i(t-\tau_{ii}(t))\},\,\,i=1,2,\ldots,n.
\end{eqnarray*}
Let $N_{i}(t)=e^{x_{i}(t)}$, obviously $N_{i}(t)>0$, above inequality yields that
\begin{eqnarray*}
[\ln(N_{i}(t))]^{\Delta}\leq b_i^U-a_{ii}^LN_{i}(t-\tau_{ii}(t)),\,\,i=1,2,\ldots,n.
\end{eqnarray*}
In view of Lemma \ref{lemd12}, we have
\begin{eqnarray*}
\frac{N_{i}^{\Delta}(t)}{N_{i}^{\sigma}(t)}\leq [\ln(N_{i}(t))]^{\Delta}\leq b_i^U-a_{ii}^LN_{i}(t-\tau_{ii}(t)),\,\,i=1,2,\ldots,n,
\end{eqnarray*}
 then
\begin{eqnarray*}
N_{i}^{\Delta}(t)\leq N_{i}^{\sigma}(t)[b_i^U-a_{ii}^LN_{i}(t-\tau_{ii}(t))],\,\,i=1,2,\ldots,n.
\end{eqnarray*}
Thus
\begin{eqnarray*}
 \left\{%
\begin{array}{lcrcl}
N_{i}^{\Delta}(t)\leq N_{i}^{\sigma}(t)[b_i^U-a_{ii}^LN_{i}(t-\tau_{ii}(t))],\,\,t\neq t_{k},\,\,t\in J,\\
N_{i}(t_{k}^{+})\leq (1+\lambda_{ik})N_{i}(t_{k}),\,\,t=t_{k},\,\,k\in\mathbb{N},\,\,i=1,2,\ldots,n.
\end{array}\right.
\end{eqnarray*}
In view of Lemma \ref{lem211}, we have
\begin{eqnarray*}
 \limsup_{t\rightarrow+\infty}N_{i}(t)\leq \frac{b_{i}^{U}}{a_{ii}^{L}}\exp\bigg\{\frac{b_{i}^{U}\tau^{+}}{1-b_{i}^{U}\bar{\mu}}\bigg\},\,\,i=1,2,\ldots,n,
   \end{eqnarray*}
that is,
\begin{eqnarray}\label{e31}
\limsup_{t\rightarrow+\infty}x_{i}(t)\leq \ln\bigg(\frac{b_{i}^{U}}{a_{ii}^{L}}\exp\bigg\{\frac{b_{i}^{U}\tau^{+}}{1-b_{i}^{U}\bar{\mu}}\bigg\}\bigg):=x_{i}^{\vee},\,\,i=1,2,\ldots,n.
\end{eqnarray}

For any positive constant $\varepsilon$ small enough, it follows from \eqref{e31} that there exists large enough $T_{1}$ such that
\begin{eqnarray}\label{e32}
x_{i}(t)\leq x_{i}^{\vee}+\varepsilon,\,\,{\rm for}\,\,t\geq T_{1},\,\,i=1,2,\ldots,n.
\end{eqnarray}

According to \eqref{e13} and \eqref{e32}, we have
\begin{eqnarray*}
y_j^\Delta(t)\leq \sum\limits_{l=1}^nd_{jl}^U\exp\{x_l^\vee+\varepsilon\}-e_{jj}^L\exp\{y_j(t-\eta_{jj}(t))\},\,\,j=1,2,\ldots,m.
\end{eqnarray*}
Let $K_{j}(t)=e^{y_{j}(t)}$, obviously $K_{j}(t)>0$, above inequality yields that
\begin{eqnarray*}
[\ln(K_{j}(t))]^{\Delta}\leq \sum\limits_{l=1}^nd_{jl}^U\exp\{x_l^\vee+\varepsilon\}-e_{jj}^LK_{j}(t-\eta_{jj}(t)),\,\,j=1,2,\ldots,m.
\end{eqnarray*}
In view of Lemma \ref{lemd12}, we have
\begin{eqnarray*}
\frac{K_{j}^{\Delta}(t)}{K_{j}^{\sigma}(t)}\leq \sum\limits_{l=1}^nd_{jl}^U\exp\{x_l^\vee+\varepsilon\}-e_{jj}^LK_{j}(t-\eta_{jj}(t)),\,\,j=1,2,\ldots,m,
\end{eqnarray*}
 then
\begin{eqnarray*}
K_{j}^{\Delta}(t)\leq K_{j}^{\sigma}(t)\bigg[\sum\limits_{l=1}^nd_{jl}^U\exp\{x_l^\vee+\varepsilon\}-e_{jj}^LK_{j}(t-\eta_{jj}(t))\bigg],\,\,j=1,2,\ldots,m.
\end{eqnarray*}
Thus
\begin{eqnarray*}
 \left\{%
\begin{array}{lcrcl}
K_{j}^{\Delta}(t)\leq K_{j}^{\sigma}(t)\bigg[\sum\limits_{l=1}^nd_{jl}^U\exp\{x_l^\vee+\varepsilon\}-e_{jj}^LK_{j}(t-\eta_{jj}(t))\bigg],\,\,t\neq t_{k},\,\,t\in J,\\
K_{j}(t_{k}^{+})\leq (1+\lambda_{jk})K_{j}(t_{k}),\,\,t=t_{k},\,\,k\in\mathbb{N},\,\,j=1,2,\ldots,m.
\end{array}\right.
\end{eqnarray*}
By Lemma \ref{lem211}, we have
\begin{eqnarray*}
 \limsup_{t\rightarrow+\infty}K_{j}(t)\leq \frac{1}{e_{jj}^{L}}\sum\limits_{l=1}^nd_{jl}^U\exp\{x_l^\vee+\varepsilon\}\exp\bigg\{\frac{\eta^{+}}{(1/\sum\limits_{l=1}^nd_{jl}^U
 \exp\{x_l^\vee+\varepsilon\})-\bar{\mu}}\bigg\},\,\,j=1,2,\ldots,m,
   \end{eqnarray*}
that is,
\begin{eqnarray*}
\limsup_{t\rightarrow+\infty}y_{j}(t)\leq \ln\bigg\{\frac{1}{e_{jj}^{L}}\sum\limits_{l=1}^nd_{jl}^U\exp\{x_l^\vee+\varepsilon\}\exp\bigg\{\frac{\eta^{+}}{(1/\sum\limits_{l=1}^nd_{jl}^U
 \exp\{x_l^\vee+\varepsilon\})-\bar{\mu}}\bigg\}\bigg\},
\end{eqnarray*}
letting $\varepsilon\rightarrow0$, then for $j=1,2,\ldots,m$,
\begin{eqnarray}\label{e33}
\limsup_{t\rightarrow+\infty}y_{j}(t)\leq \ln\bigg\{\frac{1}{e_{jj}^{L}}\sum\limits_{l=1}^nd_{jl}^U\exp\{x_l^\vee\}\exp\bigg\{\frac{\eta^{+}}{(1/\sum\limits_{l=1}^nd_{jl}^U
 \exp\{x_l^\vee\})-\bar{\mu}}\bigg\}\bigg\}
 =y_{j}^\vee.
\end{eqnarray}

For any positive constant $\varepsilon$ small enough, it follows from \eqref{e33} that there exists large enough $T_{2}$ such that
\begin{eqnarray}\label{e34}
y_{j}(t)\leq y_{j}^{\vee}+\varepsilon,\,\,{\rm for}\,\,t\geq T_{2},\,\,j=1,2,\ldots,m.
\end{eqnarray}
By \eqref{e13}, \eqref{e32} and \eqref{e34}, for $i=1,2,\ldots,n$, we arrive at
\begin{eqnarray*}
x_i^\Delta(t)\geq b_i^L-a_{ii}^U\exp\{x_i(t-\tau_{ii}(t))\}-\sum\limits_{l=1,l\neq i}^na_{il}^U\exp\{x_{l}^{\vee}+\varepsilon\}-\sum\limits_{h=1}^mc_{ih}^U\exp\{y_{h}^{\vee}+\varepsilon\}.
\end{eqnarray*}
Let $\tilde{N}_{i}(t)=e^{x_i(t)}$, obviously $\tilde{N}_{i}(t)>0$, above inequality yields that for $i=1,2,\ldots,n$,
\begin{eqnarray*}
[\ln(\tilde{N}_{i}(t))]^{\Delta}\geq b_i^L-\sum\limits_{l=1,l\neq i}^na_{il}^U\exp\{x_{l}^{\vee}+\varepsilon\}-\sum\limits_{h=1}^mc_{ih}^U\exp\{y_{h}^{\vee}+\varepsilon\}
-a_{ii}^U\tilde{N}_i(t-\tau_{ii}(t)).
\end{eqnarray*}
In view of Lemma \ref{lemd12}, for $i=1,2,\ldots,n$, we have
\begin{eqnarray*}
\frac{\tilde{N}_{i}^{\Delta}(t)}{\tilde{N}_{i}(t)}\geq[\ln(\tilde{N}_{i}(t))]^{\Delta}\geq b_i^L-\sum\limits_{l=1,l\neq i}^na_{il}^U\exp\{x_{l}^{\vee}+\varepsilon\}-\sum\limits_{h=1}^mc_{ih}^U\exp\{y_{h}^{\vee}+\varepsilon\}-a_{ii}^U \tilde{N}_i(t-\tau_{ii}(t)),
\end{eqnarray*}
 then
\begin{eqnarray*}
\tilde{N}_{i}^{\Delta}(t)\geq \tilde{N}_{i}(t)\bigg[b_i^L-\sum\limits_{l=1,l\neq i}^na_{il}^U\exp\{x_{l}^{\vee}+\varepsilon\}-\sum\limits_{h=1}^mc_{ih}^U\exp\{y_{h}^{\vee}+\varepsilon\}-a_{ii}^U \tilde{N}_i(t-\tau_{ii}(t))\bigg].
\end{eqnarray*}
Thus
\begin{eqnarray*}
 \left\{%
\begin{array}{lcrcl}
\tilde{N}_{i}^{\Delta}(t)\geq \tilde{N}_{1}(t)\bigg[b_i^L-\sum\limits_{l=1,l\neq i}^na_{il}^U\exp\{x_{l}^{\vee}+\varepsilon\}-\sum\limits_{h=1}^mc_{ih}^U\exp\{y_{h}^{\vee}+\varepsilon\}\\
\quad\quad\quad\quad\,\,-a_{ii}^U\tilde{N}_i(t-\tau_{ii}(t))\bigg],\,\,t\neq t_{k},\,\,t\in J,\\
\tilde{N}_{i}(t_{k}^{+})\geq (1+\lambda_{ik})\tilde{N}_{i}(t_{k}),\,\,t=t_{k},\,\,k\in\mathbb{N},\,\,i=1,2,\ldots,n.
\end{array}\right.
\end{eqnarray*}
According to Lemma \ref{lem211}, we have
\begin{eqnarray*}
 &&\liminf_{t\rightarrow+\infty}\tilde{N}_{i}(t)\\&\geq& \frac{r^2}{a_{ii}^U}\bigg(b_i^L-\sum\limits_{l=1,l\neq i}^na_{il}^U\exp\{x_{l}^{\vee}+\varepsilon\}-\sum\limits_{h=1}^mc_{ih}^U\exp\{y_{h}^{\vee}+\varepsilon\}\}\bigg)
 \exp\bigg\{\frac{\log(1-a_{ii}^Ue^{x_{i}^{\vee}}\bar{\mu})}{\bar{\mu}}\tau^{+}\bigg\},
   \end{eqnarray*}
that is,
\begin{eqnarray*}
&&\liminf_{t\rightarrow+\infty}x_{i}(t)\\&\geq&\ln\bigg[\frac{r^2}{a_{ii}^U}\bigg(b_i^L-\sum\limits_{l=1,l\neq i}^na_{il}^U\exp\{x_{l}^{\vee}+\varepsilon\}-\sum\limits_{h=1}^mc_{ih}^U\exp\{y_{h}^{\vee}+\varepsilon\}\bigg)
\exp\bigg\{\frac{\log(1-a_{ii}^Ue^{x_{i}^{\vee}}\bar{\mu})}{\bar{\mu}}\tau^{+}\bigg\}\bigg],
\end{eqnarray*}
letting $\varepsilon\rightarrow0$, then we have
\begin{eqnarray}\label{e35}
\liminf_{t\rightarrow+\infty}x_{i}(t) \geq  x_{i}^{\wedge},\,\,i=1,2,\ldots,n.
\end{eqnarray}

For any positive constant $\varepsilon$ small enough, it follows from \eqref{e35} that there exists large enough $T_{3}$ such that
\begin{eqnarray}\label{e36}
x_{i}(t)\geq x_{i}^{\wedge}-\varepsilon,\,\,{\rm for}\,\,t\geq T_{3},\,\,i=1,2,\ldots,n.
\end{eqnarray}
By  \eqref{e13}, \eqref{e35} and \eqref{e36}, we have
\begin{eqnarray*}
y_j^\Delta(t)\geq \sum\limits_{l=1}^nd_{jl}^L\exp\{x_l^\wedge-\varepsilon\}-e_{jj}^U\exp\{y_j(t-\eta_{jj}(t))\}-r_j^U
-\sum\limits_{h=1,h\neq j}^me_{jh}^U\exp\{y_{h}^{\vee}+\varepsilon\},
\end{eqnarray*}
let $\tilde{K}_{j}(t)=e^{y_{j}(t)}$, obviously $\tilde{K}_{j}(t)>0$, above inequality yields that for $j=1,2,\ldots,m$,
\begin{eqnarray*}
[\ln(\tilde{K}_{j}(t))]^{\Delta}\geq \sum\limits_{l=1}^nd_{jl}^L\exp\{x_l^\wedge-\varepsilon\}-r_j^U
-\sum\limits_{h=1,h\neq j}^me_{jh}^U\exp\{y_{h}^{\vee}+\varepsilon\}-e_{jj}^U\tilde{K}_j(t-\eta_{jj}(t)).
\end{eqnarray*}
From Lemma \ref{lemd12}, for $j=1,2,\ldots,m$, we have
\begin{eqnarray*}
\tilde{K}_{j}^{\Delta}(t)\geq\tilde{K}_{j}(t)\bigg[\sum\limits_{l=1}^nd_{jl}^L\exp\{x_l^\wedge-\varepsilon\}-r_j^U
-\sum\limits_{h=1,h\neq j}^me_{jh}^U\exp\{y_{h}^{\vee}+\varepsilon\}-e_{jj}^U\tilde{K}_j(t-\eta_{jj}(t))\bigg].
\end{eqnarray*}
Thus
\begin{eqnarray*}
 \left\{%
\begin{array}{lcrcl}
\tilde{K}_{j}^{\Delta}(t)\geq\tilde{K}_{j}(t)\bigg[\sum\limits_{l=1}^nd_{jl}^L\exp\{x_l^\wedge-\varepsilon\}-r_j^U
-\sum\limits_{h=1,h\neq j}^me_{jh}^U\exp\{y_{h}^{\vee}+\varepsilon\}\\
\quad\quad\quad\quad\,\,-e_{jj}^U\tilde{K}_j(t-\eta_{jj}(t))\bigg],\,\,t\neq t_{k},\,\,t\in J,\\
\tilde{K}_{j}(t_{k}^{+})\geq (1+\lambda_{jk})\tilde{K}_{j}(t_{k}),\,\,t=t_{k},\,\,k\in\mathbb{N},\,\,j=1,2,\ldots,m.
\end{array}\right.
\end{eqnarray*}
In view of Lemma \ref{lem211}, we have
\begin{eqnarray*}
 &&\liminf_{t\rightarrow+\infty}\tilde{K}_{j}(t)\\&\geq& \frac{r^2}{e_{jj}^U}\bigg(\sum\limits_{l=1}^nd_{jl}^L\exp\{x_l^\wedge-\varepsilon\}-r_j^U
-\sum\limits_{h=1,h\neq j}^me_{jh}^U\exp\{y_{h}^{\vee}+\varepsilon\}\}\bigg)
 \exp\bigg\{\frac{\log(1-e_{jj}^Ue^{y_{j}^{\vee}}\bar{\mu})}{\bar{\mu}}\eta^{+}\bigg\},
   \end{eqnarray*}
that is,
\begin{eqnarray*}
&&\liminf_{t\rightarrow+\infty}y_{j}(t)\\&\geq&\ln\bigg[\frac{r^2}{e_{jj}^U}\bigg(\sum\limits_{l=1}^nd_{jl}^L\exp\{x_l^\wedge-\varepsilon\}-r_j^U
-\sum\limits_{h=1,h\neq j}^me_{jh}^U\exp\{y_{h}^{\vee}+\varepsilon\}\bigg)
\exp\bigg\{\frac{\log(1-e_{jj}^Ue^{y_{j}^{\vee}}\bar{\mu})}{\bar{\mu}}\eta^{+}\bigg\}\bigg],
\end{eqnarray*}
letting $\varepsilon\rightarrow0$, then
\begin{eqnarray*}
\liminf_{t\rightarrow+\infty}y_{j}(t)\geq y_{j}^{\wedge},\,\,j=1,2,\ldots,m.
\end{eqnarray*}
The proof of Lemma \ref{lem31} is completed.
\end{proof}

\begin{theorem}\label{thm31}
Assume $(H_1)$--$(H_4)$ hold. Then system $\eqref{e13}$ is permanent.
\end{theorem}

\section{Existence of positive almost periodic solutions}
\setcounter{equation}{0}
{\setlength\arraycolsep{2pt}}
 \indent

In this section, we will study the existence of  positive almost periodic solutions of $\eqref{e13}$. Consider the following equation:
\begin{eqnarray}\label{e41}
 \addtolength{\arraycolsep}{-3pt}
 \left\{%
\begin{array}{lcrcl}
x^{\Delta}(t)=f(t,x(t),x(t-\tau(t))),\,\, t\neq t_{k},\,\,t\in\mathbb{T},\\
\Delta x(t_{k})=I_{k}(x(t_{k})),\,\, t= t_{k}\geq t_0,\,\,k\in\mathbb{N},
\end{array}\right.
\end{eqnarray}
where $t_0\in\mathbb{T}$, $f:\mathbb{T} \times S_{B}\times S_{B}\rightarrow \mathbb{R}$, $S_{B}= \{x\in \mathbb{R} : \|x\|< B\}, \|x\|=\sup_{t\in\mathbb{T}}|x(t)|$. $\tau:\mathbb{T}\rightarrow \mathbb{R}^+$, $\{t_k\}\in \mathcal{B}$, $I_{k}:S_{B}\rightarrow \mathbb{R}$, $k\in\mathbb{N}$.

Denote by $x(t)=x(t;t_0,\varphi_0)$, $\varphi_0\in PC_{rd}[\mathbb{T},S_{B}]$, the solution of system (4.1) with initial conditions:
\begin{eqnarray*}
 \addtolength{\arraycolsep}{-3pt}
 \left\{%
\begin{array}{lcrcl}
x(t)=\varphi_0(t),\,\,t\in(-\infty,t_0]_{\mathbb{T}},\\
x(t_0^+;t_0,\varphi_0)=\varphi_0(t_0).
\end{array}\right.
\end{eqnarray*}

Introduce the following conditions:
\begin{itemize}
    \item  [$(A_1)$] The function $f(t, x, z)$ is almost periodic in $t$ uniformly with respect to $x$ and $z\in S_{B}$, and $f(t,0,0)=0$ for $t\in\mathbb{T}$.
    \item  [$(A_2)$] The sequence of functions $I_{k}(x)$ is almost periodic uniformly with respect to $x\in S_{B}$ and is Lipschitz continuous in $x$, $I_{k}(0)=0$, $k\in\mathbb{N}$.
    \item  [$(A_3)$] The function $\tau(t)$ is almost periodic, $t-\tau(t)\in\mathbb{T}$ for $t\in\mathbb{T}$.
    \item  [$(A_4)$] The set of sequences $\{t_{k}^{j}\}$, $t_{k}^{j}=t_{k+j}-t_{k}$, $k, j\in\mathbb{N}$ is uniformly almost periodic and $\inf_{k}t_{k}^{1}=\theta>0$.
\end{itemize}

Suppose that conditions $(A_1)$--$(A_4)$ hold,  and let $\{\omega'_m\}\subset\Pi$ be an arbitrary
sequence. Then there exists a subsequence $\{\omega_n\}, \omega_n =\omega'_{m_n}$, such that the sequence $\{f(t + \omega_n, x,z)\}$ is convergent uniformly to the
function $f^s(t, x)$, and from Lemma \ref{lemd11} it follows that the set of sequences
$\{t_k-\omega_n\}, k \in \mathbb{Z}$ is convergent to the sequence $t^s_k$ uniformly with
respect to $k\in \mathbb{Z}$ as $n\rightarrow\infty$.
By $\{k_{n_i}\}$ we denote the sequence of integers, such that the subsequence
$\{t_{k_{n_i}}\}$ is convergent to the sequence $t^s_k$ uniformly with respect to $k\in \mathbb{Z}$
as $i\rightarrow\infty$. Then, for every sequence $\{\omega'_m\}$, the system \eqref{e41} ¡°moves¡± to the system

\begin{eqnarray}\label{e4111}
 \left\{%
\begin{array}{lcrcl}
x^{\Delta}(t)=f^s(t,x(t),x(t-\tau^s(t))),\,\,\,\, t\neq t_{k}^s,\\
\Delta x(t_{k}^s)=I_{k}^s(x(t_{k}^s)),\,\,\,\, t_0\leq t= t_{k}^s, k\in \mathbb{N}.
\end{array}\right.
\end{eqnarray}

Define
\begin{eqnarray*}
V_{1}&=&\bigg\{V:\mathbb{T}\times S_{B}\times S_{B}\rightarrow\mathbb{R}^{+}, V\,\, \mathrm{is}\,\, \mathrm{rd-continuous}\,\,\mathrm{in}\,\, (t_{k-1}, t_{k}]_{\mathbb{T}}\times S_{B}\times S_{B}\,\, \mathrm{and} \\ &&\lim\limits_{(t, x, y)\rightarrow(t_{k}, x_{0}, y_{0}), t>t_{k}}V(t, x, y)=V(t_{k}^{+}, x_{0}, y_{0})\bigg\}.
\end{eqnarray*}
\begin{eqnarray*}
W_{1}&=&\bigg\{W:\mathbb{T}\times S_{B}\rightarrow\mathbb{R}^{+}, W \,\, \mathrm{is}\,\, \mathrm{rd-continuous}\,\,\mathrm{in}\,\,(t,x)\in\mathbb{T}\times S_{B},\\&& W(t_{k}^{s-},x)=\lim\limits_{t\rightarrow t_{k}^s; t<t_{k}^s}W(t, x), W(t_{k}^{s+},x)=\lim\limits_{t\rightarrow t_{k}^s; t>t_{k}^s}W(t, x)\,\, \mathrm{and}\,\, \mathrm{the}\,\, \mathrm{equality}\\&& W(t_{k}^{s-},x)=W(t_{k}^{s},x)\,\,\mathrm{holds} \bigg\}.
\end{eqnarray*}
For $V(t, x, y)\in V_{1}$, $W(t, x)\in W_{1}$ and for some $t\geq t_0$, define the following set:
\[\Omega_1=\{x,y\in PC_{rd}[\mathbb{T}^{+}, S_{B}]|V(t, x(s), y(s))\leq V(t, x(t), y(t)), s\in(-\infty,t]_{\mathbb{T}}\},\]
\[\Omega_2=\{x,y\in PC_{rd}[\mathbb{T}^{+}, S_{B}]|W(t, x(s))\leq V(t, x(t)), s\in(t_0,t]_{\mathbb{T}}\}.\]

The proof of the following lemma is similar to the proof of  Lemma 3.6 in \cite{25}.
\begin{lemma}\label{lem51}
Suppose that conditions $(A_1)$--$(A_4)$ hold and there exists a Lyapunov functional $W(t, x)\in W_{1}$ satisfying the following conditions
\begin{itemize}
    \item  [$(i)$]  $a(||x||)\leq W(t, x)\leq b(||x||)$, where $(t, x)\in\mathbb{T}\times S_{B}$, $a, b\in \kappa$, $\kappa=\{a\in C(\mathbb{R}^{+}, \mathbb{R}^{+}): a(0)=0$ and $a$ is increasing\};
    \item  [$(ii)$] $|W(t, x)-W(t, x_{1})|\leq H\|x-x_{1}\|$, where $(t, x)\in\mathbb{T}\times S_{B}$, $H > 0$ is a constant;
    \item  [$(iii)$] $W(t_{k}^{+}, x+I_{k}^s(x))\leq W(t, x)$, $x\in S_{B}$, $t_0\leq t=t_{k}^s$, $k\in\mathbb{N};$
    \item  [$(iv)$] $D^{+}W_{(4.2)}^{\Delta}(t, x)\leq -cV(t, x)$, where $c> 0$, $-c\in\mathcal{R}^{+}$, $x\in \Omega_2$, $t\neq t_{k}^s$, $k\in\mathbb{N}$.
\end{itemize}
Then the zero solution of $\eqref{e4111}$ is uniformly asymptotically stable.
\end{lemma}

\begin{lemma}\label{lem52}
Suppose that conditions $(A_1)$--$(A_4)$ hold and there exists a Lyapunov functional $V(t, x, y)\in V_{1}$ satisfying the following conditions:
\begin{itemize}
    \item  [$(i)$]  $a(||x-y||)\leq V(t, x, y)\leq b(||x -y||)$, where $(t, x, y)\in\mathbb{T}\times S_{B}\times S_{B}$, $a, b\in \kappa$, $\kappa=\{a\in C(\mathbb{R}^{+}, \mathbb{R}^{+}): a(0)=0$ and $a$ is increasing\};
    \item  [$(ii)$] $|V(t, x, y)-V(t, x_{1}, y_{1})|\leq L(\|x-x_{1}\|+ \|y-y_{1}\|)$, where $(t, x, y)\in\mathbb{T}\times S_{B}\times S_{B}$, $L > 0$ is a constant;
    \item  [$(iii)$] $V(t_{k}^{+}, x+I_{k}(x), y+I_{k}(y))\leq V(t, x, y)$, $x, y\in S_{B}$, $t_0\leq t=t_{k}$, $k\in\mathbb{N};$
    \item  [$(iv)$] $D^{+}V_{(4.1)}^{\Delta}(t, x, y)\leq -cV(t, x, y)$, where $c> 0$, $-c\in\mathcal{R}^{+}$, $x, y\in \Omega_1$, $t\neq t_{k}$, $k\in\mathbb{N}$.
\end{itemize}
Moreover, if there exists a solution $x(t;t_0,\varphi_0)\subset S$ of $\eqref{e41}$ for $t_0\leq t\in\mathbb{T}$, where $S\subset S_{B}$ is a compact set, then there exists a unique almost periodic solution $p(t)$ of $\eqref{e41}$, which is uniformly asymptotically stable.
\end{lemma}
\begin{proof}
Let $\{\omega_{m}\}\subset\Pi$ be any sequence  such that $\omega_{m}\rightarrow+\infty$ as $m\rightarrow+\infty$ and $\{\omega_{m}\}$ moves $\eqref{e41}$ to $\eqref{e4111}$.
Let $U\subset \mathbb{T}$ be a compact. Then, for any $\varepsilon>0$, take large enough integer $n_{0}(\varepsilon)$ such that when $m\geq l\geq n_{0}(\varepsilon)$, we have
\begin{eqnarray*}
b(2B)e_{(-c)}(\omega_{l}, t_0)<\frac{a(\varepsilon)}{2}
\end{eqnarray*}
and
\begin{eqnarray*}
\|f(t+\omega_{m},x(t),x(t-\tau(t)))-f(t+\omega_{l},x(t),x(t-\tau(t)))\|<\frac{ca(\varepsilon)}{2L}.
\end{eqnarray*}
Then for $t>t_0$, $(x(t),x(t+\omega_{m}-\omega_{l}))\in\Omega_1$, from $(iv)$, we have
\begin{eqnarray}\label{wli1}
&&D^{+}V^{\Delta}(t, x(t), x(t+\omega_{m}-\omega_{l}))\nonumber\\
&\leq&-cV(t, x(t),x(t+\omega_{m}-\omega_{l}))\nonumber\\
&&+L\|f(t+\omega_{m}-\omega_{l},x(t+\omega_{m}-\omega_{l}),x(t+\omega_{m}-\omega_{l}-\tau(t+\omega_{m}-\omega_{l})))\nonumber\\
&&-f(t,x(t+\omega_{m}-\omega_{l}),x(t+\omega_{m}-\omega_{l}-\tau(t+\omega_{m}-\omega_{l})))\|\nonumber\\
&\leq&-cV(t, x(t), x(t+\omega_{m}-\omega_{l}))+\frac{ca(\varepsilon)}{2}.
\end{eqnarray}
It follows from \eqref{wli1} that
for $m\geq l\geq n_{0}(\varepsilon)$ and $t\in U$, we have
\begin{eqnarray*}
&&V(t+\omega_{l}, x(t+\omega_{l}), x(t+\omega_{m}))\\
&\leq&e_{(-c)}(t+\omega_{l}, t_0)V(t_0, x(t_0), x(t_0+\omega_{m}-\omega_{l}))+\frac{a(\varepsilon)}{2}(1-e_{(-c)}(t+\omega_{l}, t_0))\\
&\leq&e_{(-c)}(t+\omega_{l}, t_0)V(t_0, x(t_0), x(t_0+\omega_{m}-\omega_{l}))+\frac{a(\varepsilon)}{2}
\leq a(\varepsilon).
\end{eqnarray*}
By $(i)$, for $m\geq l\geq n_{0}(\varepsilon)$ and $t\in U$, we obtain
\begin{eqnarray*}
\|x(t+\omega_{m})-x(t+\omega_{l})\|<\varepsilon.
\end{eqnarray*}

Consequently, there exists a function $p(t)$, such that $x(t+\omega_{m})-p(t)\rightarrow\infty$ for $m\rightarrow\infty$, and $\omega(t)$ is defined uniformly on $U$.

Next, we shall show that $p(t)$ is a solution of $\eqref{e4111}$.

Since $x(t;t_0,\varphi_0)\subset S$ is a solution of $\eqref{e41}$, we have
\begin{eqnarray*}
\|x^\Delta(t+\omega_{m})-x^\Delta(t+\omega_{l})\|&\leq&\|f(t+\omega_{m},x(t+\omega_{m}),x(t+\omega_{m}-\tau(t+\omega_{m})))\\
&&-f(t+\omega_{l},x(t+\omega_{m}),x(t+\omega_{m}-\tau(t+\omega_{m})))\|\\
&&+\|f(t+\omega_{l},x(t+\omega_{m}),x(t+\omega_{m}-\tau(t+\omega_{m})))\|\\
&&-f(t,x(t+\omega_{l}),x(t+\omega_{l}-\tau(t+\omega_{l})))\|,
\end{eqnarray*}
for $t+\omega_{i}\neq t_{k}$, $i=m,l$ and $k\in\mathbb{Z}$.

As $x(t+\omega_{m})\subset S_B$, for large $\omega_{m}\in\Pi$ and for each compact subset of $\mathbb{T}$ there exists
$n_{1}(\varepsilon)$ such that when $m\geq l\geq n_{1}(\varepsilon)$, then
\begin{eqnarray*}
&&\|f(t+\omega_{m},x(t+\omega_{m}),x(t+\omega_{m}-\tau(t+\omega_{m})))\\
&&-f(t+\omega_{l},x(t+\omega_{m}),x(t+\omega_{m}-\tau(t+\omega_{m})))\|<\frac{\varepsilon}{2}.
\end{eqnarray*}

Since $x(t+\omega_{i})\subset S_B$, $i=m,l$ and from Lemma \ref{lemd11}, it follows that there exists
$n_{2}(\varepsilon)$ such that if $m\geq l\geq n_{2}(\varepsilon)$, then
\begin{eqnarray*}
&&\|f(t+\omega_{l},x(t+\omega_{m}),x(t+\omega_{m}-\tau(t+\omega_{m})))\\
&&-f(t+\omega_{l},x(t+\omega_{l}),x(t+\omega_{l}-\tau(t+\omega_{l})))\|<\frac{\varepsilon}{2}.
\end{eqnarray*}
Hence for $m\geq l\geq n(\varepsilon)$, $n(\varepsilon)=\max{n_1(\varepsilon),n_2(\varepsilon)}$, we obtain
\begin{eqnarray*}
\|x^\Delta(t+\omega_{m})-x^\Delta(t+\omega_{l})\|<\varepsilon,
\end{eqnarray*}
where $t+\omega_{i}\neq t_k^s$, $i=m,l$ and $k\in\mathbb{Z}$, which shows that $\lim_{m\rightarrow\infty}x^\Delta(t+\omega_{m})$ exists uniformly on every compact subset  of $\mathbb{T}$.

Thus, $\lim_{m\rightarrow\infty}x^\Delta(t+\omega_{m})=p^\Delta(t)$, and
\begin{eqnarray*}
p^\Delta(t)&=&\lim_{m\rightarrow\infty}[f(t+\omega_{m},x(t+\omega_{m}),x(t+\omega_{m}-\tau(t+\omega_{m})))]\\
&=&f^s(t,p(t),p(t-\tau^s(t))),
\end{eqnarray*}
where $t+\omega_{m}\neq t_k^s$.

On the other hand, for $t+\omega_{m}=t_k^s$, it follows that
\begin{eqnarray*}
p(t_k^{s+})-p(t_k^{s-})&=&\lim_{m\rightarrow\infty}[x(t_k^s+\omega_{m}+0)-x(t_k^s+\omega_{m}-0)]\\
&=&\lim_{m\rightarrow\infty}I_{k}^s(x(t_{k}^s+\omega_{m}))=I_{k}^s(p(t_{k}^s)).
\end{eqnarray*}
Then we get that $p(t)$ is a solution of $\eqref{e4111}$.

We have to show that $p(t)$ is an almost periodic function.

Since $\lim_{m\rightarrow\infty}x(t+\omega_{m})=p(t)$ and $x(t+\omega_{m})\subset S$ is bounded, then $p(t)$ is bounded. The function $f^s(t, x, z)$ is almost periodic in $t$ uniformly with respect to $x, z\in S_{B}$, then there exists $M>0$ such that $|f^s(t, x, z)|\leq M$. If the points $t^{'}$ and $t^{''}$ belong to the same interval of continuity and $|t_2-t_1|<\delta=\frac{\varepsilon}{M}$, since $p(t)$ is a solution of $\eqref{e4111}$, we have
\begin{eqnarray*}
\bigg|\int_{t_1}^{t_2}p^{\Delta}(t)\Delta t\bigg|=\bigg|\int_{t_1}^{t_2}f^s(t,x(t),x(t-\tau^s(t)))\Delta t\bigg|,
\end{eqnarray*}
then
$
|p(t_2)-p(t_1)|\leq M|t_2-t_1|,
$
thus
$
|p(t_2)-p(t_1)|<\varepsilon,
$
which shows that $p(t)$ is uniformly continuous in every continuous interval.

On the other hand, for any $\varepsilon>0$, take large enough integer $n_{3}(\varepsilon)$ such that when $m\geq l\geq n_{3}(\varepsilon)$, we have
\begin{eqnarray*}
b(2B)e_{(-c)}(\omega_{l}, (t_0)<\frac{a(\varepsilon)}{4}
\end{eqnarray*}
and
\begin{eqnarray*}
|f(t+\omega_{m},x(t),x(t-\tau(t)))-f(t+\omega_{l},x(t),x(t-\tau(t)))|<\frac{ca(\varepsilon)}{4L}.
\end{eqnarray*}
For each fixed $t\in\mathbb{T}$, let $r_\varepsilon\in\Pi$  be a $\frac{ca(\varepsilon)}{4L}$-translation number of the function
$f$ such that $t+r_\varepsilon\geq 0$.
Now, we consider the function $V(t+r_\varepsilon, p(t), p(t+\omega_{m}-\omega_{l}))$, where $t\leq t+\omega_{m}$.
Then
\begin{eqnarray*}
&&D^{+}V^{\Delta}(t+r_\varepsilon, p(t), p(t+\omega_{m}-\omega_{l}))\\
&\leq&-cV(t+r_\varepsilon, p(t), p(t+\omega_{m}-\omega_{l}))\\
&&+L\|f^s(t,p(t),p(t-\tau^s(t))-f^s(t+r_\varepsilon,p(t),p(t-\tau^s(t))\|\\
&&+L\|f^s(t+\omega_{m}-\omega_{l},p(t+\omega_{m}-\omega_{l}),p(t+\omega_{m}-\omega_{l}-\tau(t+\omega_{m}-\omega_{l})))\\
&&-f^s(t+r_\varepsilon,p(t+\omega_{m}-\omega_{l}),p(t+\omega_{m}-\omega_{l}-\tau(t+\omega_{m}-\omega_{l})))\|\\
&\leq& -cV(t+r_\varepsilon, p(t), p(t+\omega_{m}-\omega_{l}))+\frac{3ca(\varepsilon)}{4}.
\end{eqnarray*}

On the other hand, from $t= t_{k}-(\omega_{m}-\omega_{l})$ and $(iii)$ it follows that
\begin{eqnarray*}
&&V(r_\varepsilon+t_{k}^s, p(t_{k}^s)+I_{k}^s(p(t_{k}^s)), p(t_{k}^s+\omega_{m}-\omega_{l})+I_{k}^s(p(t_{k}^s+\omega_{m}-\omega_{l})))\\
&&\leq V(r_\varepsilon+t_{k}^s, p(t_{k}^s), p(t_{k}^s+\omega_{m}-\omega_{l})),
\end{eqnarray*}
then
\begin{eqnarray*}
&&V(t+r_\varepsilon+\omega_{l}, p(t+\omega_{l}), p(t+\omega_{m}))\\
&\leq&e_{(-c)}(t+r_\varepsilon+\omega_{l}, t_0)V(t_0, p(t_0), p(t_0+\omega_{m}-\omega_{l}))+\frac{3a(\varepsilon)}{4}(1-e_{(-c)}(t+\omega_{l}, t_0))\\
&\leq&e_{(-c)}(t+r_\varepsilon+\omega_{l}, t_0)V(t_0, p(t_0), p(t_0+\omega_{m}-\omega_{l}))+\frac{3a(\varepsilon)}{4}\leq a(\varepsilon).
\end{eqnarray*}
By $(i)$, for $m\geq l\geq n_{3}(\varepsilon)$ and $t\in\mathbb{T}$, we obtain
\begin{eqnarray}\label{liwj1}
\|p(t+\omega_{m})-p(t+\omega_{l})\|<\varepsilon
\end{eqnarray}
and from the definition of the sequence $\{\omega_{m}\}$ for $m\geq l\geq n_{3}(\varepsilon)$, it follows that
\begin{eqnarray}\label{liwj2}
\rho(t_k+\omega_{m},t_k+\omega_{l})<\varepsilon,
\end{eqnarray}
where $\rho(\cdot,\cdot)$ is an arbitrary distance in $\mathcal{B}$.

From \eqref{liwj1} and \eqref{liwj2}, we obtain that the sequence $p(t+\omega_{m})$ is convergent uniformly to the function $p(t)$. Hence, $p(t)$ is almost periodic.

Finally, we will prove that $p(t)$ is uniformly asymptotically stable.

Let $\bar{p}(t)$ be an arbitrary solution of $\eqref{e4111}$. Set
\begin{eqnarray*}
&&u(t)=\bar{p}(t)-p(t),\\
&&g^s(t,u(t))=f^s(t, u(t)+p(t),u(t)+p(t-\tau^s(t)))-f^s(t, p(t),p(t-\tau^s(t))),\\
&&B_k^s(u)=I_k^s(u+p)-I_k^s(u).
\end{eqnarray*}
Now, we consider the system
\begin{eqnarray}\label{e56}
 \left\{%
\begin{array}{lcrcl}
u^{\Delta}(t)=g^s(t,u(t)),\,\,\,\, t\neq t_{k}^s,\\
\Delta u(t_{k}^s)=B_k^s(u(t_{k}^s)),\,\,\,\, t_0\leq t= t_{k}^s
\end{array}\right.
\end{eqnarray}
and let $W(t, u(t)) =V(t, p(t), p(t)+u(t))$.

Then from Lemma \ref{lem51} it follows that the zero solution $u(t)=0$ of system \eqref{e56} is uniformly asymptotically stable, and hence $p(t)$ is uniformly asymptotically stable. The proof is completed.
\end{proof}

Let $(x(t),y(t))=(x_{1}(t), x_{2}(t),\ldots,x_{n}(t),y_{1}(t), y_{2}(t),\ldots,y_{m}(t))$ be any solution of system $\eqref{e13}$, $\Omega= \{(x(t),y(t)): 0 < x_i^{\wedge}\leq x_i(t)\leq  x_i^{\vee}\}, 0 < y_j^{\wedge}\leq y_j(t)\leq  y_j^{\vee}\}$. It is easy to verify that under the conditions of Theorem \ref{thm31}, $\Omega$ is an invariant set of $\eqref{e13}$.

\begin{lemma}
Assume that $(H_{1})$--$(H_{4})$ hold, then $\Omega\neq\emptyset$.
\end{lemma}
\begin{proof}
By the almost periodicity of $b_i(t)$, $a_{il}(t)$, $\tau_{il}(t)$, $c_{ih}(t)$, $\delta_{ih}(t)$, $r_j(t)$, $d_{jl}(t)$, $\xi_{jl}(t)$, $e_{jh}(t)$, $\eta_{jh}(t)$, $i,l=1,2,\ldots,n$, $j,h=1,2,\ldots,m$, there exists a  sequence $\omega=\{\omega_{p}\}\subseteq \Pi$ with $\omega_{p}\rightarrow+\infty$ as $p\rightarrow+\infty$ such that for $t\neq t_{k}$, we have
\begin{eqnarray*}
b_i(t+\omega_{p})\rightarrow b_i(t),\,\,a_{il}(t+\omega_{p})\rightarrow a_{il}(t),\,\,\tau_{il}(t+\omega_{p})\rightarrow \tau_{il}(t),\\
c_{ih}(t+\omega_{p})\rightarrow c_{ih}(t),\,\,\delta_{ih}(t+\omega_{p})\rightarrow \delta_{ih}(t),\,\,r_j(t+\omega_{p})\rightarrow r_j(t),\\
d_{jl}(t+\omega_{p})\rightarrow d_{jl}(t),\,\,\xi_{jl}(t+\omega_{p})\rightarrow \xi_{jl}(t),\,\,e_{jh}(t+\omega_{p})\rightarrow e_{jh}(t),\\
\eta_{jh}(t+\omega_{p})\rightarrow \eta_{jh}(t),\,\,t_{k}-\omega_{p}\rightarrow t_{k},\,\,p\rightarrow+\infty,\\
i,l=1,2,\ldots,n,\,\,j,h=1,2,\ldots,m,
\end{eqnarray*}
and there exists a subsequence $\{k_p\}$ of $\{p\}$, $k_p\rightarrow+\infty$, $p\rightarrow+\infty$, such that $d_{ik_p}\rightarrow d_{ik}$, $d_{jk_p}\rightarrow d_{jk}$.

In view of Lemma 3.1, for all $\varepsilon>0$, there exists a $t_{1}\in\mathbb{T}$ and $t_{1}\geq t_{0}$ such that
\begin{eqnarray*}
 x^{\wedge}-\varepsilon\leq x_i(t)\leq  x^{\vee}+\varepsilon,\,\,y^{\wedge}-\varepsilon\leq y_j(t)\leq  y^{\vee}+\varepsilon,\,\,{\rm for}\ t\geq t_{1}.
\end{eqnarray*}
Write $x_{ip}(t)=x_i(t+\omega_{p})$, $y_{jp}(t)=y_j(t+\omega_{p})$ for $t\geq t_{1}$, $p= 1, 2,\ldots$.
For any positive integer $q$, it is easy to see that there exist sequences $\{x_{ip}(t):p\geq q\}$ and $\{y_{jp}(t):p\geq q\}$ such that the sequences $\{x_{ip} (t)\}$ and $\{y_{jp} (t)\}$ has subsequences, denoted by $\{x_{ip} (t)\}$ and $\{y_{jp} (t)\}$ again, converging on any finite interval of $\mathbb{T}$ as $p\rightarrow+\infty$, respectively. Thus, there are sequences $\{\tilde{x}_{i}(t)\}$ and $\{\tilde{y}_{j}(t)\}$ defined on $\mathbb{T}$ such that
\begin{eqnarray*}
x_{ip}(t)\rightarrow \tilde{x}_{i}(t),\,\,y_{jp}(t)\rightarrow \tilde{y}_{j}(t),\,\,{\rm for}\ t\in\mathbb{T},\,\,{\rm as}\ p\rightarrow+\infty.
\end{eqnarray*}
Since
{\setlength\arraycolsep{2pt}
\begin{eqnarray*}
\left\{
\begin{array}{ll}
x_{ip}^\Delta(t)=b_i(t+\omega_{p})-\sum\limits_{l=1}^na_{il}(t+\omega_{p})\exp\{x_l(t+\omega_{p}-\tau_{il}(t+\omega_{p}))\}\\
\quad\quad\quad\quad-\sum\limits_{h=1}^mc_{ih}(t+\omega_{p})\exp\{y_h(t+\omega_{p}-\delta_{ih}(t+\omega_{p}))\},\,\,t\neq t_k-\omega_{p},\,\,t\in J, \\
y_{jp}^\Delta(t)=-r_j(t+\omega_{p})+\sum\limits_{l=1}^nd_{jl}(t+\omega_{p})\exp\{x_l(t+\omega_{p}-\xi_{jl}(t+\omega_{p}))\}\\
\quad\quad\quad\,\,\,\,-\sum\limits_{h=1}^me_{jh}(t+\omega_{p})\exp\{y_h(t+\omega_{p}-\eta_{jh}(t+\omega_{p}))\},\,\,t\neq t_k-\omega_{p},\,\,t\in J, \\
x_i(t_{k}^{+}-\omega_{p})=x_i(t_{k}-\omega_{p})+\ln(1+\lambda_{ik_p}),\,\,t=t_k-\omega_{p},\\
y_j(t_{k}^{+}-\omega_{p})=y_j(t_{k}-\omega_{p})+\ln(1+\lambda_{jk_p}),\,\,t=t_k-\omega_{p},\,\,i=1,2,\ldots,n,\,\,j=1,2,\ldots,m,
\end{array}
\right.
\end{eqnarray*}}
we have
{\setlength\arraycolsep{2pt}
\begin{eqnarray*}
 \left\{%
\begin{array}{lcrcl}
\tilde{x}_i^\Delta(t)=b_i(t)-\sum\limits_{l=1}^na_{il}(t)\exp\{\tilde{x}_l(t-\tau_{il}(t))\}-\sum\limits_{h=1}^mc_{ih}(t)\exp\{\tilde{y}_h(t-\delta_{ih}(t))\},\,\,t\neq t_k,\,\,t\in J, \\
\tilde{y}_j^\Delta(t)=-r_j(t)+\sum\limits_{l=1}^nd_{jl}(t)\exp\{\tilde{x}_l(t-\xi_{jl}(t))\}\\
\quad\quad\quad\,\,\,\,-\sum\limits_{h=1}^me_{jh}(t)\exp\{\tilde{y}_h(t-\eta_{jh}(t))\},\,\,t\neq t_k,\,\,t\in J, \\
\tilde{x}_i(t_{k}^{+})=\tilde{x}_i(t_{k})+\ln(1+\lambda_{ik}),\,\,t=t_k,\\
\tilde{y}_j(t_{k}^{+})=\tilde{y}_j(t_{k})+\ln(1+\lambda_{jk}),\,\,t=t_k,\,\,i=1,2,\ldots,n,\,\,j=1,2,\ldots,m.
\end{array}\right.
\end{eqnarray*}}
We can easily see that $(\tilde{x}(t),\tilde{y}(t))$ is a solution of system $\eqref{e13}$ and $ x_i^{\wedge}-\varepsilon\leq \tilde{x}_i(t)\leq  x_i^{\vee}+\varepsilon,$ $y_j^{\wedge}-\varepsilon\leq \tilde{y}_j(t)\leq  y_j^{\vee}+\varepsilon$ for $t\in J$, $i=1,2,\ldots,n$, $j=1,2,\ldots,m$. Since $\varepsilon$ is an arbitrary small positive number, it follows that $x_i^{\wedge}\leq \tilde{x}_i(t)\leq  x_i^{\vee}$, $y_j^{\wedge}\leq \tilde{y}_j(t)\leq  y_j^{\vee}$ for $t\in J$. The proof is complete.
\end{proof}

\begin{theorem}\label{thm41}
Assume that $(H_{1})$--$(H_{4})$ hold. Suppose further that
\begin{itemize}
 \item  [$(H_{5})$] $(x,y)\in \Omega^{\ast}$, where
  \begin{eqnarray*}
  \Omega^{\ast}=\bigg\{(x,y)&|&x,y\in PC[\mathbb{T}^{+}, \Omega], \sum\limits_{i=1}^n|x_i(s)|+\sum\limits_{j=1}^m|y_j(s)|\\&&\leq\sum\limits_{i=1}^n|x_i(t)|+\sum\limits_{j=1}^m|y_j(t)|, s\in(-\infty,t], t\geq 0\bigg\};
   \end{eqnarray*}
 \item  [$(H_{6})$]
    $\tau^{\Delta}=\max\limits_{1\leq i,l\leq n}\sup\limits_{t\in\mathbb{T}}\{\tau_{il}^{\Delta}(t)\}$, $\delta^{\Delta}=\max\limits_{1\leq i,h\leq n,}\sup\limits_{t\in\mathbb{T}}\{\delta_{ih}^{\Delta}(t)\}$, $\xi^{\Delta}=\max\limits_{1\leq j,l\leq n}\sup\limits_{t\in\mathbb{T}}\{\xi_{jl}^{\Delta}(t)\}$, $\eta^{\Delta}=\max\limits_{1\leq j,h\leq n,}\sup\limits_{t\in\mathbb{T}}\{\eta_{jh}^{\Delta}(t)\}$, $1-\tau^{\Delta}>0$, $1-\delta^{\Delta}>0$, $1-\xi^{\Delta}>0$, $1-\eta^{\Delta}>0$;
 \item  [$(H_{7})$]  $\gamma>0$ and $-\gamma\in\mathcal{R}^{+}$, where $\gamma=\min_{1\leq i\leq n,1\leq j\leq m}\{\gamma_i,\gamma_j\}$, and
    \begin{eqnarray*}
\gamma_i&=&\sum\limits_{l=1}^na_{li}^{L}e^{x_{i}^{\wedge}}-2\bar{\mu}\bigg(\sum\limits_{l=1}^na_{li}^{U}e^{x_{i}^{\vee}}\bigg)^2
-\frac{1}{1-\tau^{\Delta}}\bigg(2\bar{\mu}\sum\limits_{l=1}^na_{li}^{U}e^{x_{i}^{\vee}}+1\bigg)
\bigg(\sum\limits_{l=1}^na_{li}^{U}e^{x_{i}^{\vee}}\bigg)^{2}\nonumber\\
&&\times(2\tau^{+}-\tau^{-})-\frac{1}{1-\xi^{\Delta}}\bigg(2\bar{\mu}\sum\limits_{l=1}^na_{li}^{U}e^{x_{i}^{\vee}}+1\bigg)
\sum\limits_{j=1}^mc_{ij}^{U}e^{y_{j}^{\vee}}\sum\limits_{i=1}^nd_{ji}^Ue^{x_{i}^{\vee}}(\xi^{+}+\delta^{+}-\xi^{-})\nonumber\\
&&-\sum\limits_{j=1}^md_{ji}^{U}e^{x_{i}^{\vee}}\bigg(2\bar{\mu}\sum\limits_{h=1}^me_{hj}^{U}e^{y_{j}^{\vee}}+1\bigg)
-\frac{1}{1-\xi^{\Delta}}\sum\limits_{j=1}^md_{ji}^{U}e^{x_{i}^{\vee}}\bigg(2\bar{\mu}\sum\limits_{h=1}^me_{hj}^{U}e^{y_{j}^{\vee}}+1\bigg)
\nonumber\\
&&\times\sum\limits_{h=1}^me_{hj}^{U}e^{y_{j}^{\vee}}(\eta^++\xi^+-\xi^-)-\frac{1}
{1-\tau^{\Delta}}\sum\limits_{j=1}^md_{ji}^{U}e^{x_{i}^{\vee}}\nonumber\\
&&\times\bigg(2\bar{\mu}\sum\limits_{h=1}^me_{hj}^{U}e^{y_{j}^{\vee}}+1\bigg)
\sum\limits_{l=1}^na_{li}^{U}e^{x_{i}^{\vee}}(\tau^++\xi^+-\tau^-),\\
\tilde{\gamma}_j&=&\sum\limits_{h=1}^me_{hj}^{L}e^{y_{j}^{\wedge}}-2\bar{\mu}\bigg(\sum\limits_{h=1}^me_{hj}^{U}e^{y_{j}^{\vee}}\bigg)^2
-\frac{1}{1-\eta^{\Delta}}\bigg(2\bar{\mu}\sum\limits_{h=1}^me_{hj}^{U}e^{y_{j}^{\vee}}+1\bigg)
\bigg(\sum\limits_{h=1}^me_{hj}^{U}e^{y_{j}^{\vee}}\bigg)^{2}\nonumber\\
&&\times(2\eta^{+}-\eta^{-})-\frac{1}{1-\delta^{\Delta}}\bigg(2\bar{\mu}\sum\limits_{h=1}^me_{hj}^{U}e^{y_{j}^{\vee}}+1\bigg)
\sum\limits_{i=1}^nd_{ji}^Ue^{x_{i}^{\vee}}\sum\limits_{j=1}^mc_{ij}^{U}e^{y_{j}^{\vee}}(\delta^++\eta^+-\eta^-)\nonumber\\
&&-\sum\limits_{i=1}^nc_{ij}^{U}e^{y_{j}^{\vee}}\bigg(2\bar{\mu}\sum\limits_{l=1}^na_{li}^{U}e^{x_{i}^{\vee}}+1\bigg)
-\frac{1}{1-\delta^{\Delta}}\sum\limits_{i=1}^nc_{ij}^{U}e^{y_{j}^{\vee}}\bigg(2\bar{\mu}\sum\limits_{l=1}^na_{li}^{U}e^{x_{i}^{\vee}}+1\bigg)
\nonumber\\
&&\times\sum\limits_{l=1}^na_{li}^{U}e^{x_{i}^{\vee}}(\tau^{+}+\delta^{+}-\delta^{-})-\frac{1}{1-\eta^{\Delta}}\sum\limits_{i=1}^nc_{ij}^{U}e^{y_{j}^{\vee}}\nonumber\\
&&\times\bigg(2\bar{\mu}
\sum\limits_{l=1}^na_{li}^{U}e^{x_{i}^{\vee}}+1\bigg)\sum\limits_{h=1}^me_{hj}^{U}e^{y_{j}^{\vee}}(\eta^{+}+\delta^{+}-\eta^{-}).
\end{eqnarray*}
\end{itemize}
Then \eqref{e13} has a unique almost periodic solution $(x(t),y(t))$ that is uniformly asymptotically stable.
\end{theorem}
\begin{proof}
According to Lemma \ref{lem31}, every solution $(x(t),y(t))$ of system \eqref{e13} satisfies that $x_i^{\wedge}\leq x_i(t)\leq  x_i^{\vee}$, $ y_j^{\wedge}\leq y_j(t)\leq  y_j^{\vee}$. Hence, $|x_i(t)|\leq K_i$, $|y_j(t)|\leq B_i$ where $K_i=\max\{| x_i^{\vee}|, | x_i^\wedge|\}$, $B_i=\max\{| y_j^{\vee}|, | y_j^\wedge|\}$, $i=1,2,\ldots,n$, $j=1,2,\ldots,m$. Denote $\|(x,y)\|=\sup\limits_{t\in\mathbb{T}}\sum\limits_{i=1}^n|x_i(t)|+\sup\limits_{t\in\mathbb{T}}\sum\limits_{j=1}^m|y_j(t)|$.
 Suppose that
$X_1=(x(t),y(t))$, $X_2=(x^{\ast}(t),y^{\ast}(t))$ are any two positive solutions of system \eqref{e13}, then $\|X_1\|\leq C$ and $\|X_2\|\leq C$, where $C=\sum\limits_{i=1}^nK_i+\sum\limits_{j=1}^mB_i$. In view of system \eqref{e13}, we have
{\setlength\arraycolsep{2pt}\begin{eqnarray}\label{e42}
\left\{
\begin{array}{lll}
x_i^\Delta(t)=b_i(t)-\sum\limits_{l=1}^na_{il}(t)\exp\{x_l(t-\tau_{il}(t))\}-\sum\limits_{h=1}^mc_{ij}(t)\exp\{y_h(t-\delta_{ih}(t))\},t\neq t_k,t\in J, \\
y_j^\Delta(t)=-r_j(t)+\sum\limits_{l=1}^nd_{jl}(t)\exp\{x_l(t-\xi_{jl}(t))\}-\sum\limits_{h=1}^me_{jh}(t)\exp\{y_h(t-\eta_{jh}(t))\},t\neq t_k,t\in J, \\
x_i(t_{k}^{+})=x_i(t_{k})+\ln(1+\lambda_{ik}),t=t_k,\\
y_j(t_{k}^{+})=y_j(t_{k})+\ln(1+\lambda_{jk}),t=t_k,\\
{x_i^\ast}^\Delta(t)=b_i(t)-\sum\limits_{l=1}^na_{il}(t)\exp\{x_l^\ast(t-\tau_{il}(t))\}-\sum\limits_{h=1}^mc_{ih}(t)\exp\{y_h^\ast(t-\delta_{ih}(t))\},t\neq t_k,t\in J, \\
{y_j^\ast}^\Delta(t)=-r_j(t)+\sum\limits_{l=1}^nd_{jl}(t)\exp\{x_l^\ast(t-\xi_{jl}(t))\}-\sum\limits_{h=1}^me_{jh}(t)\exp\{y_h^\ast(t-\eta_{jh}(t))\},t\neq t_k,,t\in J, \\
{x_i^\ast}(t_{k}^{+})=x_i^\ast(t_{k})+\ln(1+\lambda_{ik}),t=t_k,\\
{y_j^\ast}(t_{k}^{+})=y_j^\ast(t_{k})+\ln(1+\lambda_{jk}),t=t_k,i=1,2,\ldots,n,j=1,2,\ldots,m.
\end{array}
\right.
\end{eqnarray}}

Consider the Lyapunov function $V(t, X_1, X_2)$ on $\mathbb{T}^{+}\times\Omega\times\Omega$ defined by
\begin{eqnarray*}
V(t, X_1, X_2)=\sum\limits_{i=1}^n|x_i(t)-x_i^{\ast}(t)|+\sum\limits_{j=1}^m|y_j(t)-y_j^{\ast}(t)|.
\end{eqnarray*}
Let the norm
\[\|X_1-X_2\|=\sup_{t\in\mathbb{T}}\sum\limits_{i=1}^n|x_i(t)-x_i^{\ast}(t)|+\sup_{t\in\mathbb{T}}\sum\limits_{j=1}^m|y_j(t)-y_j^{\ast}(t)|.\]
It is easy to see that
there exist two constants $C_{1}>0$, $C_{2}>0$ such that  $(C_{1}\|X_1-X_2\|)^{2}\leq V(t, X_1, X_2)\leq (C_{2}\|X_1-X_2\|)^{2}$. Let $a, b\in C(\mathbb{R}^{+}, \mathbb{R}^{+})$, $a(x)=C_{1}^{2}x^{2}$, $b(x)=C_{2}^{2}x^{2}$, so Condition $(i)$ of Lemma 4.1 is satisfied. Besides,
\begin{eqnarray*}
&&|V(t, X_1, X_2)-V(t, \tilde{X}_1, \tilde{X}_2)|\\
&=&\bigg|\sum\limits_{i=1}^n|x_i(t)-x_i^{\ast}(t)|+\sum\limits_{j=1}^m|y_j(t)-y_j^{\ast}(t)|
-\sum\limits_{i=1}^n|\tilde{x}_i(t)-\tilde{x}_i^{\ast}(t)|-\sum\limits_{j=1}^m|\tilde{y}_j(t)-\tilde{y}_j^{\ast}(t)|\bigg|\\
&\leq&\bigg|\sum\limits_{i=1}^n|(x_i(t)-\tilde{x}_i(t))-(x_i^{\ast}(t)-\tilde{x}_i^{\ast}(t))|
+\sum\limits_{j=1}^m|(y_j(t)-\tilde{y}_j(t))-(y_j^{\ast}(t)-\tilde{y}_j^{\ast}(t))|\bigg|\\
&\leq&\sum\limits_{i=1}^n|x_i(t)-\tilde{x}_i(t)|+\sum\limits_{j=1}^m|y_j(t)-\tilde{y}_j(t)|
+\sum\limits_{i=1}^n|x_i^{\ast}(t)-\tilde{x}_i^{\ast}(t)|+\sum\limits_{j=1}^m|y_j^{\ast}(t)-\tilde{y}_j^{\ast}(t)|\\
&\leq&L\big(\|X_1-\tilde{X}_1\|+\|X_2-\tilde{X}_2\|\big),
\end{eqnarray*}
where $L=1$, so Condition $(ii)$ of Lemma \ref{lem52} is also satisfied.

On the other hand for $t=t_{k}$, we have
\begin{eqnarray*}
V(t_{k}^{+}, X(t_{k}^{+}), Y(t_{k}^{+}))&=&\sum\limits_{i=1}^n|x_i(t_{k}^{+})-x_i^{\ast}(t_{k}^{+})|+\sum\limits_{j=1}^m|y_j(t_{k}^{+})-y_j^{\ast}(t_{k}^{+})|\\
&=&\sum\limits_{i=1}^n|x_i(t_{k})-x_i^{\ast}(t_{k})|+\sum\limits_{j=1}^m|y_i(t_{k})-y_i^{\ast}(t_{k})|\\
&=&V(t_{k}, X(t_{k}), Y(t_{k})),
\end{eqnarray*}
then Condition $(iii)$ of Lemma \ref{lem52} is also satisfied.

In view of system (4.2), we have
{\setlength\arraycolsep{2pt}
\begin{eqnarray}\label{e43}
\left\{
\begin{array}{lll}
(x_{i}(t)-x_{i}^\ast(t))^{\Delta}=-\sum\limits_{l=1}^na_{il}(t)(\exp\{x_l(t-\tau_{il}(t))\}-\exp\{x_l^\ast(t-\tau_{il}(t))\})\\
\quad\quad\quad\quad\quad\quad\quad\quad\,-\sum\limits_{h=1}^mc_{ih}(t)(\exp\{y_h(t-\delta_{ih}(t))\}-\exp\{y_h^\ast(t-\delta_{ih}(t))\}),t\neq t_{k},t\in J,\\
(y_{j}(t)-y_{j}^\ast(t))^{\Delta}=\sum\limits_{l=1}^nd_{jl}(t)(\exp\{x_l(t-\xi_{jl}(t))\}-\exp\{x_l^\ast(t-\xi_{jl}(t))\})\\
\quad\quad\quad\quad\quad\quad\quad\quad\,-\sum\limits_{h=1}^me_{jh}(t)(\exp\{y_h(t-\eta_{jh}(t))\}-\exp\{y_h^\ast(t-\eta_{jh}(t))\}),t\neq t_{k},t\in J,\\
x_{i}(t_{k}^{+})-x_{i}^\ast(t_{k}^{+})=x_{i}(t_{k})-x_{i}^\ast(t_{k}),t=t_{k},\\
y_{j}(t_{k}^{+})-y_{j}^\ast(t_{k}^{+})=y_{j}(t_{k})-y_{j}^\ast(t_{k}),t=t_{k},i=1,2,\ldots,n,j=1,2,\ldots,m,k\in\mathbb{N}.
\end{array}
\right.
\end{eqnarray}}

Using the mean value theorem, we get
{\setlength\arraycolsep{2pt}
\begin{eqnarray}\label{e44}
\left\{
\begin{array}{lll}
\exp\{x_l(t-\tau_{il}(t))\}-\exp\{x_l^\ast(t-\tau_{il}(t))\}=e^{\zeta_{l}(t)}(x_{l}(t-\tau_{il}(t))-x_{l}^\ast(t-\tau_{il}(t))),\\
\exp\{x_l(t-\xi_{jl}(t))\}-\exp\{x_l^\ast(t-\xi_{jl}(t))\}=e^{\zeta_{l}^\ast(t)}(x_{l}(t-\xi_{jl}(t))-x_{l}^\ast(t-\xi_{jl}(t))),\\
\exp\{y_h(t-\delta_{ih}(t))\}-\exp\{y_h^\ast(t-\delta_{ih}(t))\}=e^{\omega_{h}(t)}(y_{h}(t-\delta_{ih}(t))-y_{h}^\ast(t-\delta_{ih}(t))),\\
\exp\{y_h(t-\eta_{jh}(t))\}-\exp\{y_h^\ast(t-\eta_{jh}(t))\}=e^{\omega_{h}^\ast(t)}(y_{h}(t-\eta_{jh}(t))-y_{h}^\ast(t-\eta_{jh}(t))),
\end{array}
\right.
\end{eqnarray}}
where $\xi_{l}(t)$ lies between $x_{l}(t-\tau_{il}(t))$ and $x_{l}^\ast(t-\tau_{il}(t))$, $\xi_{l}^\ast(t)$ lies between $x_{l}(t-\xi_{jl}(t))$ and $x_{l}^\ast(t-\xi_{jl}(t))$, $\omega_{h}(t)$ lies between $y_h(t-\delta_{ih}(t))$ and $y_h^\ast(t-\delta_{ih}(t))$, $\omega_{h}^\ast(t)$ lies between $y_h(t-\eta_{jh}(t))$ and $y_h^\ast(t-\eta_{jh}(t))$, $i,l=1,2,\ldots,n,$ $j,h=1,2,\ldots,m$.

Let $u_{i}(t)=x_{i}(t)-x_{i}^\ast(t)$, $v_{j}(t)=y_{j}(t)-y_{j}^\ast(t)$, then by \eqref{e44} and \eqref{e43} can be written as
{\setlength\arraycolsep{2pt}
\begin{eqnarray*}
\left\{
\begin{array}{lll}
u_{i}^{\Delta}(t)=-\sum\limits_{l=1}^na_{il}(t)e^{\zeta_{l}(t)}u_{l}(t-\tau_{il}(t))-\sum\limits_{h=1}^mc_{ih}(t)e^{\omega_{h}(t)}v_{h}(t-\delta_{ih}(t)),\,\,t\neq t_{k},\,\,t\in J,\\
v_{j}^{\Delta}(t)=\sum\limits_{l=1}^nd_{jl}(t)e^{\zeta_{l}^\ast(t)}u_{l}(t-\xi_{jl}(t))-\sum\limits_{h=1}^me_{jh}(t)e^{\omega_{h}^\ast(t)}v_{h}(t-\eta_{jh}(t)),\,\,t\neq t_{k},\,\,t\in J,\\
u_{i}(t_{k}^{+})=u_{i}(t_{k}),\,\,t=t_{k},\\
v_{j}(t_{k}^{+})=v_{j}(t_{k}),\,\,t=t_{k},\,\,i=1,2,\ldots,n,\,\,j=1,2,\ldots,m,\,\,k\in\mathbb{N}.
\end{array}
\right.
\end{eqnarray*}}

Consider a Lyapunov function
\[\tilde{V}(t)=\sum_{i=1}^{n}V_{i}(t)+\sum_{j=1}^{m}W_{j}(t),\]
where for $i=1,2,\ldots,n,j=1,2,\ldots,m$,
\[V_{i}(t)=V_{i1}(t)+V_{i2}(t)+V_{i3}(t)+V_{i4}(t)+V_{i5}(t),\]
\[W_{j}(t)=W_{j1}(t)+W_{j2}(t)+W_{j3}(t)+W_{j4}(t)+W_{j5}(t),\]
{\setlength\arraycolsep{2pt}\begin{eqnarray*}
&&V_{i1}(t)=|u_{i}(t)|,\\
&&V_{i2}(t)=\frac{1}{1-\tau^{\Delta}}\bigg(2\bar{\mu}\sum\limits_{l=1}^na_{li}^{U}e^{x_{i}^{\vee}}+1\bigg)
\bigg(\sum\limits_{l=1}^na_{li}^{U}e^{x_{i}^{\vee}}\bigg)^{2}
\int_{-2\tau^+}^{-\tau^-}\int_{s+t}^{t}|u_{i}(r)|\Delta r\Delta s,\\
&&V_{i3}(t)=\frac{1}{1-\delta^{\Delta}}\bigg(2\bar{\mu}\sum\limits_{l=1}^na_{li}^{U}e^{x_{i}^{\vee}}+1\bigg)\sum\limits_{l=1}^na_{li}^{U}e^{x_{i}^{\vee}}
\sum\limits_{j=1}^mc_{ij}^{U}e^{y_{j}^{\vee}}\int_{-\tau^+-\delta^+}^{-\delta^-}\int_{s+t}^{t}|v_{j}(r)|\Delta r\Delta s,\\
&&V_{i4}(t)=\frac{1}{1-\xi^{\Delta}}\bigg(2\bar{\mu}\sum\limits_{l=1}^na_{li}^{U}e^{x_{i}^{\vee}}+1\bigg)\sum\limits_{j=1}^mc_{ij}^{U}e^{y_{j}^{\vee}}
\sum\limits_{i=1}^nd_{ji}^Ue^{x_{i}^{\vee}}\int_{-\delta^+-\xi^+}^{-\xi^-}\int_{s+t}^{t}|u_{i}(r)|\Delta r\Delta s,\\
&&V_{i5}(t)=\frac{1}{1-\eta^{\Delta}}\bigg(2\bar{\mu}\sum\limits_{l=1}^na_{li}^{U}e^{x_{i}^{\vee}}+1\bigg)
\sum\limits_{j=1}^mc_{ij}^{U}e^{y_{j}^{\vee}}\sum\limits_{h=1}^me_{hj}^{U}e^{y_{j}^{\vee}}
\int_{-\delta^+-\eta^+}^{-\eta^-}\int_{s+t}^{t}|v_{j}(r)|\Delta r\Delta s,\\
&&W_{j1}(t)=|v_{j}(t)|,\\
&&W_{j2}(t)=\frac{1}{1-\eta^{\Delta}}\bigg(2\bar{\mu}\sum\limits_{h=1}^me_{hj}^{U}e^{y_{j}^{\vee}}+1\bigg)
\bigg(\sum\limits_{h=1}^me_{hj}^{U}e^{y_{j}^{\vee}}\bigg)^{2}\int_{-2\eta^+}^{-\eta^-}\int_{s+t}^{t}|v_{j}(r)|\Delta r\Delta s,\\
&&W_{j3}(t)=\frac{1}{1-\xi^{\Delta}}\bigg(2\bar{\mu}\sum\limits_{h=1}^me_{hj}^{U}e^{y_{j}^{\vee}}+1\bigg)
\sum\limits_{h=1}^me_{hj}^{U}e^{y_{j}^{\vee}}\sum\limits_{i=1}^nd_{ji}^{U}e^{x_{i}^{\vee}}
\int_{-\eta^+-\xi^+}^{-\xi^-}\int_{s+t}^{t}|u_{i}(r)|\Delta r\Delta s,\\
&&W_{j4}(t)=\frac{1}{1-\tau^{\Delta}}\bigg(2\bar{\mu}\sum\limits_{h=1}^me_{hj}^{U}e^{y_{j}^{\vee}}+1\bigg)
\sum\limits_{i=1}^nd_{ji}^Ue^{x_{i}^{\vee}}\sum\limits_{l=1}^na_{li}^{U}e^{x_{i}^{\vee}}
\int_{-\tau^+-\xi^+}^{-\tau^-}\int_{s+t}^{t}|u_{i}(r)|\Delta r\Delta s,\\
&&W_{j5}(t)=\frac{1}{1-\delta^{\Delta}}\bigg(2\bar{\mu}\sum\limits_{h=1}^me_{hj}^{U}e^{y_{j}^{\vee}}+1\bigg)
\sum\limits_{i=1}^nd_{ji}^Ue^{x_{i}^{\vee}}\sum\limits_{j=1}^mc_{ij}^{U}e^{y_{j}^{\vee}}\int_{-\delta^+-\eta^+}^{-\eta^-}\int_{s+t}^{t}|v_{j}(r)|\Delta r\Delta s.
\end{eqnarray*}}
For $i=1,2,\ldots,n,j=1,2,\ldots,m$, we have
{\setlength\arraycolsep{2pt}
\begin{eqnarray}\label{e45}
&&D^{+}V_{i1}^{\Delta}(t)\nonumber\\
&\leq&
\mathrm{sign}(u_{i}^{\sigma}(t))u_{i}^{\Delta}(t)\nonumber\\
&=&
\mathrm{sign}(u_{i}^{\sigma}(t))\bigg[-\sum\limits_{l=1}^na_{il}(t)e^{\zeta_{l}(t)}u_{l}(t-\tau_{il}(t))
-\sum\limits_{h=1}^mc_{ih}(t)e^{\omega_{h}(t)}v_{h}(t-\delta_{ih}(t))\bigg]\nonumber\\
&=&
-\mathrm{sign}(u_{i}^{\sigma}(t))\sum\limits_{l=1}^na_{li}(t)e^{\zeta_{i}(t)}u_{i}(t-\tau_{li}(t))
-\mathrm{sign}(u_{i}^{\sigma}(t))\sum\limits_{h=1}^mc_{ih}(t)e^{\omega_{h}(t)}v_{h}(t-\delta_{ih}(t))\nonumber\\
&=&
\bigg\{-\mathrm{sign}(u_{i}^{\sigma}(t))\sum\limits_{l=1}^na_{li}(t)e^{\zeta_{i}(t)}[u_{i}(t)+u_{i}(t-\tau_{li}(t))-u_{i}(t)]\nonumber\\
&&-\sum\limits_{h=1}^mc_{ih}(t)e^{\omega_{h}(t)}u_{i}^{\sigma}(t)[v_{h}(t)+v_{h}(t-\delta_{ih}(t))-v_{h}(t)]\bigg\}\nonumber\\
&\leq&
\bigg\{-\mathrm{sign}(u_{i}^{\sigma}(t))\sum\limits_{l=1}^na_{li}(t)e^{\zeta_{i}(t)}[u_{i}^{\sigma}(t)-\mu(t)u_{i}^{\Delta}(t)]
+\sum\limits_{l=1}^na_{li}^{U}e^{x_{i}^{\vee}}\int_{t-\tau_{li}(t)}^{t}|u_{i}^{\Delta}(s)|\Delta s\nonumber\\
&&+{\displaystyle\sum\limits_{h=1}^mc_{ih}^{U}e^{y_{h}^{\vee}}|v_{h}(t)|}
+{\displaystyle\sum\limits_{j=1}^mc_{ij}^{U}e^{y_{j}^{\vee}}\int_{t-\delta_{ij}(t)}^{t}|v_{j}^{\Delta}(s)|\Delta s}\bigg\}\nonumber\\
&\leq&
\bigg\{-\sum\limits_{l=1}^na_{li}^{L}e^{x_{i}^{\wedge}}|u_{i}(t)|+2\bar{\mu} \sum\limits_{l=1}^na_{li}^{U}e^{x_{i}^{\vee}}|u_{i}^{\Delta}(t)|+\sum\limits_{l=1}^na_{li}^{U}e^{x_{i}^{\vee}}
\int_{t-\tau_{li}(t)}^{t}|u_{i}^{\Delta}(s)|\Delta s\nonumber\\
&&+\sum\limits_{h=1}^mc_{ih}^{U}e^{y_{h}^{\vee}}|v_{h}(t)|
+\sum\limits_{j=1}^mc_{ij}^{U}e^{y_{j}^{\vee}}\int_{t-\delta_{ij}(t)}^{t}|v_{j}^{\Delta}(s)|\Delta s\bigg\}\nonumber\\
&\leq&
\bigg\{-\sum\limits_{l=1}^na_{li}^{L}e^{x_{i}^{\wedge}}|u_{i}(t)|+\sum\limits_{l=1}^na_{li}^{U}e^{x_{i}^{\vee}}\int_{t-\tau_{li}(t)}^{t}|u_{i}^{\Delta}(s)|\Delta s+{\displaystyle\sum\limits_{j=1}^mc_{ij}^{U}e^{y_{j}^{\vee}}\int_{t-\delta_{ij}(t)}^{t}|v_{j}^{\Delta}(s)|\Delta s}\nonumber\\
&&+2\bar{\mu}
\sum\limits_{l=1}^na_{li}^{U}e^{x_{i}^{\vee}}\bigg|-\sum\limits_{l=1}^na_{il}(t)e^{\zeta_{l}(t)}u_{l}(t-\tau_{il}(t))\nonumber\\
&&
-\sum\limits_{h=1}^mc_{ih}(t)e^{\omega_{h}(t)}v_{h}(t-\delta_{ih}(t))\bigg|
+{\displaystyle\sum\limits_{j=1}^mc_{ij}^{U}e^{y_{j}^{\vee}}|v_{j}(t)|}\bigg\}\nonumber\\
&\leq&
\bigg\{-\sum\limits_{l=1}^na_{li}^{L}e^{x_{i}^{\wedge}}|u_{i}(t)|+\sum\limits_{l=1}^na_{li}^{U}e^{x_{i}^{\vee}}\int_{t-\tau_{li}(t)}^{t}|u_{i}^{\Delta}(s)|\Delta s+{\displaystyle\sum\limits_{j=1}^mc_{ij}^{U}e^{y_{j}^{\vee}}\int_{t-\delta_{ij}(t)}^{t}|v_{j}^{\Delta}(s)|\Delta s}+2\bar{\mu} \nonumber\\
&&\times\bigg(\sum\limits_{l=1}^na_{li}^{U}e^{x_{i}^{\vee}}\bigg)^2|u_{i}(t-\tau_{li}(t))|
+2\bar{\mu}\sum\limits_{l=1}^na_{li}^{U}e^{x_{i}^{\vee}}\sum\limits_{j=1}^mc_{ij}^{U}e^{y_{j}^{\vee}}|v_{j}(t-\delta_{ij}(t))|
+{\displaystyle\sum\limits_{j=1}^mc_{ij}^{U}e^{y_{j}^{\vee}}|v_{j}(t)|}\bigg\}\nonumber\\
&\leq&
\bigg\{-\sum\limits_{l=1}^na_{li}^{L}e^{x_{i}^{\wedge}}|u_{i}(t)|+\bigg(2\bar{\mu}\sum\limits_{l=1}^na_{li}^{U}e^{x_{i}^{\vee}}+1\bigg)
\sum\limits_{l=1}^na_{li}^{U}e^{x_{i}^{\vee}}\int_{t-\tau_{li}(t)}^{t}|u_{i}^{\Delta}(s)|\Delta s\nonumber\\
&&+{\displaystyle\bigg(2\bar{\mu}\sum\limits_{l=1}^na_{li}^{U}e^{x_{i}^{\vee}}+1\bigg)
\sum\limits_{j=1}^mc_{ij}^{U}e^{y_{j}^{\vee}}\int_{t-\delta_{ij}(t)}^{t}|v_{j}^{\Delta}(s)|\Delta s}+2\bar{\mu}\bigg(\sum\limits_{l=1}^na_{li}^{U}e^{x_{i}^{\vee}}\bigg)^2|u_{i}(t)| \nonumber\\
&&+\bigg(2\bar{\mu}\sum\limits_{l=1}^na_{li}^{U}e^{x_{i}^{\vee}}+1\bigg)\sum\limits_{j=1}^mc_{ij}^{U}e^{y_{j}^{\vee}}|v_{j}(t)|\bigg\}\nonumber\\
&\leq&
\bigg\{-\sum\limits_{l=1}^na_{li}^{L}e^{x_{i}^{\wedge}}|u_{i}(t)|+\bigg(2\bar{\mu}\sum\limits_{l=1}^na_{li}^{U}e^{x_{i}^{\vee}}+1\bigg)
\bigg(\sum\limits_{l=1}^na_{li}^{U}e^{x_{i}^{\vee}}\bigg)^{2}\int_{t-\tau_{li}(t)}^{t}|u_{i}(t-\tau_{li}(t))|\Delta s\nonumber\\
&&+\bigg(2\bar{\mu}\sum\limits_{l=1}^na_{li}^{U}e^{x_{i}^{\vee}}+1\bigg)\sum\limits_{l=1}^na_{li}^{U}e^{x_{i}^{\vee}}
\sum\limits_{j=1}^mc_{ij}^{U}e^{y_{j}^{\vee}}\int_{t-\tau_{li}(t)}^{t}|v_{j}(t-\delta_{ij}(t))|\Delta s+\bigg(2\bar{\mu}\sum\limits_{l=1}^na_{li}^{U}e^{x_{i}^{\vee}}+1\bigg)\nonumber\\
&&\times\sum\limits_{j=1}^mc_{ij}^{U}e^{y_{j}^{\vee}}\sum\limits_{i=1}^nd_{ji}^Ue^{x_{i}^{\vee}}\int_{t-\delta_{ij}(t)}^{t}|u_{i}(t-\xi_{ji}(t))|\Delta s+\bigg(2\bar{\mu}\sum\limits_{l=1}^na_{li}^{U}e^{x_{i}^{\vee}}+1\bigg)
\sum\limits_{j=1}^mc_{ij}^{U}e^{y_{j}^{\vee}}\sum\limits_{h=1}^me_{hj}^{U}e^{y_{j}^{\vee}}\nonumber\\
&&\times\int_{t-\delta_{ij}(t)}^{t}|v_{j}(s-\eta_{hj}(s))|\Delta s+2\bar{\mu}\bigg(\sum\limits_{l=1}^na_{li}^{U}e^{x_{i}^{\vee}}\bigg)^2|u_{i}(t)| \nonumber\\
&&+\bigg(2\bar{\mu}\sum\limits_{l=1}^na_{li}^{U}e^{x_{i}^{\vee}}+1\bigg)\sum\limits_{j=1}^mc_{ij}^{U}e^{y_{j}^{\vee}}|v_{j}(t)|\bigg\}\nonumber\\
&\leq&
\bigg\{-\sum\limits_{l=1}^na_{li}^{L}e^{x_{i}^{\wedge}}|u_{i}(t)|+2\bar{\mu}\bigg(\sum\limits_{l=1}^na_{li}^{U}e^{x_{i}^{\vee}}\bigg)^2|u_{i}(t)| +\bigg(2\bar{\mu}\sum\limits_{l=1}^na_{li}^{U}e^{x_{i}^{\vee}}+1\bigg)\sum\limits_{j=1}^mc_{ij}^{U}e^{y_{j}^{\vee}}|v_{j}(t)|\nonumber\\
&&+\frac{1}{1-\tau^{\Delta}}\bigg(2\bar{\mu}\sum\limits_{l=1}^na_{li}^{U}e^{x_{i}^{\vee}}+1\bigg)
\bigg(\sum\limits_{l=1}^na_{li}^{U}e^{x_{i}^{\vee}}\bigg)^{2}\int_{-2\tau^+}^{-\tau^-}|u_{i}(t+s)|\Delta s\nonumber\\
&&+\frac{1}{1-\delta^{\Delta}}\bigg(2\bar{\mu}\sum\limits_{l=1}^na_{li}^{U}e^{x_{i}^{\vee}}+1\bigg)\sum\limits_{l=1}^na_{li}^{U}e^{x_{i}^{\vee}}
\sum\limits_{j=1}^mc_{ij}^{U}e^{y_{j}^{\vee}}\int_{-\tau^+-\delta^+}^{-\delta^-}|v_{j}(t+s)|\Delta s\nonumber\\
&&+\frac{1}{1-\xi^{\Delta}}\bigg(2\bar{\mu}\sum\limits_{l=1}^na_{li}^{U}e^{x_{i}^{\vee}}+1\bigg)\sum\limits_{j=1}^mc_{ij}^{U}e^{y_{j}^{\vee}}
\sum\limits_{i=1}^nd_{ji}^Ue^{x_{i}^{\vee}}\int_{-\delta^+-\xi^+}^{-\xi^-}|u_{i}(t+s)|\Delta s\nonumber\\
&&+\frac{1}{1-\eta^{\Delta}}\bigg(2\bar{\mu}\sum\limits_{l=1}^na_{li}^{U}e^{x_{i}^{\vee}}+1\bigg)
\sum\limits_{j=1}^mc_{ij}^{U}e^{y_{j}^{\vee}}\sum\limits_{h=1}^me_{hj}^{U}e^{y_{j}^{\vee}}\int_{-\delta^+-\eta^+}^{-\eta^-}|v_{j}(t+s)|\Delta s\bigg\},
\end{eqnarray}}
{\setlength\arraycolsep{2pt}
\begin{eqnarray}
D^{+}V_{i2}^{\Delta}(t)
&=&
\frac{1}{1-\tau^{\Delta}}\bigg(2\bar{\mu}\sum\limits_{l=1}^na_{li}^{U}e^{x_{i}^{\vee}}+1\bigg)
\bigg(\sum\limits_{l=1}^na_{li}^{U}e^{x_{i}^{\vee}}\bigg)^{2}
\int_{-2\tau^{+}}^{-\tau^{-}}[|u_{i}(t)|-|u_{i}(t+s)|]\Delta s\nonumber\\
&=&
\frac{1}{1-\tau^{\Delta}}\bigg(2\bar{\mu}\sum\limits_{l=1}^na_{li}^{U}e^{x_{i}^{\vee}}+1\bigg)
\bigg(\sum\limits_{l=1}^na_{li}^{U}e^{x_{i}^{\vee}}\bigg)^{2}(2\tau^{+}-\tau^{-})|u_{i}(t)|\nonumber\\
&&-\frac{1}{1-\tau^{\Delta}}\bigg(2\bar{\mu}\sum\limits_{l=1}^na_{li}^{U}e^{x_{i}^{\vee}}+1\bigg)
\bigg(\sum\limits_{l=1}^na_{li}^{U}e^{x_{i}^{\vee}}\bigg)^{2}
\int_{-2\tau^{+}}^{-\tau^{-}}|u_{i}(t+s)|\Delta s,\label{e46}\\
D^{+}V_{i3}^{\Delta}(t)
&=&
\frac{1}{1-\delta^{\Delta}}\bigg(2\bar{\mu}\sum\limits_{l=1}^na_{li}^{U}e^{x_{i}^{\vee}}+1\bigg)\sum\limits_{l=1}^na_{li}^{U}e^{x_{i}^{\vee}}
\sum\limits_{j=1}^mc_{ij}^{U}e^{y_{j}^{\vee}}
\int_{-\tau^+-\delta^+}^{-\delta^-}[|v_{j}(t)|-|v_{j}(t+s)|]\Delta s\nonumber\\
&=&
\frac{1}{1-\delta^{\Delta}}\bigg(2\bar{\mu}\sum\limits_{l=1}^na_{li}^{U}e^{x_{i}^{\vee}}+1\bigg)\sum\limits_{l=1}^na_{li}^{U}e^{x_{i}^{\vee}}
\sum\limits_{j=1}^mc_{ij}^{U}e^{y_{j}^{\vee}}(\tau^{+}+\delta^{+}-\delta^{-})|v_{j}(t)|\nonumber\\
&&-\frac{1}{1-\delta^{\Delta}}\bigg(2\bar{\mu}\sum\limits_{l=1}^na_{li}^{U}e^{x_{i}^{\vee}}+1\bigg)\sum\limits_{l=1}^na_{li}^{U}e^{x_{i}^{\vee}}
\sum\limits_{j=1}^mc_{ij}^{U}e^{y_{j}^{\vee}}
\int_{-\tau^+-\delta^+}^{-\delta^-}|v_{j}(t+s)|\Delta s,\label{e47}\\
D^{+}V_{i4}^{\Delta}(t)
&=&
\frac{1}{1-\xi^{\Delta}}\bigg(2\bar{\mu}\sum\limits_{l=1}^na_{li}^{U}e^{x_{i}^{\vee}}+1\bigg)\sum\limits_{j=1}^mc_{ij}^{U}e^{y_{j}^{\vee}}
\sum\limits_{i=1}^nd_{ji}^Ue^{x_{i}^{\vee}}\int_{-\delta^+-\xi^+}^{-\xi^-}[|u_{i}(t)|-|u_{i}(t+s)|]\Delta s\nonumber\\
&=&
\frac{1}{1-\xi^{\Delta}}\bigg(2\bar{\mu}\sum\limits_{l=1}^na_{li}^{U}e^{x_{i}^{\vee}}+1\bigg)\sum\limits_{j=1}^mc_{ij}^{U}e^{y_{j}^{\vee}}
\sum\limits_{i=1}^nd_{ji}^Ue^{x_{i}^{\vee}}(\xi^{+}+\delta^{+}-\xi^{-})|u_{i}(t)|\nonumber\\
&&-\frac{1}{1-\xi^{\Delta}}\bigg(2\bar{\mu}\sum\limits_{l=1}^na_{li}^{U}e^{x_{i}^{\vee}}+1\bigg)\sum\limits_{j=1}^mc_{ij}^{U}e^{y_{j}^{\vee}}
\sum\limits_{i=1}^nd_{ji}^Ue^{x_{i}^{\vee}}\int_{-\delta^+-\xi^+}^{-\xi^-}|u_{i}(t+s)|\Delta s,\label{e48}\\
D^{+}V_{i5}^{\Delta}(t)
&=&
\frac{1}{1-\eta^{\Delta}}\bigg(2\bar{\mu}\sum\limits_{l=1}^na_{li}^{U}e^{x_{i}^{\vee}}+1\bigg)
\sum\limits_{j=1}^mc_{ij}^{U}e^{y_{j}^{\vee}}\sum\limits_{h=1}^me_{hj}^{U}e^{y_{j}^{\vee}}
\int_{-\delta^+-\eta^+}^{-\eta^-}[|v_{j}(t)|-|v_{j}(t+s)|]\Delta s\nonumber\\
&=&
\frac{1}{1-\eta^{\Delta}}\bigg(2\bar{\mu}\sum\limits_{l=1}^na_{li}^{U}e^{x_{i}^{\vee}}+1\bigg)
\sum\limits_{j=1}^mc_{ij}^{U}e^{y_{j}^{\vee}}\sum\limits_{h=1}^me_{hj}^{U}e^{y_{j}^{\vee}}(\eta^{+}+\delta^{+}-\eta^{-})|v_{j}(t)|\nonumber\\
&&-\frac{1}{1-\eta^{\Delta}}\bigg(2\bar{\mu}\sum\limits_{l=1}^na_{li}^{U}e^{x_{i}^{\vee}}+1\bigg)
\sum\limits_{j=1}^mc_{ij}^{U}e^{y_{j}^{\vee}}\sum\limits_{h=1}^me_{hj}^{U}e^{y_{j}^{\vee}}\int_{-\eta^+-\delta^+}^{-\eta^-}|v_{j}(t+s)|\Delta s,\label{e49}
\end{eqnarray}}
{\setlength\arraycolsep{2pt}
\begin{eqnarray}\label{e410}
&&D^{+}W_{j1}^{\Delta}(t)\nonumber\\
&\leq&
\mathrm{sign}(v_{j}^{\sigma}(t))v_{j}^{\Delta}(t)\nonumber\\
&=&
\mathrm{sign}(v_{j}^{\sigma}(t))\bigg[\sum\limits_{l=1}^nd_{jl}(t)e^{\zeta_{l}^\ast(t)}u_{l}(t-\xi_{jl}(t))
-\sum\limits_{h=1}^me_{jh}(t)e^{\omega_{h}^\ast(t)}v_{h}(t-\eta_{jh}(t))\bigg]\nonumber\\
&=&
-\mathrm{sign}(v_{j}^{\sigma}(t))\sum\limits_{h=1}^me_{jh}(t)e^{\omega_{h}^\ast(t)}v_{h}(t-\eta_{jh}(t))
+\mathrm{sign}(v_{j}^{\sigma}(t))\sum\limits_{l=1}^nd_{jl}(t)e^{\zeta_{l}^\ast(t)}u_{l}(t-\xi_{jl}(t))\nonumber\\
&=&
\bigg\{-\mathrm{sign}(v_{j}^{\sigma}(t))\sum\limits_{h=1}^me_{hj}(t)e^{\omega_{j}^\ast(t)}[v_{j}(t)+v_{j}(t-\eta_{hj}(t))-v_{j}(t)]\nonumber\\
&&+\sum\limits_{i=1}^nd_{ji}(t)e^{\zeta_{i}^\ast(t)}v_{j}^{\sigma}(t)[u_{i}(t)+u_{i}(t-\xi_{ji}(t))-u_{i}(t)]\bigg\}\nonumber\\
&\leq&
\bigg\{-\sum\limits_{h=1}^me_{hj}^{L}e^{y_{j}^{\wedge}}|v_{j}(t)|+\sum\limits_{h=1}^me_{hj}^{U}e^{y_{j}^{\vee}}\int_{t-\eta_{hj}(t)}^{t}|v_{j}^{\Delta}(s)|\Delta s+{\displaystyle\sum\limits_{i=1}^nd_{ji}^{U}e^{x_{i}^{\vee}}\int_{t-\xi_{ji}(t)}^{t}|u_{i}^{\Delta}(s)|\Delta s}+2\bar{\mu} \nonumber\\
&&\times\sum\limits_{h=1}^me_{hj}^{U}e^{y_{j}^{\vee}}\bigg|\sum\limits_{l=1}^nd_{jl}(t)e^{\zeta_{l}^\ast(t)}u_{l}(t-\xi_{jl}(t))
-\sum\limits_{h=1}^me_{jh}(t)e^{\omega_{h}^\ast(t)}v_{h}(t-\eta_{jh}(t))\bigg|
+{\displaystyle\sum\limits_{i=1}^nd_{ji}^{U}e^{x_{i}^{\vee}}|u_{i}(t)|}\bigg\}\nonumber\\
&\leq&
\bigg\{-\sum\limits_{h=1}^me_{hj}^{L}e^{y_{j}^{\wedge}}|v_{j}(t)|+\bigg(2\bar{\mu}\sum\limits_{h=1}^me_{hj}^{U}e^{y_{j}^{\vee}}+1\bigg)
\sum\limits_{h=1}^me_{hj}^{U}e^{y_{j}^{\vee}}\int_{t-\eta_{hj}(t)}^{t}|v_{j}^{\Delta}(s)|\Delta s\nonumber\\
&&+{\displaystyle\bigg(2\bar{\mu}\sum\limits_{h=1}^me_{hj}^{U}e^{y_{j}^{\vee}}+1\bigg)
\sum\limits_{i=1}^nd_{ji}^{U}e^{x_{i}^{\vee}}\int_{t-\xi_{ji}(t)}^{t}|u_{i}^{\Delta}(s)|\Delta s}+2\bar{\mu}\bigg(\sum\limits_{h=1}^me_{hj}^{U}e^{y_{j}^{\vee}}\bigg)^2|v_{j}(t)| \nonumber\\
&&+\bigg(2\bar{\mu}\sum\limits_{h=1}^me_{hj}^{U}e^{y_{j}^{\vee}}+1\bigg)
\sum\limits_{i=1}^nd_{ji}^{U}e^{x_{i}^{\vee}}|u_{i}(t)|\bigg\}\nonumber\\
&\leq&
\bigg\{-\sum\limits_{h=1}^me_{hj}^{L}e^{y_{j}^{\wedge}}|v_{j}(t)|+\bigg(2\bar{\mu}\sum\limits_{h=1}^me_{hj}^{U}e^{y_{j}^{\vee}}+1\bigg)
\bigg(\sum\limits_{h=1}^me_{hj}^{U}e^{y_{j}^{\vee}}\bigg)^{2}\int_{t-\eta_{hj}(t)}^{t}|v_{j}(t-\eta_{hj}(t))|\Delta s\nonumber\\
&&+\bigg(2\bar{\mu}\sum\limits_{h=1}^me_{hj}^{U}e^{y_{j}^{\vee}}+1\bigg)
\sum\limits_{h=1}^me_{hj}^{U}e^{y_{j}^{\vee}}\sum\limits_{i=1}^nd_{ji}^{U}e^{x_{i}^{\vee}}\int_{t-\eta_{hj}(t)}^{t}|u_{i}(t-\xi_{ji}(t))|\Delta s+\bigg(2\bar{\mu}\sum\limits_{h=1}^me_{hj}^{U}e^{y_{j}^{\vee}}+1\bigg)\nonumber\\
&&\times\sum\limits_{i=1}^nd_{ji}^Ue^{x_{i}^{\vee}}\sum\limits_{l=1}^na_{li}^{U}e^{x_{i}^{\vee}}\int_{t-\xi_{ji}(t)}^{t}|u_{i}(t-\tau_{li}(t))|\Delta s+\bigg(2\bar{\mu}\sum\limits_{h=1}^me_{hj}^{U}e^{y_{j}^{\vee}}+1\bigg)
\sum\limits_{i=1}^nd_{ji}^Ue^{x_{i}^{\vee}}\sum\limits_{j=1}^mc_{ij}^{U}e^{y_{j}^{\vee}}\nonumber\\
&&\times\int_{t-\xi_{ji}(t)}^{t}|v_{j}(s-\delta_{ij}(s))|\Delta s+2\bar{\mu}\bigg(\sum\limits_{h=1}^me_{hj}^{U}e^{y_{j}^{\vee}}\bigg)^2|v_{j}(t)|\nonumber\\
&& +\bigg(2\bar{\mu}\sum\limits_{h=1}^me_{hj}^{U}e^{y_{j}^{\vee}}+1\bigg)
\sum\limits_{i=1}^nd_{ji}^{U}e^{x_{i}^{\vee}}|u_{i}(t)|\bigg\}\nonumber\\
&\leq&
\bigg\{-\sum\limits_{h=1}^me_{hj}^{L}e^{y_{j}^{\wedge}}|v_{j}(t)|+2\bar{\mu}\bigg(\sum\limits_{h=1}^me_{hj}^{U}e^{y_{j}^{\vee}}\bigg)^2|v_{j}(t)| +\bigg(2\bar{\mu}\sum\limits_{h=1}^me_{hj}^{U}e^{y_{j}^{\vee}}+1\bigg)
\sum\limits_{i=1}^nd_{ji}^{U}e^{x_{i}^{\vee}}|u_{i}(t)|\nonumber\\
&&+\frac{1}{1-\eta^{\Delta}}\bigg(2\bar{\mu}\sum\limits_{h=1}^me_{hj}^{U}e^{y_{j}^{\vee}}+1\bigg)
\bigg(\sum\limits_{h=1}^me_{hj}^{U}e^{y_{j}^{\vee}}\bigg)^{2}\int_{-2\eta^+}^{-\eta^-}|v_{j}(t+s)|\Delta s\nonumber\\
&&+\frac{1}{1-\xi^{\Delta}}\bigg(2\bar{\mu}\sum\limits_{h=1}^me_{hj}^{U}e^{y_{j}^{\vee}}+1\bigg)
\sum\limits_{h=1}^me_{hj}^{U}e^{y_{j}^{\vee}}\sum\limits_{i=1}^nd_{ji}^{U}e^{x_{i}^{\vee}}\int_{-\eta^+-\xi^+}^{-\xi^-}|u_{i}(t+s)|\Delta s\nonumber\\
&&+\frac{1}{1-\tau^{\Delta}}\bigg(2\bar{\mu}\sum\limits_{h=1}^me_{hj}^{U}e^{y_{j}^{\vee}}+1\bigg)
\sum\limits_{i=1}^nd_{ji}^Ue^{x_{i}^{\vee}}\sum\limits_{l=1}^na_{li}^{U}e^{x_{i}^{\vee}}\int_{-\tau^+-\xi^+}^{-\tau^-}|u_{i}(t+s)|\Delta s\nonumber\\
&&+\frac{1}{1-\delta^{\Delta}}\bigg(2\bar{\mu}\sum\limits_{h=1}^me_{hj}^{U}e^{y_{j}^{\vee}}+1\bigg)
\sum\limits_{i=1}^nd_{ji}^Ue^{x_{i}^{\vee}}\sum\limits_{j=1}^mc_{ij}^{U}e^{y_{j}^{\vee}}\int_{-\delta^+-\eta^+}^{-\eta^-}|v_{j}(t+s)|\Delta s\bigg\},
\end{eqnarray}}
{\setlength\arraycolsep{2pt}
\begin{eqnarray}
D^{+}W_{j2}^{\Delta}(t)
&=&
\frac{1}{1-\eta^{\Delta}}\bigg(2\bar{\mu}\sum\limits_{h=1}^me_{hj}^{U}e^{y_{j}^{\vee}}+1\bigg)
\bigg(\sum\limits_{h=1}^me_{hj}^{U}e^{y_{j}^{\vee}}\bigg)^{2}\int_{-2\eta^+}^{-\eta^-}[|v_{j}(t)|-|v_{j}(t+s)|]\Delta s\nonumber\\
&=&
\frac{1}{1-\eta^{\Delta}}\bigg(2\bar{\mu}\sum\limits_{h=1}^me_{hj}^{U}e^{y_{j}^{\vee}}+1\bigg)
\bigg(\sum\limits_{h=1}^me_{hj}^{U}e^{y_{j}^{\vee}}\bigg)^{2}(2\eta^{+}-\eta^{-})|v_{j}(t)|\nonumber\\
&&-\frac{1}{1-\eta^{\Delta}}\bigg(2\bar{\mu}\sum\limits_{h=1}^me_{hj}^{U}e^{y_{j}^{\vee}}+1\bigg)
\bigg(\sum\limits_{h=1}^me_{hj}^{U}e^{y_{j}^{\vee}}\bigg)^{2}\int_{-2\eta^+}^{-\eta^-}|v_{j}(t+s)|\Delta s,\label{e411}\\
D^{+}W_{j3}^{\Delta}(t)
&=&
\frac{1}{1-\xi^{\Delta}}\bigg(2\bar{\mu}\sum\limits_{h=1}^me_{hj}^{U}e^{y_{j}^{\vee}}+1\bigg)
\sum\limits_{h=1}^me_{hj}^{U}e^{y_{j}^{\vee}}\sum\limits_{i=1}^nd_{ji}^{U}e^{x_{i}^{\vee}}\int_{-\eta^+-\xi^+}^{-\xi^-}[|u_{i}(t)|-|u_{i}(t+s)|]\Delta s\nonumber\\
&=&
\frac{1}{1-\xi^{\Delta}}\bigg(2\bar{\mu}\sum\limits_{h=1}^me_{hj}^{U}e^{y_{j}^{\vee}}+1\bigg)
\sum\limits_{h=1}^me_{hj}^{U}e^{y_{j}^{\vee}}\sum\limits_{i=1}^nd_{ji}^{U}e^{x_{i}^{\vee}}(\eta^++\xi^+-\xi^-)|u_{i}(t)|\nonumber\\
&&-\frac{1}{1-\xi^{\Delta}}\bigg(2\bar{\mu}\sum\limits_{h=1}^me_{hj}^{U}e^{y_{j}^{\vee}}+1\bigg)
\sum\limits_{h=1}^me_{hj}^{U}e^{y_{j}^{\vee}}\sum\limits_{i=1}^nd_{ji}^{U}e^{x_{i}^{\vee}}\int_{-\eta^+-\xi^+}^{-\xi^-}|u_{i}(t+s)|\Delta s,\label{e412}\\
D^{+}W_{j4}^{\Delta}(t)
&=&
\frac{1}{1-\tau^{\Delta}}\bigg(2\bar{\mu}\sum\limits_{h=1}^me_{hj}^{U}e^{y_{j}^{\vee}}+1\bigg)
\sum\limits_{i=1}^nd_{ji}^Ue^{x_{i}^{\vee}}\sum\limits_{l=1}^na_{li}^{U}e^{x_{i}^{\vee}}\int_{-\tau^+-\xi^+}^{-\tau^-}[|u_{i}(t)|-|u_{i}(t+s)|]\Delta s\nonumber\\
&=&
\frac{1}{1-\tau^{\Delta}}\bigg(2\bar{\mu}\sum\limits_{h=1}^me_{hj}^{U}e^{y_{j}^{\vee}}+1\bigg)
\sum\limits_{i=1}^nd_{ji}^Ue^{x_{i}^{\vee}}\sum\limits_{l=1}^na_{li}^{U}e^{x_{i}^{\vee}}(\tau^++\xi^+-\tau^-)|u_{i}(t)|\nonumber\\
&&-\frac{1}{1-\tau^{\Delta}}\bigg(2\bar{\mu}\sum\limits_{h=1}^me_{hj}^{U}e^{y_{j}^{\vee}}+1\bigg)
\sum\limits_{i=1}^nd_{ji}^Ue^{x_{i}^{\vee}}\sum\limits_{l=1}^na_{li}^{U}e^{x_{i}^{\vee}}\int_{-\tau^+-\xi^+}^{-\tau^-}|u_{i}(t+s)|\Delta s,\label{e413}\\
D^{+}W_{j5}^{\Delta}(t)
&=&
\frac{1}{1-\delta^{\Delta}}\bigg(2\bar{\mu}\sum\limits_{h=1}^me_{hj}^{U}e^{y_{j}^{\vee}}+1\bigg)
\sum\limits_{i=1}^nd_{ji}^Ue^{x_{i}^{\vee}}\sum\limits_{j=1}^mc_{ij}^{U}e^{y_{j}^{\vee}}\int_{-\delta^+-\eta^+}^{-\eta^-}[|v_{j}(t)|-|v_{j}(t+s)|]\Delta s\nonumber\\
&=&
\frac{1}{1-\delta^{\Delta}}\bigg(2\bar{\mu}\sum\limits_{h=1}^me_{hj}^{U}e^{y_{j}^{\vee}}+1\bigg)
\sum\limits_{i=1}^nd_{ji}^Ue^{x_{i}^{\vee}}\sum\limits_{j=1}^mc_{ij}^{U}e^{y_{j}^{\vee}}(\delta^++\eta^+-\eta^-)|v_{j}(t)|\nonumber\\
&&-\frac{1}{1-\delta^{\Delta}}\bigg(2\bar{\mu}\sum\limits_{h=1}^me_{hj}^{U}e^{y_{j}^{\vee}}+1\bigg)
\sum\limits_{i=1}^nd_{ji}^Ue^{x_{i}^{\vee}}\sum\limits_{j=1}^mc_{ij}^{U}e^{y_{j}^{\vee}}\int_{-\delta^+-\eta^+}^{-\eta^-}|v_{j}(t+s)|\Delta s.\label{e414}
\end{eqnarray}}

In view of $\eqref{e45}$--$\eqref{e49}$, we can obtain for  $i=1,2,\ldots,n$,
{\setlength\arraycolsep{2pt}\begin{eqnarray*}
&&D^{+}V_{i}^{\Delta}(t)\nonumber\\
&=&
D^{+}V_{i1}^{\Delta}(t)+D^{+}V_{i2}^{\Delta}(t)+D^{+}V_{i3}^{\Delta}(t)+D^{+}V_{i4}^{\Delta}(t)+D^{+}V_{i5}^{\Delta}(t)\nonumber\\
&\leq&
-\sum\limits_{l=1}^na_{li}^{L}e^{x_{i}^{\wedge}}|u_{i}(t)|+2\bar{\mu}\bigg(\sum\limits_{l=1}^na_{li}^{U}e^{x_{i}^{\vee}}\bigg)^2|u_{i}(t)| +\bigg(2\bar{\mu}\sum\limits_{l=1}^na_{li}^{U}e^{x_{i}^{\vee}}+1\bigg)\sum\limits_{j=1}^mc_{ij}^{U}e^{y_{j}^{\vee}}|v_{j}(t)|\nonumber\\
&&+\frac{1}{1-\tau^{\Delta}}\bigg(2\bar{\mu}\sum\limits_{l=1}^na_{li}^{U}e^{x_{i}^{\vee}}+1\bigg)
\bigg(\sum\limits_{l=1}^na_{li}^{U}e^{x_{i}^{\vee}}\bigg)^{2}(2\tau^{+}-\tau^{-})|u_{i}(t)|\nonumber\\
&&+\frac{1}{1-\delta^{\Delta}}\bigg(2\bar{\mu}\sum\limits_{l=1}^na_{li}^{U}e^{x_{i}^{\vee}}+1\bigg)\sum\limits_{l=1}^na_{li}^{U}e^{x_{i}^{\vee}}
\sum\limits_{j=1}^mc_{ij}^{U}e^{y_{j}^{\vee}}(\tau^{+}+\delta^{+}-\delta^{-})|v_{j}(t)|\nonumber\\
&&+\frac{1}{1-\xi^{\Delta}}\bigg(2\bar{\mu}\sum\limits_{l=1}^na_{li}^{U}e^{x_{i}^{\vee}}+1\bigg)\sum\limits_{j=1}^mc_{ij}^{U}e^{y_{j}^{\vee}}
\sum\limits_{i=1}^nd_{ji}^Ue^{x_{i}^{\vee}}(\xi^{+}+\delta^{+}-\xi^{-})|u_{i}(t)|\nonumber\\
&&+\frac{1}{1-\eta^{\Delta}}\bigg(2\bar{\mu}\sum\limits_{l=1}^na_{li}^{U}e^{x_{i}^{\vee}}+1\bigg)
\sum\limits_{j=1}^mc_{ij}^{U}e^{y_{j}^{\vee}}\sum\limits_{h=1}^me_{hj}^{U}e^{y_{j}^{\vee}}(\eta^{+}+\delta^{+}-\eta^{-})|v_{j}(t)|\nonumber\\
&=&
-\bigg\{\sum\limits_{l=1}^na_{li}^{L}e^{x_{i}^{\wedge}}-2\bar{\mu}\bigg(\sum\limits_{l=1}^na_{li}^{U}e^{x_{i}^{\vee}}\bigg)^2-\frac{1}{1-\tau^{\Delta}}
\bigg(2\bar{\mu}\sum\limits_{l=1}^na_{li}^{U}e^{x_{i}^{\vee}}+1\bigg)
\bigg(\sum\limits_{l=1}^na_{li}^{U}e^{x_{i}^{\vee}}\bigg)^{2}(2\tau^{+}-\tau^{-})
\nonumber\\&&-\frac{1}{1-\xi^{\Delta}}\bigg(2\bar{\mu}\sum\limits_{l=1}^na_{li}^{U}e^{x_{i}^{\vee}}+1\bigg)\sum\limits_{j=1}^mc_{ij}^{U}e^{y_{j}^{\vee}}
\sum\limits_{i=1}^nd_{ji}^Ue^{x_{i}^{\vee}}(\xi^{+}+\delta^{+}-\xi^{-})\bigg\}|u_{i}(t)|\nonumber\\
&&+\bigg\{\bigg(2\bar{\mu}\sum\limits_{l=1}^na_{li}^{U}e^{x_{i}^{\vee}}+1\bigg)\sum\limits_{j=1}^mc_{ij}^{U}e^{y_{j}^{\vee}}
+\frac{1}{1-\delta^{\Delta}}\bigg(2\bar{\mu}\sum\limits_{l=1}^na_{li}^{U}e^{x_{i}^{\vee}}+1\bigg)\sum\limits_{l=1}^na_{li}^{U}e^{x_{i}^{\vee}}
\sum\limits_{j=1}^mc_{ij}^{U}e^{y_{j}^{\vee}}(\tau^{+}\nonumber\\
&&+\delta^{+}-\delta^{-})+\frac{1}{1-\eta^{\Delta}}\bigg(2\bar{\mu}\sum\limits_{l=1}^na_{li}^{U}e^{x_{i}^{\vee}}+1\bigg)
\sum\limits_{j=1}^mc_{ij}^{U}e^{y_{j}^{\vee}}\sum\limits_{h=1}^me_{hj}^{U}e^{y_{j}^{\vee}}(\eta^{+}+\delta^{+}-\eta^{-})\bigg\}|v_{j}(t)|.
\end{eqnarray*}}

In view of $\eqref{e410}$--$\eqref{e414}$, we can obtain $j=1,2,\ldots,m$,
{\setlength\arraycolsep{2pt}\begin{eqnarray*}
&&D^{+}W_{j}^{\Delta}(t)\nonumber\\
&=&
D^{+}W_{j1}^{\Delta}(t)+D^{+}W_{j2}^{\Delta}(t)+D^{+}W_{j3}^{\Delta}(t)+D^{+}W_{j4}^{\Delta}(t)+D^{+}W_{j5}^{\Delta}(t)\nonumber\\
&\leq&
-\sum\limits_{h=1}^me_{hj}^{L}e^{y_{j}^{\wedge}}|v_{j}(t)|+2\bar{\mu}\bigg(\sum\limits_{h=1}^me_{hj}^{U}e^{y_{j}^{\vee}}\bigg)^2|v_{j}(t)| +\bigg(2\bar{\mu}\sum\limits_{h=1}^me_{hj}^{U}e^{y_{j}^{\vee}}+1\bigg)
\sum\limits_{i=1}^nd_{ji}^{U}e^{x_{i}^{\vee}}|u_{i}(t)|\nonumber\\
&&+\frac{1}{1-\eta^{\Delta}}\bigg(2\bar{\mu}\sum\limits_{h=1}^me_{hj}^{U}e^{y_{j}^{\vee}}+1\bigg)
\bigg(\sum\limits_{h=1}^me_{hj}^{U}e^{y_{j}^{\vee}}\bigg)^{2}(2\eta^{+}-\eta^{-})|v_{j}(t)|\nonumber\\
&&+\frac{1}{1-\xi^{\Delta}}\bigg(2\bar{\mu}\sum\limits_{h=1}^me_{hj}^{U}e^{y_{j}^{\vee}}+1\bigg)
\sum\limits_{h=1}^me_{hj}^{U}e^{y_{j}^{\vee}}\sum\limits_{i=1}^nd_{ji}^{U}e^{x_{i}^{\vee}}(\eta^++\xi^+-\xi^-)|u_{i}(t)|\nonumber\\
&&+\frac{1}{1-\tau^{\Delta}}\bigg(2\bar{\mu}\sum\limits_{h=1}^me_{hj}^{U}e^{y_{j}^{\vee}}+1\bigg)
\sum\limits_{i=1}^nd_{ji}^Ue^{x_{i}^{\vee}}\sum\limits_{l=1}^na_{li}^{U}e^{x_{i}^{\vee}}(\tau^++\xi^+-\tau^-)|u_{i}(t)|\nonumber\\
&&+\frac{1}{1-\delta^{\Delta}}\bigg(2\bar{\mu}\sum\limits_{h=1}^me_{hj}^{U}e^{y_{j}^{\vee}}+1\bigg)
\sum\limits_{i=1}^nd_{ji}^Ue^{x_{i}^{\vee}}\sum\limits_{j=1}^mc_{ij}^{U}e^{y_{j}^{\vee}}(\delta^++\eta^+-\eta^-)|v_{j}(t)|\nonumber\\
&=&
-\bigg\{\sum\limits_{h=1}^me_{hj}^{L}e^{y_{j}^{\wedge}}-2\bar{\mu}\bigg(\sum\limits_{h=1}^me_{hj}^{U}e^{y_{j}^{\vee}}\bigg)^2-\frac{1}{1-\eta^{\Delta}}\bigg(2\bar{\mu}\sum\limits_{h=1}^me_{hj}^{U}e^{y_{j}^{\vee}}+1\bigg)
\bigg(\sum\limits_{h=1}^me_{hj}^{U}e^{y_{j}^{\vee}}\bigg)^{2}(2\eta^{+}-\eta^{-})\nonumber\\
&&-\frac{1}{1-\delta^{\Delta}}\bigg(2\bar{\mu}\sum\limits_{h=1}^me_{hj}^{U}e^{y_{j}^{\vee}}+1\bigg)
\sum\limits_{i=1}^nd_{ji}^Ue^{x_{i}^{\vee}}\sum\limits_{j=1}^mc_{ij}^{U}e^{y_{j}^{\vee}}(\delta^++\eta^+-\eta^-)\bigg\}|v_{j}(t)|\nonumber\\
&&+\bigg\{\bigg(2\bar{\mu}\sum\limits_{h=1}^me_{hj}^{U}e^{y_{j}^{\vee}}+1\bigg)\sum\limits_{i=1}^nd_{ji}^{U}e^{x_{i}^{\vee}}
+\frac{1}{1-\xi^{\Delta}}\bigg(2\bar{\mu}\sum\limits_{h=1}^me_{hj}^{U}e^{y_{j}^{\vee}}+1\bigg)
\sum\limits_{h=1}^me_{hj}^{U}e^{y_{j}^{\vee}}\sum\limits_{i=1}^nd_{ji}^{U}e^{x_{i}^{\vee}}(\eta^+\nonumber\\&&+\xi^+-\xi^-)+\frac{1}{1-\tau^{\Delta}}\bigg(2\bar{\mu}\sum\limits_{h=1}^me_{hj}^{U}e^{y_{j}^{\vee}}+1\bigg)
\sum\limits_{i=1}^nd_{ji}^Ue^{x_{i}^{\vee}}\sum\limits_{l=1}^na_{li}^{U}e^{x_{i}^{\vee}}(\tau^++\xi^+-\tau^-)\bigg\}|u_{i}(t)|.
\end{eqnarray*}}

Then for $t\neq t_{k}$, we have
{\setlength\arraycolsep{2pt}\begin{eqnarray*}
&&D^{+}V^{\Delta}(t)\leq D^{+}\tilde{V}^{\Delta}(t)\\
&\leq&
-\sum\limits_{i=1}^n\bigg\{\sum\limits_{l=1}^na_{li}^{L}e^{x_{i}^{\wedge}}-2\bar{\mu}\bigg(\sum\limits_{l=1}^na_{li}^{U}e^{x_{i}^{\vee}}\bigg)^2
-\frac{1}{1-\tau^{\Delta}}\bigg(2\bar{\mu}\sum\limits_{l=1}^na_{li}^{U}e^{x_{i}^{\vee}}+1\bigg)
\bigg(\sum\limits_{l=1}^na_{li}^{U}e^{x_{i}^{\vee}}\bigg)^{2}\nonumber\\
&&\times(2\tau^{+}-\tau^{-})-\frac{1}{1-\xi^{\Delta}}\bigg(2\bar{\mu}\sum\limits_{l=1}^na_{li}^{U}e^{x_{i}^{\vee}}+1\bigg)
\sum\limits_{j=1}^mc_{ij}^{U}e^{y_{j}^{\vee}}\sum\limits_{i=1}^nd_{ji}^Ue^{x_{i}^{\vee}}(\xi^{+}+\delta^{+}-\xi^{-})\bigg\}|u_{i}(t)|\nonumber\\
&&+\sum\limits_{i=1}^n\bigg\{\bigg(2\bar{\mu}\sum\limits_{l=1}^na_{li}^{U}e^{x_{i}^{\vee}}+1\bigg)\sum\limits_{j=1}^mc_{ij}^{U}e^{y_{j}^{\vee}}
+\frac{1}{1-\delta^{\Delta}}\bigg(2\bar{\mu}\sum\limits_{l=1}^na_{li}^{U}e^{x_{i}^{\vee}}+1\bigg)\sum\limits_{l=1}^na_{li}^{U}e^{x_{i}^{\vee}}
\sum\limits_{j=1}^mc_{ij}^{U}e^{y_{j}^{\vee}}\nonumber\\
&&\times(\tau^{+}+\delta^{+}-\delta^{-})+\frac{1}{1-\eta^{\Delta}}\bigg(2\bar{\mu}\sum\limits_{l=1}^na_{li}^{U}e^{x_{i}^{\vee}}+1\bigg)
\sum\limits_{j=1}^mc_{ij}^{U}e^{y_{j}^{\vee}}\sum\limits_{h=1}^me_{hj}^{U}e^{y_{j}^{\vee}}(\eta^{+}+\delta^{+}-\eta^{-})\bigg\}|v_{j}(t)|\nonumber\\
&&-\sum\limits_{j=1}^m\bigg\{\sum\limits_{h=1}^me_{hj}^{L}e^{y_{j}^{\wedge}}-2\bar{\mu}\bigg(\sum\limits_{h=1}^me_{hj}^{U}e^{y_{j}^{\vee}}\bigg)^2
-\frac{1}{1-\eta^{\Delta}}\bigg(2\bar{\mu}\sum\limits_{h=1}^me_{hj}^{U}e^{y_{j}^{\vee}}+1\bigg)
\bigg(\sum\limits_{h=1}^me_{hj}^{U}e^{y_{j}^{\vee}}\bigg)^{2}\nonumber\\
&&\times(2\eta^{+}-\eta^{-})-\frac{1}{1-\delta^{\Delta}}\bigg(2\bar{\mu}\sum\limits_{h=1}^me_{hj}^{U}e^{y_{j}^{\vee}}+1\bigg)
\sum\limits_{i=1}^nd_{ji}^Ue^{x_{i}^{\vee}}\sum\limits_{j=1}^mc_{ij}^{U}e^{y_{j}^{\vee}}(\delta^++\eta^+-\eta^-)\bigg\}|v_{j}(t)|\nonumber\\
&&+\sum\limits_{j=1}^m\bigg\{\bigg(2\bar{\mu}\sum\limits_{h=1}^me_{hj}^{U}e^{y_{j}^{\vee}}+1\bigg)\sum\limits_{i=1}^nd_{ji}^{U}e^{x_{i}^{\vee}}
+\frac{1}{1-\xi^{\Delta}}\bigg(2\bar{\mu}\sum\limits_{h=1}^me_{hj}^{U}e^{y_{j}^{\vee}}+1\bigg)
\sum\limits_{h=1}^me_{hj}^{U}e^{y_{j}^{\vee}}\sum\limits_{i=1}^nd_{ji}^{U}e^{x_{i}^{\vee}}\nonumber\\&&\times(\eta^++\xi^+-\xi^-)+\frac{1}{1-\tau^{\Delta}}\bigg(2\bar{\mu}\sum\limits_{h=1}^me_{hj}^{U}e^{y_{j}^{\vee}}+1\bigg)
\sum\limits_{i=1}^nd_{ji}^Ue^{x_{i}^{\vee}}\sum\limits_{l=1}^na_{li}^{U}e^{x_{i}^{\vee}}(\tau^++\xi^+-\tau^-)\bigg\}|u_{i}(t)|\nonumber\\
&=&
-\sum\limits_{i=1}^n\bigg\{\sum\limits_{l=1}^na_{li}^{L}e^{x_{i}^{\wedge}}-2\bar{\mu}\bigg(\sum\limits_{l=1}^na_{li}^{U}e^{x_{i}^{\vee}}\bigg)^2
-\frac{1}{1-\tau^{\Delta}}\bigg(2\bar{\mu}\sum\limits_{l=1}^na_{li}^{U}e^{x_{i}^{\vee}}+1\bigg)
\bigg(\sum\limits_{l=1}^na_{li}^{U}e^{x_{i}^{\vee}}\bigg)^{2}\nonumber\\
&&\times(2\tau^{+}-\tau^{-})-\frac{1}{1-\xi^{\Delta}}\bigg(2\bar{\mu}\sum\limits_{l=1}^na_{li}^{U}e^{x_{i}^{\vee}}+1\bigg)
\sum\limits_{j=1}^mc_{ij}^{U}e^{y_{j}^{\vee}}\sum\limits_{i=1}^nd_{ji}^Ue^{x_{i}^{\vee}}(\xi^{+}+\delta^{+}-\xi^{-})\nonumber\\
&&-\sum\limits_{j=1}^md_{ji}^{U}e^{x_{i}^{\vee}}\bigg(2\bar{\mu}\sum\limits_{h=1}^me_{hj}^{U}e^{y_{j}^{\vee}}+1\bigg)
-\frac{1}{1-\xi^{\Delta}}\sum\limits_{j=1}^md_{ji}^{U}e^{x_{i}^{\vee}}\bigg(2\bar{\mu}\sum\limits_{h=1}^me_{hj}^{U}e^{y_{j}^{\vee}}+1\bigg)
\sum\limits_{h=1}^me_{hj}^{U}e^{y_{j}^{\vee}}(\eta^+\nonumber\\&&
+\xi^+-\xi^-)-\frac{1}
{1-\tau^{\Delta}}\sum\limits_{j=1}^md_{ji}^{U}e^{x_{i}^{\vee}}\bigg(2\bar{\mu}\sum\limits_{h=1}^me_{hj}^{U}e^{y_{j}^{\vee}}+1\bigg)
\sum\limits_{l=1}^na_{li}^{U}e^{x_{i}^{\vee}}(\tau^++\xi^+-\tau^-)\bigg\}|u_{i}(t)|\nonumber\\
&&-\sum\limits_{j=1}^m\bigg\{\sum\limits_{h=1}^me_{hj}^{L}e^{y_{j}^{\wedge}}-2\bar{\mu}\bigg(\sum\limits_{h=1}^me_{hj}^{U}e^{y_{j}^{\vee}}\bigg)^2
-\frac{1}{1-\eta^{\Delta}}\bigg(2\bar{\mu}\sum\limits_{h=1}^me_{hj}^{U}e^{y_{j}^{\vee}}+1\bigg)
\bigg(\sum\limits_{h=1}^me_{hj}^{U}e^{y_{j}^{\vee}}\bigg)^{2}\nonumber\\
&&\times(2\eta^{+}-\eta^{-})-\frac{1}{1-\delta^{\Delta}}\bigg(2\bar{\mu}\sum\limits_{h=1}^me_{hj}^{U}e^{y_{j}^{\vee}}+1\bigg)
\sum\limits_{i=1}^nd_{ji}^Ue^{x_{i}^{\vee}}\sum\limits_{j=1}^mc_{ij}^{U}e^{y_{j}^{\vee}}(\delta^++\eta^+-\eta^-)\nonumber\\
&&-\sum\limits_{i=1}^nc_{ij}^{U}e^{y_{j}^{\vee}}\bigg(2\bar{\mu}\sum\limits_{l=1}^na_{li}^{U}e^{x_{i}^{\vee}}+1\bigg)
-\frac{1}{1-\delta^{\Delta}}\sum\limits_{i=1}^nc_{ij}^{U}e^{y_{j}^{\vee}}\bigg(2\bar{\mu}\sum\limits_{l=1}^na_{li}^{U}e^{x_{i}^{\vee}}+1\bigg)
\sum\limits_{l=1}^na_{li}^{U}e^{x_{i}^{\vee}}(\tau^{+}\nonumber\\
&&+\delta^{+}-\delta^{-})-\frac{1}{1-\eta^{\Delta}}\sum\limits_{i=1}^nc_{ij}^{U}e^{y_{j}^{\vee}}\bigg(2\bar{\mu}
\sum\limits_{l=1}^na_{li}^{U}e^{x_{i}^{\vee}}+1\bigg)\sum\limits_{h=1}^me_{hj}^{U}e^{y_{j}^{\vee}}(\eta^{+}+\delta^{+}-\eta^{-})\bigg\}|v_{j}(t)|\nonumber\\
&\leq&
-\sum\limits_{i=1}^n\gamma_i|u_{i}(t)|-\sum\limits_{j=1}^m\tilde{\gamma}_j|v_{j}(t)|\\
&\leq&
-\gamma V(t).
\end{eqnarray*}}
By $(H_{5})$--$(H_{7})$, we see that Condition $(iv)$ of Lemma \ref{lem52} holds. Hence, according to Lemma \ref{lem52}, there exists a unique uniformly asymptotically stable almost periodic solution $(x(t),y(t))$ of system $\eqref{e13}$, and $(x(t),y(t))\in\Omega$. The proof is complete.
\end{proof}

\section{Examples}
 \setcounter{equation}{0}
 \indent

In this section, we present two examples to illustrate the feasibility of our results obtained
in previous sections.

\begin{example} In system \eqref{e13}, let  $\mathbb{T}=\mathbb{R },\mu(t)=0, J=[0, +\infty)_{\mathbb{T}}, i,j,l,h=1,2$ and take coefficients as follows:
\[b_{1}(t)=9-|\cos\sqrt{2}t|,\,\,b_{2}(t)=8+|\sin t|,\]
\[r_{1}(t)=0.09-0.01|\sin\sqrt{2}t|,\,\,r_{2}(t)=0.07+0.02\cos^{2}t,\]
\[
(a_{il}(t))_{2\times 2}= \left(
\begin{array}{cccc}
2-0.1|\sin (\sqrt{2}t)| &0.2- 0.1|\cos (\sqrt{3}t)|\\
0.3-0.2|\sin t| & 2-0.2|\sin t|\\
\end{array}\right),\]
\[
(c_{ih}(t))_{2\times 2}= \left(
\begin{array}{cccc}
0.015+0.005\sin^{2}t& 0.02-0.01|\cos(\sqrt{3} t)|\\
0.01+0.01\cos^{2}t & 0.02-0.01\sin(\sqrt{5}t)\\
\end{array}\right),\]
\[(d_{jl}(t))_{2\times 2}= \left(
\begin{array}{cccc}
0.2-0.05|\cos(\sqrt{2} t)| & 0.3+0.1|\cos t|\\
0.2-0.01|\sin(\sqrt{3} t)|& 0.2-0.01|\cos(\sqrt{3} t)|\\
\end{array}\right),\]
\[(e_{jh}(t))_{2\times 2}= \left(
\begin{array}{cccc}
0.6-0.01\cos t & 0.003+0.002\sin t\\
0.004+0.002\sin t & 0.55-0.01\sin\sqrt{2} t\\
\end{array}\right),\]
\[\tau_{il}(t)=\xi_{jl}(t)=0.003-0.001\sin( 2\pi t),\,\,\delta_{ih}(t)=\eta_{jh}(t)=0.002-0.001\cos( 2\pi t),\]
\[\lambda_{ik}=\lambda_{jk}=e^{(0.04)^{\frac{1}{2^{k}}}}-1,\,\,t_k=k,\,\,i,j,l,h=1,2,\]
 then, by a sample calculation, we have
\[{b_{1}^U}={b_{2}^U}=9,\,\,{b_{1}^L}={b_{2}^L}=8,\,\,{r_{1}^U}={r_{2}^U}=0.09,\]
\[
(a_{il})_{2\times 2}^U= \left(
\begin{array}{cccc}
2 & 0.2\\
0.3 & 2\\
\end{array}\right),\,\,
(a_{il})_{2\times 2}^L= \left(
\begin{array}{cccc}
1.9 & 0.1\\
0.1 & 1.8\\
\end{array}\right),\]
\[
(c_{ih})_{2\times 2}^U= \left(
\begin{array}{cccc}
0.02 & 0.02\\
0.02 & 0.03\\
\end{array}\right),\,\,
(c_{ih})_{2\times 2}^L= \left(
\begin{array}{cccc}
0.015 & 0.01\\
0.01 & 0.01\\
\end{array}\right),\]
\[
(d_{jl})_{2\times 2}^U= \left(
\begin{array}{cccc}
0.2 & 0.4\\
0.2 & 0.2\\
\end{array}\right),\,\,
(d_{jl})_{2\times 2}^L= \left(
\begin{array}{cccc}
0.15 & 0.3\\
0.19 & 0.19\\
\end{array}\right),\]
\[
(e_{jh})_{2\times 2}^U= \left(
\begin{array}{cccc}
0.61 & 0.005\\
0.006 & 0.56\\
\end{array}\right),\,\,
(e_{jh})_{2\times 2}^L= \left(
\begin{array}{cccc}
0.59 & 0.001\\
0.002 & 0.54\\
\end{array}\right),\]
\[\tau^{+}=\xi^+=0.004,\,\,\tau^{-}=\xi^-=0.002,\,\,\tau^{\Delta}=\xi^{\Delta}=0.002,\]
\[\delta^{+}=\eta^+=0.003,\,\,\delta^{-}=\eta^-=0.001,\,\,\delta^{\Delta}=\eta^{\Delta}=0.002,\,\,r=0.9.\]
By calculating, we have
\[x_{1}^{\vee}\approx1.591,\,\,x_{2}^{\vee}\approx1.645,\,\,y_{1}^{\vee}\approx1.653,\,\,y_{2}^{\vee}\approx1.324,\]
\[x_{1}^{\wedge}\approx 0.979,\,\,x_{2}^{\wedge}\approx 0.896,\,\,y_{1}^{\wedge}\approx0.296,\,\,y_{2}^{\wedge}\approx0.198,\]
then
 \begin{eqnarray*}
 \gamma_{1}\approx2.408, \,\,\gamma_{2}\approx0.466,\,\,\tilde{\gamma}_{1}\approx0.502,\,\,\tilde{\gamma}_{2}\approx0.418,\,\,
\gamma=\min\{\gamma_{1},\gamma_{2},\tilde{\gamma}_{1},\tilde{\gamma}_{2}\}=0.418.
\end{eqnarray*}

Thus, conditions $(H_{1})$--$(H_{7})$ are satisfied. According to Theorems \ref{thm31} and \ref{thm41}, system \eqref{e13} is permanent and   has a unique almost periodic solution $(x(t),y(t))$ that is uniformly asymptotically stable.
\end{example}

\begin{example}In system \eqref{e13}, let  $\mathbb{T}=\mathbb{Z}, \mu(t)=1, J=[0, +\infty)_{\mathbb{T}}, i,j,l,h=1,2$ and take coefficients as follows:
\[b_{1}(t)=0.1,\,\,b_{2}(t)=0.09,\,\,r_{1}(t)=0.001|\cos(\sqrt{3}t)|,\,\,r_{2}(t)=0.001|\sin(\sqrt{2}t)|,\]
\[
(a_{il}(t))_{2\times 2}= \left(
\begin{array}{cccc}
0.096 & 0.002-0.001\sin\frac{\pi}{4}t\\
0.002+0.001\sin\frac{\pi}{3}t & 0.087\\
\end{array}\right),\]
\[
(c_{ih}(t))_{2\times 2}= \left(
\begin{array}{cccc}
0.001|\cos(\sqrt{5}t)| & 0.001|\sin\frac{\pi}{4}t|\\
0.001|\sin\frac{\pi}{3}t| & 0.001|\sin(\sqrt{5}t)|\\
\end{array}\right),\]
\[(d_{jl}(t))_{2\times 2}= \left(
\begin{array}{cccc}
0.05 & 0.03\\
0.04 & 0.03\\
\end{array}\right),\,\,(e_{jh}(t))_{2\times 2}= \left(
\begin{array}{cccc}
0.082 & 0.001\\
0.001 & 0.07\\
\end{array}\right),\]
\[\tau_{il}(t)=\xi_{jl}(t)=\frac{1+(-1)^{t}}{1000},\,\,\delta_{ih}(t)=\eta_{jh}(t)=\frac{2+(-1)^{t}}{1000},\]
\[\lambda_{ik}=\lambda_{jk}=e^{(0.0004)^{\frac{1}{2^{k}}}}-1,\,\,t_k=k,\,\,i,j,l,h=1,2,\]
 then
\[{b_{1}^U}={b_{1}^L}=0.1,\,\, {b_{2}^U}={b_{2}^L}=0.09,\,\,{r_{1}^U}=0.001,\,\,{r_{2}^U}=0.001,\]
\[
(a_{il})_{2\times 2}^U= \left(
\begin{array}{cccc}
0.096 & 0.003\\
0.003 & 0.087\\
\end{array}\right),\,\,
(a_{il})_{2\times 2}^L= \left(
\begin{array}{cccc}
0.096 & 0.001\\
0.001 & 0.087\\
\end{array}\right),\]
\[
(c_{ih})_{2\times 2}^U= \left(
\begin{array}{cccc}
0.001 & 0.001\\
0.001 & 0.001\\
\end{array}\right),\,\,
(d_{jl})_{2\times 2}^U=(d_{jl})_{2\times 2}^L= \left(
\begin{array}{cccc}
0.05 & 0.03\\
0.04 & 0.03\\
\end{array}\right),\]
\[
(e_{jh})_{2\times 2}^U=(e_{jh})_{2\times 2}^L= \left(
\begin{array}{cccc}
0.082 & 0.001\\
0.001 & 0.07\\
\end{array}\right),\]
\[\tau^{+}=\xi^+=0.002,\,\,\tau^{-}=\xi^-=0,\,\,\tau^{\Delta}=\xi^{\Delta}=0.001,\]
\[\delta^{+}=\eta^+=0.003,\,\,\delta^{-}=\eta^-=0.001,\,\,\delta^{\Delta}=\eta^{\Delta}=0.001,\,\,r=2.\]

By calculating, we have
\[x_{1}^{\vee}\approx 0.041,\,\,x_{2}^{\vee}\approx 0.034,\,\,y_{1}^{\vee}\approx0.012,\,\,y_{2}^{\vee}\approx0.038,\]
\[x_{1}^{\wedge}\approx 1.374,\,\,x_{2}^{\wedge}\approx 1.360,\,\,y_{1}^{\wedge}\approx2.724,\,\,y_{2}^{\wedge}\approx2.747,\]
then
\begin{eqnarray*}
\gamma_{1}\approx0.239, \,\,\gamma_{2}\approx0.244,\,\,\tilde{\gamma}_{1}\approx1.248,\,\,\tilde{\gamma}_{2}\approx1.093,\,\,
\gamma=\min\{\gamma_{1},\gamma_{2},\tilde{\gamma}_{1},\tilde{\gamma}_{2}\}=0.239.
\end{eqnarray*}
Thus, conditions $(H_{1})$--$(H_{7})$ are satisfied. According to Theorems \ref{thm31} and  \ref{thm41}, system \eqref{e13} is permanent and   has a unique almost periodic solution $(x(t),y(t))$ that is uniformly asymptotically stable.
\end{example}

\section{Conclusion}
\indent

In this paper, we first proposed two types of definitions of  piecewise continuous almost periodic functions  on time scales and indicated that this two types of almost periodic functions are the same. Then, we   establish some comparison theorems of dynamic equations with impulses and delays on time scales, and, as an application, we obtain some sufficient conditions for the permanence of \eqref{e13} by using the comparison theorems. Finally,
by Lyapunov method, we investigate
the existence and uniqueness of almost
periodic solutions of impulsive dynamic equations on time scales and as an application, we obtain some sufficient conditions for the existence and uniformly asymptotic stability of unique positive almost periodic solution of \eqref{e13}.
  Our results and methods of this paper can be used to study other types of population models with impulses and delays.

\end{document}